\documentclass[1 leqno,11pt,psfig]{amsart}
\usepackage{amssymb, amsmath,amsmath,latexsym,amssymb,amsfonts,amsbsy, amsthm,mathtools,graphicx,CJKutf8,CJKnumb,CJKulem,color}
\usepackage{graphicx,epsfig}
\usepackage{graphicx}
\usepackage{amssymb}
\usepackage{graphicx}
\usepackage{graphicx,epsfig}
\usepackage{amsmath}
\usepackage{mathrsfs}
\usepackage{amssymb}

\usepackage{amssymb,mathrsfs,amsfonts,amsmath,amsbsy,tikz}\usepackage{fontenc}\usepackage{textcomp}

\numberwithin{equation}{section}

\allowdisplaybreaks
\let\al=\alpha

\let\f=\frac

\let\om=\omega

\let\pa=\partial

\def\bbT{\mathbb{T}}

\newcommand{\beq}{\begin{equation}}
\newcommand{\eeq}{\end{equation}}
\newcommand{\ben}{\begin{eqnarray}}
\newcommand{\een}{\end{eqnarray}}
\newcommand{\beno}{\begin{eqnarray*}}
\newcommand{\eeno}{\end{eqnarray*}}

\newtheorem{theorem}{Theorem}[section]
\newtheorem{definition}[theorem]{Definition}
\newtheorem{lemma}[theorem]{Lemma}
\newtheorem{proposition}[theorem]{Proposition}

\newtheorem{remark}[theorem]{Remark}
\begin{document}
\begin{CJK*}{UTF8}{gkai}
\title[Linear inviscid damping]{Linear inviscid damping for the $\beta$-plane equation}

\author{Dongyi Wei}
\address{School of Mathematical Science, Peking University, 100871, Beijing, P. R. China}
\email{jnwdyi@163.com}

\author{Zhifei Zhang}
\address{School of Mathematical Science, Peking University, 100871, Beijing, P. R. China}
\email{zfzhang@math.pku.edu.cn}

\author{Hao Zhu}
\address{Chern Institute of Mathematics, Nankai University, 300071, Tianjin, P. R. China}
\email{haozhu@nankai.edu.cn}

\date{\today}

\maketitle

\begin{abstract}
In this paper, we study the linear inviscid damping for the linearized $\beta$-plane equation around shear flows.
We develop a new method to give the explicit decay rate of the velocity for a class of monotone shear flows. This method is based on the space-time
estimate and the vector field method in sprit of the wave equation. For general shear flows including the Sinus flow, we also prove the linear damping by establishing the limiting absorption principle, which is based on the compactness method introduced by Wei-Zhang-Zhao in \cite{WZZ2}. The main difficulty is that the Rayleigh-Kuo equation has more singular points due to the Coriolis effects so that the compactness argument becomes more involved and  delicate.

\end{abstract}
\section{Introduction}

In this paper, we are concerned with the large-scale motion of ocean and atmosphere. By a large-scale motion, we mean the ratio $L/D\gg1$, where $L$ and $D$ are horizonal and vertical scale length, respectively. For such large scale flows, the rotation of the earth may affect the dynamics of the fluid significantly, therefore the Coriolis force must be taken into account. While, the vertical acceleration can be neglected in the equation of motion.
Under the $\beta$-plane approximation of the Coriolis force, the motion for large scale flow could be described by 2-D incompressible Euler equation with rotation
\begin{align}\label{Euler equation}
		\partial_{t}\vec{v}+(\vec{v}\cdot\nabla)\vec{v}=-\nabla P-\beta yJ\vec
{v},\quad  \nabla\cdot\vec{v}=0,
	\end{align}
where $\vec{v}=(v_{1},v_{2})$ is the fluid velocity, $P$ is the pressure,
\[
J=%
\begin{pmatrix}
0 & -1\\
1 & 0
\end{pmatrix}
\]
is the rotation matrix, and $\beta$ is the Rossby number. Here we study the fluid in a finite channel, i.e.,
\beno
\Omega=\big\{(x,y): x\in \bbT, y\in {[y_1,y_2]}\big\}
\eeno
with non-slip boundary condition on $\pa\Omega$:
\beno
v_2=0\quad \text{on}\quad  y=y_1, y_2.
\eeno
Let us refer to \cite{Pedlosky1987} for more introduction on geophysical fluids.\smallskip

The vorticity $\omega$ is defined as $\omega=\partial_{x}v_{2}-\partial_{y}v_{1}$, and the stream function $\psi$ is
introduced such that $\vec{v}=\nabla^{\perp}\psi=(\pa_y\psi,-\pa_x\psi)$. The
vorticity form of (\ref{Euler equation}) takes
\begin{equation}
\partial_{t}\omega+(\vec{v}\cdot\nabla)\omega+\beta v_{2}%
=0.\label{vorticity-eqn}%
\end{equation}

Consider the shear flow $(u(y),0)$, which is a steady solution of (\ref{vorticity-eqn}). The
linearized equation of (\ref{vorticity-eqn}) around $(u(y),0)$ takes
\begin{align}\label{linearized Euler equation}
	\partial_{t}\omega+u\partial_{x}\omega+(\beta-u^{\prime\prime})v_{2}=0.
	\end{align}
In terms of the stream function, (\ref{linearized Euler equation}) can be written as
\begin{align*}
\partial_{t}\Delta\psi+u\partial_{x}\Delta\psi+(\beta-u^{\prime\prime})\partial_x\psi=0.
\end{align*}
 By taking Fourier transform in $x$, we get
 \begin{align}\label{equation after Fourier transform}
 (\partial^2_y-\alpha^2)\partial_t\widehat{\psi}=i\alpha((u''-\beta)-u(\partial^2_y-\alpha^2))\widehat{\psi}.
 \end{align}
For any fixed $\alpha>0$ and $\beta\in \mathbb{R}$, we define
\begin{align*}
\mathcal{R}_{\alpha,\beta}\widehat{\psi}:=-(\partial^2_y-\alpha^2)^{-1}((u''-\beta)-u(\partial^2_y-\alpha^2))\widehat{\psi}.
\end{align*}
Then (\ref{equation after Fourier transform}) is equivalent to
\beno
-\frac{1}{ i\alpha}\partial_t\widehat{\psi}=\mathcal{R}_{\alpha,\beta}\widehat{\psi}.
\eeno

Dynamical behavior of fluid around a shear flow  under the Coriolis force is believed to be more fruitful. Barotropic instability of shear flows is a classical problem in geophysical fluid  dynamics. Kuo \cite{Kuo1949} gave a necessary condition for the instability that $\beta-u''$ must change sign in $[y_1,y_2]$, which is a generalization of Rayleigh's inflection-point theorem. Pedlosky  proved that an unstable wave speed must lie in the semicircle with {center} ${u_{\min}+u_{\max}\over2}$ and radius ${u_{\max}-u_{\min}\over2}+{|\beta|\over 2\alpha^2}$ in \cite{Pedlosky1963, Pedlosky1964}, which is a generalization of Howard's semicircle theorem.
In the literature, there are several numerical analysis on barotropic instability, see \cite{Kuo1974, Pedlosky1987} for the flow with Sinus profile; see \cite{Balmforth-Piccolo2001, Burns-Maslowe-Brown2002, Engevik2004, Maslowe1991} for the Bickley jet; and see \cite{Dickinson-Clare1973} for the hyperbolic-tangent flow.
In a recent paper \cite{LYZ}, Lin, Yang and the third author gave a systematic study for the barotropic instability, where they proved several results sketched below.

\begin{itemize}

\item[1.]  Give a classification of neutral modes in $H^2$ (i.e. {regular}, singular and non-resonant neutral modes) for general shear flows;

\item[2.] Introduce a method based on Hamiltonian structure to study the stability for a class of shear flows, and especially obtain precise lower transition from unstable waves to stable ones for the Sinus flow;

\item[3.] Construct traveling waves, which is purely due to Coriolis effect, near the Sinus flow with traveling speeds beyond the range of the basic flow;

\item[4.] Prove the linear inviscid damping in time averaged sense for the Sinus flow with $\beta\in(-{\pi^2\over2},{\pi^2\over2})$.

\end{itemize}

In this paper, we study the linear inviscid damping for the linearized $\beta$-plane equation. This could be regarded as the first step toward understanding the asymptotic stability of shear flows in a large scale motion. Since the work on Landau damping by Mouhot and Villani \cite{MV}, the study of the inviscid damping has become a very
active field as an analogue of Landau damping in hydrodynamics. In fact,  Orr in 1907(\cite{Orr}) found the damping phenomena for the Couette flow $(y,0)$
earlier than Landau damping in 1946(\cite{Lan}). Recently, Bedrossian and Masmoudi \cite{BM1} proved nonlinear inviscid damping for the 2-D Euler equations around the Couette flow for the perturbation in Gevrey class.
On the other hand, Lin and Zeng \cite{LZ} proved that nonlinear inviscid damping is not true for the perturbation of vorticity in $H^s$ for $s<\f32$.
The linear damping for the Couette flow could be easily generalized to the $\beta$-plane equation. It also seems possible to generalize nonlinear damping result in \cite{BM1} to the $\beta$-plane equation. Let us also mention recent results on long time behaviour of the $\beta$-plane equation near the trivial solution \cite{EW, PW}.\smallskip

For general shear flows, the linear damping is a highly nontrivial problem due to the presence of nonlocal part $u''(y)\pa_x(-\Delta)^{-1}$ and the Coriolis effect. In this case, the linear dynamics is associated with the singularities of the solution for the Rayleigh-Kuo equation at the critical layers(i.e., $u=c$):
\begin{align*}
(u-c)(\phi''-\alpha^2\phi)-(u''-\beta)\phi=f.
\end{align*}

When $\beta=0$, Case \cite{Case} gave a first prediction of linear damping for monotone shear flows. His prediction was confirmed by a series of works
\cite{Ros, Ste, Zill1, Zill2}, and finally by \cite{WZZ1}. In \cite{WZZ2, WZZ3}, the first two authors and Zhao proved the linear damping for
non-monotone flows including Poiseuille flow $u(y)=y^2$ and Kolmogorov flow $u(y)=\cos y$. In such case, there are two mechanisms leading to the damping: the vorticity mixing and  the vorticity depletion phenomena at the stationary streamlines, which was first observed by Bouchet and Morita \cite{BM} for the latter.
Let us emphasize that nonlocal part $u''(y)\pa_x(-\Delta)^{-1}$ plays an important role for non-monotone flows.\smallskip

The case of $\beta\neq 0$ is the goal of this paper. We first consider the linear damping for a class of monotone shear flows, and prove the same decay estimates of the velocity as the case of $\beta=0$. More importantly, we develop a new method, which is much simpler than that in \cite{WZZ1}. This method is based on the space-time estimate and the vector field method in sprit of the wave equation. First of all, we establish the space-time estimate of the velocity by using the limiting absorption principle. Next we derive the decay estimates of the velocity from the space-time estimate with the help of  the vector field method. We believe that new method could be used to the other related problems such as
the setting considered in \cite{BCV, GNRS}, and might shed some light on nonlinear inviscid damping for stable monotone shear flows.\smallskip

In the following theorem, we assume that $x\in \mathbb{T}_{L}$(i.e., the period is  $2\pi/L$), $(y_1,y_2)=(0,1),\ u(y)\in C^4([0,1])$ and  $u'(y)\ge c_0$ for some $c_0>0$.

\begin{theorem}\label{thm:monotone}
Assume that the linearized operator $\mathcal{R}_{\alpha,\beta}$ has no embedding eigenvalues or eigenvalues for $\alpha\neq 0$, and the initial vorticity satisfies $\int_{\bbT_L}\om_0(x,y)dx=0$. Then it holds that
\begin{itemize}

 \item[1.] if $\om_0(x,y)\in H^{-1}_xH^1_y$, then
\beno
\|\vec{v}(t)\|_{L^2}\leq \frac{C}{\langle t\rangle}\|\omega_0\|_{H^{-1}_xH_y^1};
\eeno
\item[2.] if $\om_0(x,y)\in H^{-1}_xH^2_y$, then
\beno
\|v_2(t)\|_{L^2}\leq \frac{C}{\langle t\rangle^2}\|\omega_0\|_{H^{-1}_xH_y^2};
\eeno
\item[3.] if $\om_0(x,y)\in L_{x,y}^2$, there exists $\om_\infty(x,y)\in L_{x,y}^2$ such that
\beno
\|\om(t,x+tu(y),y)-\om_\infty\|_{L^2}\longrightarrow 0\quad \textrm{as}\quad t\rightarrow +\infty.
\eeno
\end{itemize}
\end{theorem}

The second part of this paper is to consider the linear damping for general shear flows, which satisfy
\beno
(\textbf{H1}) \quad u\in H^{4}(y_{1},y_{2}), \,\,  u''(y_c)\neq0,\,\, \beta/u''(y_c)<9/8\,\, \text{at critical points } u'(y_c)=0.
\eeno

\begin{theorem}\label{thm:non-monotone}
Assume that $u$ satisfies $(\textbf{H1})$, $\mathcal{R}_{\alpha,\beta}$ has no embedding eigenvalues for $\al\neq 0$,
and the initial vorticity satisfies $\widehat{\omega}_0(\alpha,y)=0$ for $y \in\{y_1,y_2\}\cap(u')^{-1}\{0\}$ and ${\widehat{\omega}_0(\alpha,\cdot)\over p(\cdot)}\in H_y^1(y_1,y_2)$, where
\begin{align}\label{def-p}
p(z)=\prod\limits_{y\in A}(z-y),\quad A=\big\{y\in[y_1,y_2]|u'(y)=0,u''(y)=\beta\big\}.
\end{align}
Moreover, $P_{\sigma_d(\mathcal{R}_{\alpha,\beta})}\widehat{\psi}(0,\alpha,\cdot)=0$, where $P_{\sigma_d(\mathcal{R}_{\alpha,\beta})}$ is the spectral projection to $\sigma_d(\mathcal{R}_{\alpha,\beta})$. Then it holds that
\begin{align*}
\|\widehat{v}(\cdot, \alpha,\cdot)\|_{H^1_tL^2_y}\leq C \big\|{\widehat{\omega}_0(\alpha,\cdot)\over p}\big\|_{H_y^1}.
\end{align*}
Here $C$ is a constant depending on $\al,\beta$.
In particular, we have
$$\lim\limits_{t\to\infty}\|\widehat{v}(t,\alpha,\cdot)\|_{L^2_y}=0.$$
\end{theorem}

The proof of Theorem \ref{thm:non-monotone} follows the  method introduced in \cite{WZZ2}, where the key ingredient is to
establish the limiting absorption principle by using the compactness argument. Compared with the case with no Coriolis effects, new difficulty is that the Rayleigh-Kuo equation has  more singular points due to the influence of $\beta$ so that the compactness argument becomes more involved and delicate.\smallskip

In section 5, we  will apply Theorem \ref{thm:non-monotone} to the flow with Sinus profile. For this flow, the region of $(\alpha,\beta)$ parameters so that $\mathcal{R}_{\alpha,\beta}$ has no embedding eigenvalues can be precisely determined.

\section{Linear inviscid damping for monotone shear flows}

In this section, we prove the explicit decay estimate of the velocity for a class of monotone shear flows, which satisfy
\beno
 u(y)\in C^4([0,1]),\quad u'(y)\ge c_0.
\eeno
And the period $2\pi/L\ge c_0$ in $x$ variable. We use the $L^2$ inner product $ \langle f,g\rangle=\int_0^1f(y)\overline{g(y)}dy,$ and use $\psi=-(\partial^2_y-\alpha^2)^{-1}\om $ to denote the unique solution of $-(\partial^2_y-\alpha^2)\psi=\om,\ \psi(0)=\psi(1)=0. $

\subsection{Space-time estimate}

For any fixed $\alpha\in(2\pi\mathbb{Z}/L)\setminus\{ 0\}$ and $\beta\in \mathbb{R}$, we define
\begin{align*}
\mathcal{R}'_{\alpha,\beta}\widehat{\om}=-((u''-\beta)(\partial^2_y-\alpha^2)^{-1}-u)\widehat{\om}.
\end{align*}
Then we have $\mathcal{R}'_{\alpha,\beta}(\partial^2_y-\alpha^2)=(\partial^2_y-\alpha^2)\mathcal{R}_{\alpha,\beta}$ in $ H_0^1(0,1)$ and (\ref{linearized Euler equation}) is equivalent to
$$\partial_t\widehat{\om}=-i\alpha\mathcal{R}'_{\alpha,\beta}\widehat{\om}$$ after taking Fourier transform in $x$.
Without loss of generality, we may assume $\al>0$ in the sequel, so $\alpha\in \Lambda:=\big\{2\pi k/L|k\in\mathbb{Z}_+\big\}$ and $\al\ge 2\pi/L\ge c_0$.

\begin{proposition}\label{lem5.1}
Assume that $\mathcal{R}_{\alpha,\beta} $ has no embedding eigenvalues or eigenvalues.
Let $\psi=-(\partial^2_y-\alpha^2)^{-1}\om$ and $\om(t,y)$ solve
\beno
 \partial_t{\om}+i\alpha\mathcal{R}'_{\alpha,\beta}{\om}+f=0
\eeno
for $t\in[0,T]$ and $y\in[0,1]$. Then we have
\beno
&&\|\om(T)\|_{L^2}^2+\alpha^2\int_0^T\big(\|\partial_y\psi(t)\|_{L^2}^2+\alpha^2\|\psi(t)\|_{L^2}^2\big)dt\\
&&\quad\leq C \|\om(0)\|_{L^2}^2+C\alpha^{-2}\int_0^T\big(\|\partial_yf(t)\|_{L^2}^2+\alpha^2\|f(t)\|_{L^2}^2\big)dt.
\eeno
Moreover, if $f(t,0)=f(t,1)=0$, then
\beno
\alpha\int_0^T\big(|\partial_y\psi(t,0)|^2+|\partial_y\psi(t,1)|^2\big)dt\leq C \|\om(0)\|_{L^2}^2+C\int_0^T\big(\alpha^{-2}\|\partial_yf(t)\|_{L^2}^2+\|f(t)\|_{L^2}^2\big)dt.
\eeno
Here the constant $C$ only depends on $\beta$ and $u$.
\end{proposition}

We need the following lemmas.

\begin{lemma}\label{lem5.2}
Let $\psi=-(\partial^2_y-\alpha^2)^{-1}\om$ and $\om(t,y)$ solve $\partial_t{\om}+i\alpha u{\om}=0$ for $t\in\mathbb{R}$ and $y\in[0,1]$.
Then we have
\beno
&&\alpha^2\int_{\mathbb{R}}\big(\|\partial_y\psi(t)\|_{L^2}^2+\alpha^2\|\psi(t)\|_{L^2}^2\big)dt \leq C \|\om(0)\|_{L^2}^2,\\
&&\alpha\int_{\mathbb{R}}\big(|\partial_y\psi(t,0)|^2+|\partial_y\psi(t,1)|^2\big)dt \leq C \|\om(0)\|_{L^2}^2,
\eeno
where the constant $C$ only depends  on $c_0$.
\end{lemma}

\begin{proof}
We use the basis in $L^2(0,1): \varphi_k(y)=\sin(\pi ky),\ k\in\mathbb{Z}_+.$ Then we have
\begin{align*}
&\om=\sum_{k=1}^{+\infty}2\langle \omega,\varphi_k\rangle\varphi_k,\ \|\om\|_{L^2}^2=\sum_{k=1}^{+\infty}2|\langle \omega,\varphi_k\rangle|^2,\ \psi=\sum_{k=1}^{+\infty}\frac{2\langle \omega,\varphi_k\rangle}{(\pi k)^2+\alpha^2}\varphi_k,\\ &\|\partial_y\psi(t)\|_{L^2}^2+\alpha^2\|\psi(t)\|_{L^2}^2=\langle \psi(t),\om(t)\rangle=\sum_{k=1}^{+\infty}\frac{2|\langle \omega(t),\varphi_k\rangle|^2}{(\pi k)^2+\alpha^2}.
\end{align*}
Since $\partial_t{\om}+i\alpha u{\om}=0$, the solution is given by $\om(t,y)=e^{-i\alpha tu(y)}\om(0,y)$. So,
\begin{align*}
&\langle \omega(t),\varphi_k\rangle=\int_0^1e^{-i\alpha tu(y)}\om(0,y)\varphi_k(y)dy=\int_{u(0)}^{u(1)}e^{-i\alpha tz}\frac{\om(0,u^{-1}(z))\varphi_k(u^{-1}(z))}{u'(u^{-1}(z))}dz,
\end{align*}
from which and Plancherel's formula, we infer that
\begin{align*}
\int_{\mathbb{R}}|\langle \omega(t),\varphi_k\rangle|^2dt&=\frac{2\pi}{\alpha}\int_{u(0)}^{u(1)}\left| \frac{\om(0,u^{-1}(z))\varphi_k(u^{-1}(z))}{u'(u^{-1}(z))}\right|^2dz=\frac{2\pi}{\alpha}\int_{0}^{1} \frac{|\om(0,y)\varphi_k(y)|^2}{u'(y)}dy\\ &\leq \frac{2\pi}{\alpha}\int_{0}^{1} \frac{|\om(0,y)|^2}{u'(y)}dy\leq \frac{2\pi}{\alpha c_0}\|\om(0)\|_{L^2}^2.
\end{align*}
Therefore,
\begin{align*}
\int_{\mathbb{R}}(\|\partial_y\psi(t)\|_{L^2}^2+\alpha^2\|\psi(t)\|_{L^2}^2)dt&=\sum_{k=1}^{+\infty}\int_{\mathbb{R}}\frac{2|\langle \omega(t),\varphi_k\rangle|^2}{(\pi k)^2+\alpha^2}dt\leq\sum_{k=1}^{+\infty}\frac{2\pi}{\alpha c_0}\frac{2\|\om(0)\|_{L^2}^2}{(\pi k)^2+\alpha^2}\\
&\leq \int_{0}^{+\infty}\frac{2\pi}{\alpha c_0}\frac{2\|\om(0)\|_{L^2}^2}{(\pi z)^2+\alpha^2}dz=\frac{2\pi}{\alpha^2 c_0}\|\om(0)\|_{L^2}^2,
\end{align*}
which gives the first inequality.

Let
\begin{align}\label{gamma12-def}
\gamma_1(y)=\dfrac{\sinh(\alpha y)}{\sinh\alpha }, \;\gamma_0(y)=\dfrac{\sinh(\alpha(1- y))}{\sinh\alpha }.
\end{align}
 Then we have\begin{align*}
&\langle \om,\gamma_1\rangle=-\langle (\partial^2_y-\alpha^2)\psi,\gamma_1\rangle=-\langle \psi,(\partial^2_y-\alpha^2)\gamma_1\rangle-(\psi'\gamma_1-\psi\gamma_1')|_0^1=-\partial_y\psi(t,1),
\end{align*} and $\langle \om,\gamma_0\rangle=\partial_y\psi(t,0). $
As in the proof of the first inequality, we have
\begin{align*}
\int_{\mathbb{R}}|\partial_y\psi(t,{j})|^2dt=\int_{\mathbb{R}}|\langle \omega(t),\gamma_j\rangle|^2dt&=\frac{2\pi}{\alpha}\int_{0}^{1} \frac{|\om(0,y)\gamma_j(y)|^2}{u'(y)}dy\leq \frac{2\pi}{\alpha c_0}\|\om(0)\|_{L^2}^2,\ j=0,1,
\end{align*}
which gives the second inequality.
\end{proof}

\begin{lemma}\label{lem5.7}
Let $\psi=-(\partial^2_y-\alpha^2)^{-1}\om$ and $\om(t,y)$ solve $\partial_t{\om}+i\alpha u{\om}+f=0$ for $t\in[0,T]$ and $y\in[0,1]$ and $ \om(0)=0$. Then we have
\beno
\|\om(T)\|_{L^2}^2\leq C\int_0^T\big(\alpha^{-2}\|\partial_yf(t)\|_{L^2}^2+\|f(t)\|_{L^2}^2\big)dt,
\eeno
where the constant $C$ only depends on $c_0$.
\end{lemma}

\begin{proof}
Let $\om_1(t,y)=e^{i\alpha(T- t)u(y)}\om(T,y)$ and $\psi_1=-(\partial^2_y-\alpha^2)^{-1}\om_1$.
Then we have
\beno
\partial_t{\om_1}+i\alpha u{\om_1}=0,\quad \om_1(T)=\om(T),\ \|\om_1(0)\|_{L^2}=\|\om(T)\|_{L^2}.
\eeno
Then it follows from Lemma \ref{lem5.2} that 
\beno
&&\alpha^2\int_{0}^T\big(\|\partial_y\psi_1(t)\|_{L^2}^2+\alpha^2\|\psi_1(t)\|_{L^2}^2\big)dt
\leq C \|\om_1(0)\|_{L^2}^2=C \|\om(T)\|_{L^2}^2,\\
&&\alpha\int_0^T\big(|\partial_y\psi_1(t,0)|^2+|\partial_y\psi_1(t,1)|^2\big)dt \leq C \|\om_1(0)\|_{L^2}^2=C \|\om(T)\|_{L^2}^2.
\eeno

Noticing that
\begin{align*}
\partial_t\langle\om_1,\om\rangle=&\langle\partial_t\om_1,\om\rangle+\langle\om_1,\partial_t\om\rangle=-\langle i\alpha u\om_1,\om\rangle-\langle \om_1,f+i\alpha u\om\rangle=-\langle \om_1,f\rangle\\
=&\langle (\partial^2_y-\alpha^2)\psi_1,f\rangle=-\langle \partial_y\psi_1,\partial_yf\rangle-\alpha^2\langle \psi_1,f\rangle+({\partial_y\psi_1}\overline{f})|_{y=0}^1,
\end{align*}
we infer that
\begin{align*}
&|\partial_t\langle\om_1,\om\rangle|\leq \|\partial_yf\|_{L^2}\|\partial_y\psi_1\|_{L^2}+\alpha^2\| f\|_{L^2}\|\psi_1\|_{L^2}+\| f\|_{L^{\infty}}\big(|\partial_y\psi_1(t,0)|+|\partial_y\psi_1(t,1)|\big).
\end{align*}
By Gagliardo-Nirenberg inequality, we get
\begin{align*}
\|f\|_{L^{\infty}}^2\leq C\|f\|_{L^{2}}\|f\|_{H^{1}}\leq C\alpha^{-1}(\|f\|_{H^1}^2+\alpha^2\|f\|_{L^2}^2)\leq C\alpha^{-1}(\|\partial_yf\|_{L^2}^2+\alpha^2\|f\|_{L^2}^2).
\end{align*}

In summary, we obtain
\begin{align*}
\|\om(T)\|_{L^2}^2=&\langle\om_1,\om\rangle|_0^T\leq\int_0^T|\partial_t\langle\om_1,\om\rangle|dt\\
\leq&\int_0^T\big(\| \partial_yf\|_{L^2}\|\partial_y\psi_1\|_{L^2}+\alpha^2\| f\|_{L^2}\|\psi_1\|_{L^2}+\| f\|_{L^{\infty}}(|\partial_y\psi_1(t,0)|+|\partial_y\psi_1(t,1)|)\big)dt\\
\leq& \left(\int_0^T(\|\partial_yf(t)\|_{L^2}^2+\alpha^2\|f(t)\|_{L^2}^2+2\alpha\|f(t)\|_{L^{\infty}}^2)dt\right)^{\frac{1}{2}}\\ &\times\left(\int_0^T(\|\partial_y\psi_1(t)\|_{L^2}^2+\alpha^2\|\psi_1(t)\|_{L^2}^2+\alpha^{-1}
|\partial_y\psi_1(t,0)|^2+\alpha^{-1}|\partial_y\psi_1(t,1)|^2)dt\right)^{\frac{1}{2}}\\ \leq&\left(C\int_0^T(\|\partial_yf(t)\|_{L^2}^2+\alpha^2\|f(t)\|_{L^2}^2)dt\right)^{\frac{1}{2}} \left(C\alpha^{-2}\|\om(T)\|_{L^2}^2\right)^{\frac{1}{2}},
\end{align*}
which gives our result.
\end{proof}

The following limiting absorption lemma will be proved in next section.

\begin{lemma}\label{lem5.6}
Let  $\beta\in\mathbb{R}$.
Assume that $\mathcal{R}_{\alpha,\beta}$ has no embedding eigenvalues or eigenvalues for any $\alpha\in \Lambda$. Then there exists $\varepsilon_0>0$ such that
for any $c\in \mathbb{C},\ 0<\text{Im}(c)<\varepsilon_0$ and $\alpha\in \Lambda,$ the unique solution $\Phi$ to the boundary value problem
\begin{align*}
(u-c)(\Phi''-\alpha^2\Phi)-(u''-\beta)\Phi=\omega, \;\Phi(0)=\Phi(1)=0
\end{align*}
has the following uniform bound
\begin{align*}
\|\partial_y\Phi\|_{L^2}+\alpha\|\Phi\|_{L^2}\leq C\alpha^{-1}\big(\|\partial_y\om\|_{L^2}+\alpha\|\om\|_{L^2}\big).
\end{align*}
Moreover, if $\om(0)=\om(1)=0 $, we have
\begin{align*}
|\partial_y\Phi(0)|+|\partial_y\Phi(1)|\leq C\alpha^{-\frac{1}{2}}\big(\|\partial_y\om\|_{L^2}+\alpha\|\om\|_{L^2}\big).
\end{align*}
\end{lemma}

Now we are in a position to prove Proposition \ref{lem5.1}.

\begin{proof}
{\bf Step 1.} We introduce
\beno
\om_1(t,y)=e^{-i\alpha tu(y)}\om(0,y),\quad \om_2=\om-\om_1,\quad \psi_j=-(\partial^2_y-\alpha^2)^{-1}\om_j,\ j=1,2.
\eeno
Then we have $\om=\om_1+\om_2,\ \psi=\psi_1+\psi_2$ and
\beno
\partial_t{\om_1}+i\alpha u{\om_1}=0.
\eeno
By Lemma \ref{lem5.2}, we have
\begin{align}\label{psi1a}&\alpha^2\int_{0}^T(\|\partial_y\psi_1(t)\|_{L^2}^2+\alpha^2\|\psi_1(t)\|_{L^2}^2)dt\leq C \|\om_1(0)\|_{L^2}^2=C \|\om(0)\|_{L^2}^2,\\ \label{psi1b}&\alpha\int_0^T(|\partial_y\psi_1(t,0)|^2+|\partial_y\psi_1(t,1)|^2)dt \leq C \|\om_1(0)\|_{L^2}^2=C \|\om(0)\|_{L^2}^2.
\end{align}
Also we have $\|\om_1(T)\|_{L^2}^2=\|\om(0)\|_{L^2}^2 .$

Thanks to the definition of $\mathcal{R}'_{\alpha,\beta} $ and $ \psi_1$, we have $ \mathcal{R}'_{\alpha,\beta}{\om_1}=(u''-\beta) \psi_1+u{\om_1}.$ Thus, $\partial_t{\om_1}+i\alpha\mathcal{R}'_{\alpha,\beta}\om_1=i\alpha(u''-\beta) \psi_1,$ $ \om_2(0)=\om(0)-\om_1(0)=0,$ and
\begin{align*}
&\partial_t{\om_2}+i\alpha\mathcal{R}'_{\alpha,\beta}\om_2=-f-i\alpha(u''-\beta) \psi_1:=f_1.
\end{align*}
Moreover, we have
\begin{align*}
\|\partial_yf_1\|_{L^2}+\alpha\|f_1\|_{L^2}&\leq\|\partial_yf\|_{L^2}+\alpha\|f\|_{L^2}+\alpha\|\partial_y((u''-\beta) \psi_1)\|_{L^2}+\alpha^2\|(u''-\beta) \psi_1\|_{L^2}\\&\leq\|\partial_yf\|_{L^2}+\alpha\|f\|_{L^2}+C\alpha(\|\partial_y \psi_1\|_{L^2}+\alpha\|\psi_1\|_{L^2}),
\end{align*}
which gives
\begin{align}\label{f1}
&\int_{0}^T(\|\partial_yf_1(t)\|_{L^2}^2+\alpha^2\|f_1(t)\|_{L^2}^2)dt\\ \nonumber
&\leq C\int_{0}^T(\|\partial_yf(t)\|_{L^2}^2+\alpha^2\|f(t)\|_{L^2}^2)dt+C\alpha^2\int_{0}^T(\|\partial_y\psi_1(t)\|_{L^2}^2
+\alpha^2\|\psi_1(t)\|_{L^2}^2)dt\\ \nonumber
&\leq C\int_{0}^T(\|\partial_yf(t)\|_{L^2}^2+\alpha^2\|f(t)\|_{L^2}^2)dt+C\|\om(0)\|_{L^2}^2.
\end{align}
This means that $f_1\in L^2((0,T);H^1(0,1)).$
 \smallskip

{\bf Step 2.} Now we extend $\om_2,\ \psi_2,\ f_1$ to $t\in[0,+\infty)$ in the following way
\beno
\om_2(t)=e^{-i(t-T)\alpha\mathcal{R}'_{\alpha,\beta}}\om_2(T),\quad \psi_2(t)=-(\partial^2_y-\alpha^2)^{-1}\om_2(t),\quad f_1(t)=0\,\,\text{ for }\,\,t>T.
\eeno
Then $\partial_t{\om_2}+i\alpha\mathcal{R}'_{\alpha,\beta}\om_2=f_1$ for $t\in[0,+\infty)$.
Since $\mathcal{R}'_{\alpha,\beta} $ is a bounded operator on $H^1(0,1) $ and $\om_2(0)=0$,  we have $\om_2\in C([0,T];H^1(0,1)).$
Since $\om_2(t)=e^{-i(t-T)\alpha\mathcal{R}'_{\alpha,\beta}}\om_2(T),$ $\psi_2(t)=-(\partial^2_y-\alpha^2)^{-1}\om_2(t)$ for $t>T,$ we have $\psi_2 \in L^2((T,+\infty);H^1(0,1))$ {(See the proof of Theorem \ref{thm:non-monotone})}. Thanks to $\mathcal{R}'_{\alpha,\beta}(\partial^2_y-\alpha^2)=(\partial^2_y-\alpha^2)\mathcal{R}_{\alpha,\beta}$, we find
\beno
\partial_t{\psi_2}+i\alpha\mathcal{R}_{\alpha,\beta}\psi_2=f_2,\quad \psi_2(0)=0,
\eeno
with $f_2=-(\partial^2_y-\alpha^2)^{-1}f_1.$ Let $f_3=i\alpha(u''-\beta) \psi_2-f_1$. Then
\beno
\partial_t{\om_2}+i\alpha u{\om_2}+f_3=0,
\eeno
where $f_3$ satisfies
\begin{align*}
\|\partial_yf_3\|_{L^2}+\alpha\|f_3\|_{L^2}&\leq\|\partial_yf_1\|_{L^2}+\alpha\|f_1\|_{L^2}+\alpha\|\partial_y((u''-\beta) \psi_2)\|_{L^2}+\alpha^2\|(u''-\beta) \psi_2\|_{L^2}\\&\leq\|\partial_yf_1\|_{L^2}+\alpha\|f_1\|_{L^2}+C\alpha(\|\partial_y \psi_2\|_{L^2}+\alpha\|\psi_2\|_{L^2}).
\end{align*}
Then it follows from Lemma \ref{lem5.7} that for any $s>0$,
\begin{align}\label{om2}&\|\om_2(s)\|_{L^2}^2\leq C\int_0^s\big(\alpha^{-2}\|\partial_yf_3(t)\|_{L^2}^2+\|f_3(t)\|_{L^2}^2\big)dt\\ \nonumber
&\leq C\int_0^{+\infty}\big(\alpha^{-2}\|\partial_yf_1(t)\|_{L^2}^2+\|f_1(t)\|_{L^2}^2+\|\partial_y \psi_2(t)\|_{L^2}^2+\alpha^2\|\psi_2(t)\|_{L^2}^2\big)dt<+\infty.
\end{align}
Thus, $\om_2\in L^{\infty}((0,+\infty);L^2(0,1))$ and $\psi_2=-(\partial^2_y-\alpha^2)^{-1}\om_2\in {L^{\infty}}((0,+\infty);H^2(0,1)).$

Now we can take Laplace transform in $t.$ For $\text{Re}(\lambda)>0$, let
\begin{align*}
&\Phi(\lambda,y)=\int_0^{+\infty}\psi_2(t,y)e^{-\lambda t}dt,\quad F_j(\lambda,y)=\int_0^{T}f_j(t,y)e^{-\lambda t}dt,\ j=1,2.
\end{align*}
Then $ \Phi(\lambda,\cdot)\in H^2(0,1),\ F_1(\lambda,\cdot)\in H^1(0,1)$ for $\text{Re}(\lambda)>0. $ Using Plancherel's formula, we know that
for $\varepsilon>0,\ j=0,1,$
\begin{align}\label{PhiH1}
&\int_{\mathbb{R}}\|\partial_y^j\Phi(\varepsilon+is)\|_{L^2}^2ds=
{2\pi}\int_0^{+\infty}e^{-2\varepsilon t}\|\partial_y^j\psi_2(t)\|_{L^2}^2dt,\\ \label{PhiC1} &\int_{\mathbb{R}}|\partial_y\Phi(\varepsilon+is,j)|^2ds=
{2\pi}\int_0^{+\infty}e^{-2\varepsilon t}|\partial_y\psi_2(t,j)|^2dt,\\ \label{F1}
&\int_{\mathbb{R}}\|\partial_y^jF_1(\varepsilon+is)\|_{L^2}^2ds=
{2\pi}\int_0^{T}e^{-2\varepsilon t}\|\partial_y^jf_1(t)\|_{L^2}^2dt.
\end{align}
Furthermore, $\Phi$ satisfies
\begin{align}\label{W}
(u-i\lambda/\alpha)(\partial_y^2\Phi-\alpha^2\Phi)-(u''-\beta)\Phi=W,\quad \Phi(\lambda,0)=\Phi(\lambda,1)=0
\end{align}
with $W=-(i/\alpha)(\partial^2_y-\alpha^2)F_2=(i/\alpha)F_1.$

If $ \text{Re}(\lambda)\in(0,\alpha\varepsilon_0)$, then $ \text{Im}(i\lambda/\alpha)\in(0,\varepsilon_0)$, and by Lemma \ref{lem5.6},
\begin{align*}
\|\partial_y\Phi(\lambda)\|_{L^2}^2+\alpha^2\|\Phi(\lambda)\|_{L^2}^2\leq C\alpha^{-4}\big(\|\partial_yF_1(\lambda)\|_{L^2}^2+\alpha^2\|F_1(\lambda)\|_{L^2}^2\big).
\end{align*}
Integrating this over $\text{Re}(\lambda)=\varepsilon\in(0,\alpha\varepsilon_0) $ and using \eqref{PhiH1}, \eqref{F1}, we deduce that 
\begin{align*}
\int_0^{+\infty}e^{-2\varepsilon t}\big(\|\partial_y\psi_2(t)\|_{L^2}^2+\alpha^2\|\psi_2(t)\|_{L^2}^2\big)dt\leq C\alpha^{-4}\int_0^{T}e^{-2\varepsilon t}\big(\|\partial_yf_1(t)\|_{L^2}^2+\alpha^2\|f_1(t)\|_{L^2}^2\big)dt.
\end{align*}
Letting $\varepsilon\to 0+$, we obtain
\begin{align}\label{psi2}
\int_0^{+\infty}\big(\|\partial_y\psi_2(t)\|_{L^2}^2+\alpha^2\|\psi_2(t)\|_{L^2}^2\big)dt\leq C\alpha^{-4}\int_0^{T}\big(\|\partial_yf_1(t)\|_{L^2}^2+\alpha^2\|f_1(t)\|_{L^2}^2\big)dt.
\end{align}

{\bf Step 3.} Recall that $\om=\om_1+\om_2,\ \psi=\psi_1+\psi_2.$ It follows from \eqref{psi1a}, \eqref{om2}, \eqref{psi2} and \eqref{f1} that
\begin{align*}
&\|\om(T)\|_{L^2}^2+\alpha^2\int_0^T\big(\|\partial_y\psi(t)\|_{L^2}^2+\alpha^2\|\psi(t)\|_{L^2}^2\big)dt\\ \leq& 2\sum_{j=1}^2\|\om_j(T)\|_{L^2}^2+2\alpha^2\sum_{j=1}^2\int_0^T\big(\|\partial_y\psi_j(t)\|_{L^2}^2+\alpha^2\|\psi_j(t)\|_{L^2}^2\big)dt\\ \leq& C \|\om(0)\|_{L^2}^2+C\int_0^{+\infty}\big(\alpha^{-2}\|\partial_yf_1(t)\|_{L^2}^2+\|f_1(t)\|_{L^2}^2+\alpha^2\|\partial_y \psi_2(t)\|_{L^2}^2+\alpha^4\|\psi_2(t)\|_{L^2}^2\big)dt\\
\leq& C \|\om(0)\|_{L^2}^2+C\alpha^{-2}\int_0^T\big(\|\partial_yf_1(t)\|_{L^2}^2+\alpha^2\|f_1(t)\|_{L^2}^2)dt\\ \leq& C \|\om(0)\|_{L^2}^2+C\alpha^{-2}\int_0^T\big(\|\partial_yf(t)\|_{L^2}^2+\alpha^2\|f(t)\|_{L^2}^2\big)dt,
\end{align*}
which gives the first inequality.

If $f(t,0)=f(t,1)=0 $, then $f_1=0,\ F_1=0$ and  $W=0$ at $y=0,1$. Thus, by Lemma \ref{lem5.6} and \eqref{W}, we deduce that for $ \text{Re}(\lambda)\in(0,\alpha\varepsilon_0),\ j=0,1,$
\begin{align*}
|\partial_y\Phi(\lambda,j)|\leq C\alpha^{-\frac{1}{2}}\big(\|\partial_yW\|_{L^2}+\alpha\|W\|_{L^2}\big)
=C\alpha^{-\frac{3}{2}}\big(\|\partial_yF_1(\lambda)\|_{L^2}+\alpha\|F_1(\lambda)\|_{L^2}\big).
\end{align*}
Hence,
\begin{align*}
\alpha|\partial_y\Phi(\lambda,j)|^2\leq C\alpha^{-2}(\|\partial_yF_1(\lambda)\|_{L^2}^2+\alpha^2\|F_1(\lambda)\|_{L^2}^2).\
\end{align*} Integrating this over $\text{Re}(\lambda)=\varepsilon\in(0,\alpha\varepsilon_0) $ and using \eqref{PhiC1}, \eqref{F1}, we obtain
\begin{align*}
\alpha\int_0^{+\infty}e^{-2\varepsilon t}|\partial_y\psi_2(t,j)|^2dt\leq C\alpha^{-2}\int_0^{T}e^{-2\varepsilon t}\big(\|\partial_yf_1(t)\|_{L^2}^2+\alpha^2\|f_1(t)\|_{L^2}^2\big)dt,\ j=0,1.
\end{align*}
Letting $ \varepsilon\to 0+$, we get
\begin{align}\label{psi2a}
\alpha\int_0^{+\infty}|\partial_y\psi_2(t,j)|^2dt\leq C\alpha^{-2}\int_0^{T}(\|\partial_yf_1(t)\|_{L^2}^2+\alpha^2\|f_1(t)\|_{L^2}^2)dt,\ j=0,1.
\end{align}
Now the second inequality follows from \eqref{psi1b}, \eqref{psi2a} and \eqref{f1}.
\end{proof}

\subsection{Decay estimates via the vector field method}
In this subsection,
we assume that $\mathcal{R}_{\alpha,\beta} $ has no embedding eigenvalues or eigenvalues. Let $\psi=-(\partial^2_y-\alpha^2)^{-1}\om$ and $\om(t,y)$ solve $\partial_t{\om}+i\alpha\mathcal{R}'_{\alpha,\beta}{\om}=0$ for $t\in[0,+\infty)$, $y\in[0,1]$. 

First of all, it follows from Proposition \ref{lem5.1} that
\begin{align}\label{psiL2H1}
&\sup_{t>0}\|\om(t)\|_{L^2}^2+\alpha^2\int_0^{+\infty}(\|\partial_y\psi(t)\|_{L^2}^2+\alpha^2\|\psi(t)\|_{L^2}^2)dt\leq C\|\om(0)\|_{L^2}^2,\\
\label{psiL2bd}&\alpha\int_0^{+\infty}(|\partial_y\psi(t,0)|^2+|\partial_y\psi(t,1)|^2)dt\leq C \|\om(0)\|_{L^2}^2.
\end{align}

We introduce the vector field $X=(1/u')\partial_y+i\alpha t$, which commutes with $\partial_t+i\alpha u$.
Then we have
\begin{align*}
(\partial_t+i\alpha u)X\om=&X(\partial_t{\om}+i\alpha u{\om})=-i\alpha((1/u')\partial_y+i\alpha t)((u''-\beta) \psi)\\
=&-i\alpha(u'''/u')\psi-i\alpha(u''-\beta)X\psi.
\end{align*}
We denote
\beno
\om_1=X{\om},\quad \psi_1=-(\partial^2_y-\alpha^2)^{-1}\om_1,\quad \psi_2=-(\partial^2_y-\alpha^2)^{-1}(\om'/u'),\quad \psi_3=\psi_2-\psi'/u'.
\eeno
Then we find
\beno
\psi_1=\psi_2+i\alpha t\psi,\quad X\psi=\psi_1-\psi_3.
\eeno
This shows that
\begin{align*}
&(\partial_t+i\alpha u)\om_1=-i\alpha(u'''/u')\psi-i\alpha(u''-\beta)( \psi_1-\psi_3),
\end{align*}
which implies
\begin{align*}
\partial_t{\om_1}+i\alpha\mathcal{R}'_{\alpha,\beta}{\om_1}&=(\partial_t+i\alpha u)\om_1+i\alpha(u''-\beta) \psi_1\\
&=-i\alpha(u'''/u')\psi+i\alpha(u''-\beta)\psi_3:=\psi_4.
\end{align*}

\begin{lemma}\label{lem5.3}
It holds that for any $t>0$,
\beno
\alpha^2\big(\|\partial_y\psi(t)\|_{L^2}+\alpha\|\psi(t)\|_{L^2}\big)\leq C(1+t)^{-1}\big(\|\partial_y\om(0)\|_{L^2}+\alpha\|\om(0)\|_{L^2}\big),
\eeno
where the constant $C$ only depends  on $\beta$ and $u$.
\end{lemma}

\begin{proof}
By Proposition \ref{lem5.1} we have\begin{align}\label{psi1L2H1}
\sup_{t>0}\|\om_1(t)\|_{L^2}^2\leq C\|\om_1(0)\|_{L^2}^2+C\int_0^{+\infty}\big(\alpha^{-2}\|\partial_y\psi_4(t)\|_{L^2}^2+\|\psi_4(t)\|_{L^2}^2\big)dt.
\end{align}
To proceed it, let us first claim that
\begin{align}\label{psi4H1a}
&\alpha^{-1}\|\partial_y\psi_4\|_{L^2}+\|\psi_4\|_{L^2}
 \leq C\big(\|\partial_y\psi\|_{L^2}+\alpha\|\psi\|_{L^2})+C\alpha^{\frac{1}{2}}(|\psi'(t,0)|+|\psi'(t,1)|\big).
\end{align}

Using \eqref{psi1L2H1}, \eqref{psi4H1a}, \eqref{psiL2H1} and \eqref{psiL2bd}, we conclude that
\begin{align}\label{psi1L2H1a}
\sup_{t>0}\|\om_1(t)\|_{L^2}^2\leq& C\|\om_1(0)\|_{L^2}^2+C\int_0^{+\infty}\big(\|\partial_y\psi(t)\|_{L^2}^2+\alpha^{2}\|\psi(t)\|_{L^2}^2\big)dt\\ \nonumber &+C\alpha\int_0^{+\infty}\big(|\partial_y\psi(t,0)|^2+|\partial_y\psi(t,1)|^2\big)dt\\ \nonumber
\leq& C\|\om_1(0)\|_{L^2}^2+C\alpha^{-2}\|\om(0)\|_{L^2}^2+C\|\om(0)\|_{L^2}^2\leq C\|\om(0)\|_{H^1}^2,
\end{align}
here we used the fact that
\begin{align*}
&\|\om_1(0)\|_{L^2}=\|(\om'/u'+i\alpha t\om)|_{t=0}\|_{L^2}=\|(\om'/u')|_{t=0}\|_{L^2}\leq C\|\om'|_{t=0}\|_{L^2}\leq C\|\om(0)\|_{H^1}.
\end{align*}

Since $\om_1=\om'/u'+i\alpha t\om,\ \psi=-(\partial^2_y-\alpha^2)^{-1}\om,$ and $1/u'\in C^1([0,1])$, we have
\begin{align*}
\alpha t\big(\|\partial_y\psi\|_{L^2}^2
+\alpha^2\|\psi\|_{L^2}^2\big)=&\alpha t\langle \psi,\om\rangle=i\langle \psi,\om_1-\om'/u'\rangle=i\langle \psi,\om_1\rangle+i\langle (\psi/u')',\om\rangle\\
\leq& \|\psi\|_{L^2}\|\om_1\|_{L^2}+\|(\psi/u')'\|_{L^2}\|\om\|_{L^2}\\
\leq&\|\psi\|_{L^2}\|\om_1\|_{L^2}+
C\big(\|\psi\|_{L^2}+\|\psi'\|_{L^2}\big)\|\om\|_{L^2}\\ \leq&
C\big(\alpha\|\psi\|_{L^2}+\|\psi'\|_{L^2}\big)\big(\alpha^{-1}\|\om_1\|_{L^2}+\|\om\|_{L^2}\big),
\end{align*}
from which, \eqref{psiL2H1} and \eqref{psi1L2H1a}, we infer that
\begin{align}\label{psiH1t>1}
\alpha t\big(\|\partial_y\psi(t)\|_{L^2}+\alpha\|\psi(t)\|_{L^2}\big)\leq&
C\big(\alpha^{-1}\|\om_1(t)\|_{L^2}+\|\om(t)\|_{L^2}\big)\\
\nonumber\leq&
C\big(\alpha^{-1}\|\om(0)\|_{H^1}+\|\om(0)\|_{L^2}\big)\\
\nonumber\leq& C\alpha^{-1}\big(\|\partial_y\om(0)\|_{L^2}+\alpha\|\om(0)\|_{L^2}\big).
\end{align}
On the other hand, using
$\|\partial_y\psi\|_{L^2}^2
+\alpha^2\|\psi\|_{L^2}^2=\langle \psi,\om\rangle\leq \|\psi\|_{L^2}\|\om\|_{L^2}$,
we get
\begin{align}\label{psiH1t<1}
\|\partial_y\psi(t)\|_{L^2}+\alpha\|\psi(t)\|_{L^2}\leq
C\alpha^{-1}\|\om(t)\|_{L^2}\leq
C\alpha^{-1}\|\om(0)\|_{L^2}.
\end{align}
Then the lemma is a consequence of \eqref{psiH1t>1} and \eqref{psiH1t<1}.\smallskip

It remains to prove \eqref{psi4H1a}. As $u(y)\in C^4([0,1])$ and $u'(y)\ge c_0$, we have $u'''/u',\ u''-\beta\in C^1([0,1])$, and
\begin{align}\label{psi4H1}
\alpha^{-1}\|\partial_y\psi_4\|_{L^2}+\|\psi_4\|_{L^2}\leq& \|\partial_y((u'''/u')\psi)\|_{L^2}+\|\partial_y((u''-\beta)\psi_3)\|_{L^2}\\ \nonumber&+
\alpha\|(u'''/u')\psi\|_{L^2}+\alpha\|(u''-\beta)\psi_3\|_{L^2}\\ \nonumber
\leq&C\big(\|\partial_y\psi\|_{L^2}+\|\partial_y\psi_3\|_{L^2}+\alpha\|\psi\|_{L^2}
+\alpha\|\psi_3\|_{L^2}\big).
\end{align}
To estimate $\psi_3$, we decompose $ \psi_3= \psi_{3,1}+\psi_{3,2}$, where
\beno
(\partial^2_y-\alpha^2)\psi_{3,1}=(\partial^2_y-\alpha^2)\psi_{3},\quad (\partial^2_y-\alpha^2)\psi_{3,2}=0
\eeno
with $\psi_{3,1}=0,\ \psi_{3,2}=\psi_3$ at $y=0,1.$ Recall that
\beno
(\partial^2_y-\alpha^2)\psi_2=-\om'/u',\quad (\partial^2_y-\alpha^2)\psi=-\om.
\eeno
Then we have
\begin{align}\label{psi31}
(\partial^2_y-\alpha^2)\psi_{3,1}=&(\partial^2_y-\alpha^2)\psi_{3}=(\partial^2_y-\alpha^2)(\psi_2-\psi'/u')\\ \nonumber
=&(\partial^2_y-\alpha^2)\psi_2-\partial_y(\partial^2_y-\alpha^2)\psi/u'-2(\psi'(1/u')')'+\psi'(1/u')''\\ \nonumber=&-\om'/u'-
\partial_y(-\om)/u'-2(\psi'(1/u')')'+\psi'(1/u')''\\\nonumber
=&-2(\psi'(1/u')')'+\psi'(1/u')'',
\end{align}
which implies
\begin{align*}
\|\partial_y\psi_{3,1}\|_{L^2}^2
+\alpha^2\|\psi_{3,1}\|_{L^2}^2=&-\langle \psi_{3,1},(\partial^2_y-\alpha^2)\psi_{3,1}\rangle=-\langle \psi_{3,1},(\partial^2_y-\alpha^2)\psi_{3}\rangle\\=&-\langle \psi_{3,1},-2(\psi'(1/u')')'+\psi'(1/u')''\rangle\\
=&-2\langle \partial_y\psi_{3,1},\psi'(1/u')'\rangle-\langle \psi_{3,1},\psi'(1/u')''\rangle\\ \leq&C\|\partial_y\psi_{3,1}\|_{L^2}\|\partial_y\psi\|_{L^2}+C\|\psi_{3,1}\|_{L^2}\|\partial_y\psi\|_{L^2}.
\end{align*}
This shows that 
\begin{align}\label{psi31H1}
&\|\partial_y\psi_{3,1}\|_{L^2}
+\alpha\|\psi_{3,1}\|_{L^2} \leq C\|\partial_y\psi\|_{L^2}.
\end{align}

To estimate $\psi_{3,2}$, we recall that
\beno
(\partial^2_y-\alpha^2)\gamma_j=0,\ \gamma_j(j)=1,\ \gamma_j(1-j)=0\ \text{for}\ j\in\{0,1\},
\eeno
where $\gamma_j$  is defined in (\ref{gamma12-def}). So, $\psi_{3,2}=\psi_{3,2}(t,0)\gamma_0+\psi_{3,2}(t,1)\gamma_1.$ Thanks to $|\gamma_j'(j)|= \alpha\coth\alpha\leq C\alpha$ for $j\in\{0,1\},$ we get
\begin{align*}
&\|\gamma'_j\|_{L^2}^2
+\alpha^2\|\gamma_j\|_{L^2}^2=-\langle \gamma_j,(\partial^2_y-\alpha^2)\gamma_j\rangle+\gamma_j'\gamma_j|_0^1=|\gamma_j'\gamma_j(j)|=|\gamma_j'(j)| \leq C\alpha,
\end{align*}
which gives
\begin{align*}
\|\partial_y\psi_{3,2}\|_{L^2}
+\alpha\|\psi_{3,2}\|_{L^2} &\leq |\psi_{3,2}(t,0)|(\|\gamma'_0\|_{L^2}
+\alpha\|\gamma_0\|_{L^2})+|\psi_{3,2}(t,1)|(\|\gamma'_1\|_{L^2}
+\alpha\|\gamma_1\|_{L^2})\\
&\leq C\alpha^{\frac{1}{2}}\big(|\psi_{3,2}(t,0)|+|\psi_{3,2}(t,1)|\big).
\end{align*}
Thanks to $\psi_{3,2}(t,j)=\psi_{3}(t,j)$, $\psi_{2}(t,j)=0$ for $j\in\{0,1\}$, and $\psi_3=\psi_2-\psi'/u'$, we get
\beno
|\psi_{3,2}(t,j)|=|\psi_{3}(t,j)|=|\psi'(t,j)/{u'}(j)|\leq C|\psi'(t,j)|
\eeno
 and hence,
\begin{align}\label{psi32H1}
\|\partial_y\psi_{3,2}\|_{L^2}
+\alpha\|\psi_{3,2}\|_{L^2}  &\leq C\alpha^{\frac{1}{2}}(|\psi_{3,2}(t,0)|+|\psi_{3,2}(t,1)|)\\ \nonumber&\leq C\alpha^{\frac{1}{2}}(|\psi'(t,0)|+|\psi'(t,1)|).
\end{align}
Now \eqref{psi4H1a} follows from \eqref{psi4H1}, \eqref{psi31H1} and \eqref{psi32H1}.
\end{proof}

Since ${\psi}(t,j)=0$ for $j=0,1$, we have
\begin{align*}
\partial_{t}{\omega}(t,j)+i \alpha u(j){\omega}(t,j)=0,\ \text{for}\ j\in\{0,1\},
\end{align*} and $|{\omega}(t,j)|=|e^{-i \alpha tu(j)}{\omega}(0,j)|\leq \|\om(0)\|_{L^{\infty}}. $
With $\gamma_j$ defined as above, using the fact that
\begin{align*}
&\langle \om,\gamma_1\rangle=-\langle (\partial^2_y-\alpha^2)\psi,\gamma_1\rangle=-\langle \psi,(\partial^2_y-\alpha^2)\gamma_1\rangle-(\psi'\gamma_1-\psi\gamma_1')|_0^1=-\psi'(t,1),
\end{align*}
we infer that for any $t>0$,
\begin{align*}
\alpha t|\psi'(t,1)|=&\alpha t|\langle \om,\gamma_1\rangle|=|\langle \om_1-\om'/u',\gamma_1\rangle|=\left|\langle \om_1,\gamma_1\rangle+\langle \om,(\gamma_1/u')'\rangle-\om\gamma_1/u'|_{y=0}^1\right|\\ \leq &\|\om_1\|_{L^2}\|\gamma_1\|_{L^2}+\|\om\|_{L^2}\|(\gamma_1/u')'\|_{L^2}+|{\omega}(t,1)/u'(1)|\\
\leq& \|\om_1\|_{L^2}\|\gamma_1\|_{L^2}+C\|\om\|_{L^2}(\|\gamma_1\|_{L^2}+\|\gamma_1'\|_{L^2})+C\|\om(0)\|_{L^{\infty}}\\
\leq& C\|\om_1\|_{L^2}\alpha^{-\frac{1}{2}}+C\|\om\|_{L^2}(\alpha^{-\frac{1}{2}}+\alpha^{\frac{1}{2}})+C\|\om(0)\|_{L^{\infty}}\\
\leq& C\alpha^{\frac{1}{2}}(\alpha^{-1}\|\om_1\|_{L^2}+\|\om\|_{L^2})+C\|\om(0)\|_{H^1}^{\f12}\|\om(0)\|_{L^2}^{\f12}\\ \leq&
C\alpha^{\frac{1}{2}}(\alpha^{-1}\|\om(0)\|_{H^1}+\|\om(0)\|_{L^2})\leq C\alpha^{-\frac{1}{2}}(\|\partial_y\om(0)\|_{L^2}+\alpha\|\om(0)\|_{L^2}).
\end{align*}
On the other hand, we have
\begin{align*}
|\psi'(t,1)|=|\langle \om,\gamma_1\rangle|\leq \|\om\|_{L^2}\|\gamma_1\|_{L^2}\leq C\|\om(t)\|_{L^2}\alpha^{-\frac{1}{2}}\leq C\|\om(0)\|_{L^2}\alpha^{-\frac{1}{2}}.
\end{align*}
This shows that
\begin{align}\label{psit1}
&|\psi'(t,1)|\leq C\alpha^{-\frac{3}{2}}(1+t)^{-1}(\|\partial_y\om(0)\|_{L^2}+\alpha\|\om(0)\|_{L^2}).
\end{align}
Similarly, we have
\begin{align}\label{psit0}
&|\psi'(t,0)|=|\langle \om,\gamma_0\rangle|\leq C\alpha^{-\frac{3}{2}}(1+t)^{-1}(\|\partial_y\om(0)\|_{L^2}+\alpha\|\om(0)\|_{L^2}).
\end{align}

The following lemma is devoted to the decay estimate for the second component of the velocity. For this, we introduce the following norms:
\beno
\|\om\|_{-2}^2=\|\psi\|_{L^2}^2,\quad \|\om\|_{-1}^2=\|\partial_y\psi\|_{L^2}^2+\alpha^2\|\psi\|_{L^2}^2,
\quad \|\om\|_{0}^2=\|\om\|_{L^2}^2,
\eeno
where $\psi=-(\partial^2_y-\alpha^2)^{-1}\om$ and
\beno
\|\om\|_{1}^2=\|\partial_y\om\|_{L^2}^2+\alpha^2\|\om\|_{L^2}^2,\quad \|\om\|_{2}^2=\|\partial_y^2\om\|_{L^2}^2+2\alpha^2\|\partial_y\om\|_{L^2}^2+\alpha^4\|\om\|_{L^2}^2.
\eeno
Since $\|\om\|_{0}^2=\|(\partial^2_y-\alpha^2)\psi\|_{L^2}^2=\|\partial_y^2\psi\|_{L^2}^2+2\alpha^2\|\partial_y\psi\|_{L^2}^2
+\alpha^4\|\psi\|_{L^2}^2$, we have
\beno
\alpha^{k-j}\|\om\|_{j}\leq \|\om\|_{k}\quad \text{for every}\  -2\leq j\leq k\leq 2.
\eeno

 We denote by the semigroup $\om(t)=e^{-it\alpha\mathcal{R}'_{\alpha,\beta}}\om_0$ the solution to $\partial_t{\om}+i\alpha\mathcal{R}'_{\alpha,\beta}{\om}=0,\ \om(0)=\om_0$. Then Lemma \ref{lem5.3} and \eqref{psiL2H1} imply that
 \begin{align}\label{H-1H1}
 \alpha^2\|e^{-it\alpha\mathcal{R}'_{\alpha,\beta}}f\|_{-1}\leq C(1+t)^{-1}\|f\|_{1},\quad
 \|e^{-it\alpha\mathcal{R}'_{\alpha,\beta}}f\|_{0}\leq C\|f\|_{0}.
 \end{align}

\begin{lemma}\label{lem5.3-2}
It holds that for any $t>0$,
\beno
\alpha^4\|\psi(t)\|_{L^2}\leq C(1+t)^{-2} (\|\partial_y^2\om(0)\|_{L^2}+\alpha\|\partial_y\om(0)\|_{L^2}+\alpha^2\|\om(0)\|_{L^2}),
\eeno
where the constant $C$ only depends on $\beta$ and $u$.
\end{lemma}

\begin{proof}
It suffices to show that
\beno
\alpha^4\|e^{-it\alpha\mathcal{R}'_{\alpha,\beta}}f\|_{-2}\leq C(1+t)^{-2}\|f\|_{2}.
\eeno
For $T>0$, we define
\begin{align*}
M=M(T):=\sup\big\{\alpha^4(1+t)^2\|e^{-it\alpha\mathcal{R}'_{\alpha,\beta}}f\|_{-2}: 0<t<T,\ f\in  H^2(0,1),\ \|f\|_2\leq 1\big\}.
\end{align*}

First of all, we get by \eqref{H-1H1} that
\beno
\alpha^4(1+t)^2\|e^{-it\alpha\mathcal{R}'_{\alpha,\beta}}f\|_{-2}\leq\alpha^2(1+t)^2\|e^{-it\alpha\mathcal{R}'_{\alpha,\beta}}f\|_{0}
\leq C\alpha^2(1+t)^2\|f\|_{0}\leq C(1+T)^2\|f\|_{2},
\eeno
which implies that $M(T)\leq C(1+T)^2.$ Now we fix $T>0$ and assume $M=M(T)>1$. We  will show that
\beno
M(T)\leq C\big(\ln(M(T)+1)+1\big)
\eeno
with $C$ independent of $T$ and $\alpha$.

Let us first claim that for $0<t<T$,
\ben\label{om1-est}
\|\om_1(t)\|_{-2}\le C\alpha^{-3}t^{-1}\|\om(0)\|_{2}\big(1+
\ln(M+1)\big),
\een
which will be proved in Lemma \ref{lem:om-1}.

Recall that $\psi_1=-(\partial^2_y-\alpha^2)^{-1}\om_1,\ \psi_3=\psi_2-\psi'/u',\ \psi_1=\psi_2+i\alpha t\psi,\ \psi_3= \psi_{3,1}+\psi_{3,2}$. Then by \eqref{psi31H1}, \eqref{psi32H1}, \eqref{psiH1t>1}, \eqref{psit1} and \eqref{psit0}, we get
\begin{align*}
\alpha t\|\psi\|_{L^2}\leq \|\psi_1\|_{L^2}+\|\psi_2\|_{L^2}\leq& \|\om_1\|_{-2}+\|\psi_3\|_{L^2}+\|\psi'/u'\|_{L^2}\\ \leq &\|\om_1\|_{-2}+\|\psi_{3,1}\|_{L^2}+\|\psi_{3,2}\|_{L^2}+C\|\psi'\|_{L^2}\\ \leq&
\|\om_1\|_{-2}+C\alpha^{-1}\|\psi'\|_{L^2}+C\alpha^{-\frac{1}{2}}(|\psi'(t,0)|+|\psi'(t,1)|)+C\|\psi'\|_{L^2}\\ \leq&
\|\om_1\|_{-2}+C\alpha^{-2}t^{-1}(\|\partial_y\om(0)\|_{L^2}+\alpha\|\om(0)\|_{L^2}).
\end{align*}
This means that
\begin{align*}
&\alpha t\|\psi(t)\|_{L^2}\leq
\|\om_1(t)\|_{-2}+C\alpha^{-2}t^{-1}\|\om(0)\|_{1},
\end{align*}
which along with \eqref{om1-est} gives
\begin{align*}
\alpha t\|\psi(t)\|_{L^2}\leq&
\|\om_1(t)\|_{-2}+C\alpha^{-2}t^{-1}\|\om(0)\|_{1}\\ \leq& C\alpha^{-3}t^{-1}\|\om(0)\|_{2}(1+
\ln(M+1))+C\alpha^{-3}t^{-1}\|\om(0)\|_{2}\\ \leq& C\alpha^{-3}t^{-1}\|\om(0)\|_{2}(1+
\ln(M+1)).
\end{align*}
And by \eqref{psiH1t<1}, we have
\begin{align*}
&\|\psi(t)\|_{L^2}\leq C\alpha^{-2}\|\om(0)\|_{L^2}\leq C\alpha^{-4}\|\om(0)\|_{2}.
\end{align*}
Then we conclude that for $0<t<T$,
\begin{align*}
\|e^{-it\alpha\mathcal{R}'_{\alpha,\beta}}\om(0)\|_{-2}=&\|\om(t)\|_{-2}=\|\psi(t)\|_{L^2}\leq C\alpha^{-4}\|\om(0)\|_{2}\min(t^{-2}(1+
\ln(M+1)),1)\\ \leq& C\alpha^{-4}(1+t)^{-2}\|\om(0)\|_{2}\big(1+
\ln(M+1)\big),
\end{align*}
Here $C$ is a constant independent of $T,\ \alpha$ and $\om(0).$ Thanks to the definition of $M(T)$, we have
\beno
M(T)\leq C\big(1+\ln(M(T)+1)\big).
\eeno
Thus, there exists a constant $C_0>0$ independent of $T$ and $\alpha$ so that if $M(T)>1$, then $M(T)\leq C_0\big(1+
\ln(M(T)+1)\big)$. This implies the existence of a constant $C_1>1$ so that $M(T)<C_1$ for every $T>0$.
Now we have
\begin{align*}
\alpha^4\|\psi(t)\|_{L^2}=&\al^4\|e^{-it\alpha\mathcal{R}'_{\alpha,\beta}}\om(0)\|_{-2}\leq C_1(1+t)^{-2} \|\om(0)\|_{2}\\
\leq& C_1(1+t)^{-2}\big(\|\partial_y^2\om(0)\|_{L^2}+\alpha\|\partial_y\om(0)\|_{L^2}+\alpha^2\|\om(0)\|_{L^2}\big),
\end{align*}
which gives our result.
\end{proof}.

\begin{lemma}\label{lem:om-1}
It holds that for any $0<t<T$,
\beno
\|\om_1(t)\|_{-2}\le C\alpha^{-3}t^{-1}\|\om(0)\|_{2}\big(1+
\ln(M+1)\big),
\eeno
where the constant $C$ is independent of $T$ and $ \al$.
\end{lemma}

\begin{proof}

Recall that $\partial_t{\om_1}+i\alpha\mathcal{R}'_{\alpha,\beta}{\om_1}=\psi_4$. By Duhamel's principle, we get
\begin{align*}
&\om_1(t)=e^{-it\alpha\mathcal{R}'_{\alpha,\beta}}\om_1(0)+\int_0^te^{-i(t-s)\alpha\mathcal{R}'_{\alpha,\beta}}\psi_4(s)ds,
\end{align*}
from which, we infer that 
\begin{align}\label{om1}
&\|\om_1(t)\|_{-2}\leq\|e^{-it\alpha\mathcal{R}'_{\alpha,\beta}}\om_1(0)\|_{-2}+
\int_0^t\|e^{-i(t-s)\alpha\mathcal{R}'_{\alpha,\beta}}\psi_4(s)\|_{-2}ds.
\end{align}

Thanks to $\om_1(0,y)=\partial_y\om(0,y)/u'(y)$, we get
\begin{align*}
\|\om_1(0)\|_{1}=\|\partial_y\om(0)/u'\|_{1}\leq C\|\partial_y\om(0)\|_{1}\leq C\|\om(0)\|_{2},
\end{align*}
which along with \eqref{H-1H1} gives
\begin{align}\label{om1-2}
\|e^{-it\alpha\mathcal{R}'_{\alpha,\beta}}\om_1(0)\|_{-2}&\leq \alpha^{-1}\|e^{-it\alpha\mathcal{R}'_{\alpha,\beta}}\om_1(0)\|_{-1}\leq C\alpha^{-3}t^{-1}\|\om_1(0)\|_{1}\\ \nonumber&\leq C\alpha^{-3}t^{-1}
\|\om(0)\|_{2}.
\end{align}

By \eqref{psi4H1a}, \eqref{psiH1t>1}, \eqref{psit1} and \eqref{psit0}, we have
\begin{align*}
\alpha^{-1}\|\psi_4(t)\|_{1}\leq&
C(\|\partial_y\psi(t)\|_{L^2}+\alpha\|\psi(t)\|_{L^2})+C\alpha^{\frac{1}{2}}(|\psi'(t,0)|+|\psi'(t,1)|)\\
\leq &C(\alpha^{-2}+\alpha^{-1})t^{-1}(\|\partial_y\om(0)\|_{L^2}+\alpha\|\om(0)\|_{L^2})\leq C\alpha^{-1}t^{-1}\|\om(0)\|_{1}.
\end{align*}
from which and \eqref{H-1H1}, we infer that for $t>s>0$,
\begin{align}\label{psi4-2}
\|e^{-i(t-s)\alpha\mathcal{R}'_{\alpha,\beta}}\psi_4(s)\|_{-2}&\leq \alpha^{-1}\|e^{-i(t-s)\alpha\mathcal{R}'_{\alpha,\beta}}\psi_4(s)\|_{-1}\leq C\alpha^{-3}(t-s)^{-1}\|\psi_4(s)\|_{1}\\ \nonumber&\leq C\alpha^{-3}(t-s)^{-1}
s^{-1}\|\om(0)\|_{1}.
\end{align}

As $(t-s)^{-1}s^{-1} $ is not integrable, we have to improve the estimate for $s$ close to $t$ or $0$.
To this end, we decompose $\psi_4=-i\alpha(u'''/u')\psi+i\alpha(u''-\beta)\psi_3= \psi_{4,1}+\psi_{4,2}+\psi_{4,3}+\psi_{4,4}$,
where
\begin{align*}
&\psi_{4,1}(t,y)=-i\alpha\int_0^y(u'''/u')'(z)\psi(t,z)dz,\\
&\psi_{4,2}(t,y)=-i\alpha\int_0^y(u'''/u')(z)\partial_y\psi(t,z)dz,\\ &\psi_{4,3}=i\alpha(u''-\beta)\psi_{3,1},\ \psi_{4,4}=i\alpha(u''-\beta)\psi_{3,2}.
\end{align*}
Then we have
\begin{align*}
&\psi_{4,1}+\psi_{4,2}=-i\alpha(u'''/u')\psi,\ \psi_{4,3}+\psi_{4,4}=i\alpha(u''-\beta)\psi_{3}.
\end{align*}

Thanks to the definition of $M=M(T)$, we deduce that for any $f\in H^2(0,1)$ and $0<s<T$,
\begin{align*}
\|e^{-is\alpha\mathcal{R}'_{\alpha,\beta}}f\|_{-2}\leq M\alpha^{-4}(1+s)^{-2}\|f\|_{2},
\end{align*}
and by \eqref{H-1H1}, we have
\begin{align*}
\|e^{-is\alpha\mathcal{R}'_{\alpha,\beta}}f\|_{-2}\leq& \alpha^{-1}\|e^{-is\alpha\mathcal{R}'_{\alpha,\beta}}f\|_{-1}\leq C\alpha^{-3}(1+s)^{-1}\|f\|_{1}\\
\leq& C\alpha^{-4}(1+s)^{-1}\|f\|_{2}.
\end{align*}
Therefore,
\begin{align}\label{MT}
\|e^{-is\alpha\mathcal{R}'_{\alpha,\beta}}f\|_{-2}\leq& \alpha^{-4}\min(M(1+s)^{-2} ,C(1+s)^{-1})\|f\|_{2}\\ \nonumber
\leq& CM\alpha^{-4}(1+s)^{-1}(1+M+s)^{-1}\|f\|_{2},
\end{align}
which implies
\begin{align}\label{psiL2}
\|\psi(s)\|_{L^2}=\|\om(s)\|_{-2}=\|e^{-is\alpha\mathcal{R}'_{\alpha,\beta}}\om(0)\|_{-2}\leq \frac{CM\|\om(0)\|_{2}}{\alpha^{4}(1+s)(1+M+s)}.
\end{align}

Let $\psi_{3,3}=(\partial^2_y-\alpha^2)^{-1}\psi_{3,1}$.  Using \eqref{psi31} and  $\psi_{3,1}=0$ at $y=0,1,$ we get
\begin{align*}
\|\psi_{3,1}\|_{L^2}^2=&\langle \psi_{3,1},(\partial^2_y-\alpha^2)\psi_{3,3}\rangle=\langle (\partial^2_y-\alpha^2)\psi_{3,1},\psi_{3,3}\rangle\\=&\langle -2(\psi'(1/u')')'+\psi'(1/u')'',\psi_{3,3}\rangle=-2\langle \psi,(\psi_{3,3}'(1/u')')'\rangle-\langle \psi,(\psi_{3,3}(1/u')'')'\rangle\\ =&-\langle \psi,2\psi_{3,3}''(1/u')'+3\psi_{3,3}'(1/u')''+\psi_{3,3}(1/u')'''\rangle\\ \leq&C\|\psi\|_{L^2}(\|\partial_y^2\psi_{3,3}\|_{L^2}+
\|\partial_y\psi_{3,3}\|_{L^2}+\|\psi_{3,3}\|_{L^2})\leq C\|\psi\|_{L^2}\|\psi_{3,1}\|_{L^2},
\end{align*}
here we used $\|\psi_{3,1}\|_{L^2}^2=\|\partial_y^2\psi_{3,3}\|_{L^2}^2+
2\alpha^2\|\partial_y\psi_{3,3}\|_{L^2}^2+\alpha^4\|\psi_{3,3}\|_{L^2}^2. $ This gives $\|\psi_{3,1}\|_{L^2}\leq C\|\psi\|_{L^2} $, and then
\begin{align}\label{psi4123}
\|(\psi_{4,1}+\psi_{4,2}+\psi_{4,3})(s)\|_0=&\|-i\alpha(u'''/u')\psi(s)+i\alpha(u''-\beta)\psi_{3,1}(s)\|_{0}\\ \nonumber\leq & C\alpha\|\psi(s)\|_{L^2}+C\alpha\|\psi_{3,1}(s)\|_{L^2}\\\nonumber
\leq& C\alpha\|\psi(s)\|_{L^2}\leq CM\alpha^{-3}s^{-2}\|\om(0)\|_{2}.
\end{align}

Using $\|\partial^2_y\psi_{3,2}\|_{L^2}=\alpha^2\|\psi_{3,2}\|_{L^2}$, \eqref{psi32H1}, \eqref{psit1} and \eqref{psit0}, we obtain
\begin{align*}
\|\psi_{3,2}\|_2^2=&\|\partial_y^2\psi_{3,2}\|_{L^2}^2+
2\alpha^2\|\partial_y\psi_{3,2}\|_{L^2}^2+\alpha^4\|\psi_{3,2}\|_{L^2}^2\\
\leq& C\alpha^2(\|\partial_y\psi_{3,2}\|_{L^2}+\alpha\|\psi_{3,2}\|_{L^2})^2\leq C\alpha^3(|\psi'(t,0)|+|\psi'(t,1)|)^2\\
\leq& C(1+t)^{-2}(\|\partial_y\om(0)\|_{L^2}+\alpha\|\om(0)\|_{L^2})^2\leq C(1+t)^{-2}\|\om(0)\|_{1}^2,
\end{align*}
which gives
\begin{align*}
\|\psi_{3,2}(t)\|_2\leq C(1+t)^{-1}\|\om(0)\|_{1},
\end{align*}
and
\begin{align}\label{psi44}
\|\psi_{4,4}(s)\|_2=\|i\alpha(u''-\beta)\psi_{3,2}(s)\|_2\leq C\alpha\|\psi_{3,2}(s)\|_2\leq C\alpha(1+s)^{-1}\|\om(0)\|_{1}.
\end{align}
Then by \eqref{H-1H1}, \eqref{psi4123}, \eqref{psi44} and \eqref{MT}, we infer that for $0<s<t<T$,
\begin{align}\nonumber
\|e^{-i(t-s)\alpha\mathcal{R}'_{\alpha,\beta}}\psi_4(s)\|_{-2}\leq& \|e^{-i(t-s)\alpha\mathcal{R}'_{\alpha,\beta}}(\psi_{4,1}+\psi_{4,2}+\psi_{4,3})(s)\|_{-2}+
\|e^{-i(t-s)\alpha\mathcal{R}'_{\alpha,\beta}}\psi_{4,4}(s)\|_{-2}\\ \nonumber\leq & \alpha^{-2}\|e^{-i(t-s)\alpha\mathcal{R}'_{\alpha,\beta}}(\psi_{4,1}+\psi_{4,2}+\psi_{4,3})(s)\|_{0}\\ \nonumber&+
CM\alpha^{-4}(1+t-s)^{-1}(1+M+t-s)^{-1}\|\psi_{4,4}(s)\|_{2}\\ \label{psi4st}\leq& C\alpha^{-2}\|(\psi_{4,1}+\psi_{4,2}+\psi_{4,3})(s)\|_{0}+
\frac{CM\|\om(0)\|_{1}}{\alpha^{3}(1+t-s)(1+M+t-s)s}\\ \nonumber\leq& \frac{CM\|\om(0)\|_{2}}{\alpha^{5}s^{2}}+
\frac{CM\|\om(0)\|_{1}}{\alpha^{3}(1+t-s)(1+M+t-s)s}.
\end{align}

Thanks to $\psi_{3,1}=0,\ \psi=0$ at $y=0,1,$ we get by  \eqref{psi31} and \eqref{psiL2H1} that
\begin{align*}
\|\psi_{3,1}\|_2=&\|(\partial^2_y-\alpha^2)\psi_{3,1}\|_{L^2}=
\|-2(\psi'(1/u')')'+\psi'(1/u')''\|_{L^2}\\
\leq & C(\|\psi''\|_{L^2}+\|\psi'\|_{L^2})\leq C\|(\partial^2_y-\alpha^2)\psi\|_{L^2}= C\|\om\|_{0}\leq C\|\om(0)\|_{0},
\end{align*}
which gives
\begin{align}\label{psi43}
&\|\psi_{4,3}(s)\|_2=\|i\alpha(u''-\beta)\psi_{3,1}(s)\|_2\leq C\alpha\|\psi_{3,1}(s)\|_2\leq C\alpha\|\om(0)\|_{0}.
\end{align}
Since $\partial_y\psi_{4,1}=-i\alpha(u'''/u')'\psi$ and $\psi_{4,1}(t,0)=0$, we get
\begin{align*}
\|\psi_{4,1}\|_{L^2}\leq\|\partial_y\psi_{4,1}\|_{L^2}=\|-i\alpha(u'''/u')'\psi\|_{L^2}\leq C\alpha\|\psi\|_{L^2}.
\end{align*}
Thanks to $\partial_y\psi_{4,2}=-i\alpha(u'''/u')\psi'$ and $\psi_{4,1}+\psi_{4,2}=-i\alpha(u'''/u')\psi$, we get
\begin{align*}
&\|\psi_{4,2}\|_{L^2}\leq\|\psi_{4,1}\|_{L^2}+\|\psi_{4,1}+\psi_{4,2}\|_{L^2}\leq C\alpha\|\psi\|_{L^2}+\|\alpha(u'''/u')\psi\|_{L^2}\leq C\alpha\|\psi\|_{L^2},\\&\alpha\|\partial_y\psi_{4,2}\|_{L^2}+\|\partial_y^2\psi_{4,2}\|_{L^2}\leq C\|\partial_y\psi_{4,2}\|_{1}= C\|\alpha(u'''/u')\psi'\|_{1}\leq C\alpha\|\psi'\|_{1}\leq C\alpha\|\psi\|_{2}.
\end{align*}
Summing up, we conclude that
\begin{align*}
&\|\psi_{4,1}\|_{1}\leq C\alpha^2\|\psi\|_{L^2},\ \|\psi_{4,2}\|_{2}\leq C\alpha^3\|\psi\|_{L^2}+C\alpha\|\psi\|_{2}\leq C\alpha\|\psi\|_{2},
\end{align*}
which together with \eqref{psiL2} and \eqref{psiL2H1} gives
\begin{align}\label{psi41}
&\|\psi_{4,1}(s)\|_1\leq C\alpha^2\|\psi(s)\|_{L^2}\leq CM\alpha^{-2}(1+s)^{-1}(1+M+s)^{-1}\|\om(0)\|_{2},\\ \label{psi42}&\|\psi_{4,2}(s)\|_{2}\leq C\alpha\|\psi(s)\|_{2}=C\alpha\|\om(s)\|_{0}\leq C\alpha\|\om(0)\|_{0}.
\end{align}

It follows from \eqref{H-1H1}, \eqref{psi44}, \eqref{psi43}, \eqref{psi41} and \eqref{psi42} that for $0<s<t<T$,
\begin{align}\label{psi4s0}
\|e^{-i(t-s)\alpha\mathcal{R}'_{\alpha,\beta}}\psi_4(s)\|_{-2}\leq& \sum_{j=1}^4\|e^{-i(t-s)\alpha\mathcal{R}'_{\alpha,\beta}}\psi_{4,j}(s)\|_{-2}\\ \nonumber
\leq & \alpha^{-1}\|e^{-i(t-s)\alpha\mathcal{R}'_{\alpha,\beta}}\psi_{4,1}(s)\|_{-1}\\\nonumber
&+M\alpha^{-4}(t-s)^{-2}(\|\psi_{4,2}(s)\|_{2}+\|\psi_{4,3}(s)\|_{2}+\|\psi_{4,4}(s)\|_{2})\\ \nonumber\leq& C\alpha^{-3}(t-s)^{-1}\|\psi_{4,1}(s)\|_{1}+
CM\alpha^{-4}(t-s)^{-2}\alpha\|\om(0)\|_{1}\\ \nonumber\leq& \frac{CM\|\om(0)\|_{2}}{\alpha^{5}(t-s)(1+s)(1+M+s)}+
\frac{CM\|\om(0)\|_{1}}{\alpha^{3}(t-s)^{2}}.
\end{align}

Then we infer from \eqref{om1}, \eqref{om1-2}, \eqref{psi4-2}, \eqref{psi4st} and \eqref{psi4s0} that
\begin{align*}
\|\om_1(t)\|_{-2}\leq&\|e^{-it\alpha\mathcal{R}'_{\alpha,\beta}}\om_1(0)\|_{-2}+
\int_0^{\frac{t}{M+1}}\|e^{-i(t-s)\alpha\mathcal{R}'_{\alpha,\beta}}\psi_4(s)\|_{-2}ds\\ &+\int_{\frac{t}{M+1}}^{\frac{Mt}{M+1}}\|e^{-i(t-s)\alpha\mathcal{R}'_{\alpha,\beta}}\psi_4(s)\|_{-2}ds
+\int_{\frac{Mt}{M+1}}^{t}\|e^{-i(t-s)\alpha\mathcal{R}'_{\alpha,\beta}}\psi_4(s)\|_{-2}ds\\
\leq& C\alpha^{-3}t^{-1}\|\om(0)\|_{2}+
\int_0^{\frac{t}{M+1}}\left(\frac{CM\|\om(0)\|_{2}}{\alpha^{5}(t-s)(1+s)(1+M+s)}+
\frac{CM\|\om(0)\|_{1}}{\alpha^{3}(t-s)^{2}}\right)ds\\&
+\int_{\frac{t}{M+1}}^{\frac{Mt}{M+1}}C\alpha^{-3}(t-s)^{-1}
s^{-1}\|\om(0)\|_{1}ds
\\&+\int_{\frac{Mt}{M+1}}^{t}\left(\frac{CM\|\om(0)\|_{2}}{\alpha^{5}s^{2}}+
\frac{CM\|\om(0)\|_{1}}{\alpha^{3}(1+t-s)(1+M+t-s)s}\right)ds
\\ \leq &C\alpha^{-3}t^{-1}\|\om(0)\|_{2}+
\int_0^{\frac{t}{M+1}}\left(\frac{CM\|\om(0)\|_{2}}{\alpha^{5}t(1+s)(1+M+s)}+
\frac{CM\|\om(0)\|_{1}}{\alpha^{3}t^{2}}\right)ds
\\&+\int_{\frac{t}{M+1}}^{\frac{t}{2}}C\alpha^{-3}t^{-1}
s^{-1}\|\om(0)\|_{1}ds+\int_{\frac{t}{2}}^{\frac{Mt}{M+1}}C\alpha^{-3}(t-s)^{-1}
t^{-1}\|\om(0)\|_{1}ds
\\&+\int_{\frac{Mt}{M+1}}^{t}\left(\frac{CM\|\om(0)\|_{2}}{\alpha^{5}t^{2}}+
\frac{CM\|\om(0)\|_{1}}{\alpha^{3}(1+t-s)(1+M+t-s)t}\right)ds
\\ \leq& C\alpha^{-3}t^{-1}\|\om(0)\|_{2}+
C\alpha^{-5}t^{-1}\|\om(0)\|_{2}\ln(M+1)+
\frac{CM\|\om(0)\|_{1}}{\alpha^{3}t^{2}}\frac{t}{M+1}
\\&+2C\alpha^{-3}t^{-1}
\|\om(0)\|_{1}\ln\frac{M+1}{2}
+\frac{CM\|\om(0)\|_{2}}{\alpha^{5}t^{2}}\frac{t}{M+1}+
\frac{C\|\om(0)\|_{1}\ln(M+1)}{\alpha^{3}t}
\\ \leq& C\alpha^{-3}t^{-1}\|\om(0)\|_{2}(1+
\ln(M+1)).
\end{align*}
Here we used the facts that $\|\om(0)\|_{1}\leq C\alpha^{-1}\|\om(0)\|_{2}\leq C\|\om(0)\|_{2} $ and
\begin{align*}
\int_0^{+\infty}\frac{M}{(1+s)(1+M+s)}ds=\ln\frac{1+s}{1+M+s}\Big|_0^{+\infty}=\ln(M+1).
\end{align*}
This completes the proof of the lemma.
\end{proof}

\subsection{Proof of Theorem \ref{thm:monotone}}
Here we only need the following slightly weak results in Lemmas \ref{lem5.3} and \ref{lem5.3-2}{(the case $\alpha<0$ or $t<0$ can be proved by taking conjugation)}:
\begin{align}
\label{v1}&|\alpha|(\|\partial_y\psi(t)\|_{L^2}+|\alpha|\|\psi(t)\|_{L^2})\leq C\langle t\rangle^{-1} \|\om(0)\|_{H^1},\ |\alpha|^2\|\psi(t)\|_{L^2}\leq C\langle t\rangle^{-2} \|\om(0)\|_{H^2}.
\end{align}

Thanks to $\vec{v}=\nabla^{\perp}\psi=(\psi_{y},-\psi_{x}),$ we get by \eqref{v1} that 
\begin{align*}
\|\vec{v}(t)\|_{L^2_{x,y}}^2&=C\sum_{\al\neq0}(\al^2\|\widehat{\psi}(t,\al,\cdot)\|_{L^2_y}^2
+\|\partial_y\widehat{\psi}(t,\al,\cdot)\|_{L^2_y}^2)\\
&\leq C\sum_{\al\neq0}|\al|^{-2}\langle t\rangle^{-2}\|\widehat{\om}_0(\al,\cdot)\|_{H^1_y}^2\leq C\langle t\rangle^{-2}\|{\om}_0\|_{H_x^{-1}H^1_y}^2,
\end{align*}
and
\begin{align*}
\|v_2(t)\|_{L_{x,y}^2}^2&=C\sum_{\al\neq0}\al^2\|\widehat{\psi}(t,\al,\cdot)\|_{L^2_y}^2\leq C\sum_{\al\neq0}\frac{\|\widehat{\om}_0(\al,\cdot)\|_{H^2_y}^2}{|\al|^{2}\langle t\rangle^{4}}\leq C\frac{\|{\om}_0\|_{H_x^{-1}H^2_y}^2}{\langle t\rangle^{4}}.
\end{align*}
This shows that
\beno
\|\vec{v}(t)\|_{L^2_{x,y}}\le C\langle t\rangle^{-1}\|\om_0\|_{H^{-1}_xH^1_y},\quad  \ \|v_2(t)\|_{L^2_{x,y}}\le  C\langle t\rangle^{-2}\|\om_0\|_{H^{-1}_xH^2_y}.
\eeno

The proof of the scattering part is the same as the case of $ \beta=0$ in Section 10.2 of \cite{WZZ1}. Here we omit the details.

\section{The limiting absorption principle}

\subsection{Compactness results for Rayleigh-Kuo equation}

The limiting absorption principle is based on the contradiction argument, blow-up analysis and compactness.
To this end, we first study the compactness of the solution sequence of the Rayleigh-Kuo equation.
In this subsection, we always assume that the flow $u(y)$ satisfies \textbf{(H1)}, $\al>0$ and $\beta\in \mathbb{R}$.
We denote by $c^i=\textrm{Im}(c)$ and $c^r=\textrm{Re}(c)$ for $c\in\mathbb{C}$ in the sequel.\smallskip

The following two lemmas deal with the compactness in the domain without critical points.

\begin{lemma}\label{noncritical point}
Let $c\in \text{Ran }(u)$, $[a,b]\cap u^{-1}\{c\}=\{y_0\}$ and $u'(y_0)u'(y) >0$ on $[a,b]$.
Assume that $\omega_n, \phi_n\in H^1(a,b), u_n\in H^3(a,b)$ and  $c_n\in \mathbb{C}$  such  that  $\omega_n\to\omega$, $\phi_n\rightharpoonup\phi$ in  $ H^1(a,b)$,
$u_n\to u$ in  $ H^3(a,b)$, $c_n^i>0$, $c_n\to c$ and
\begin{align*}
(u_n-c_n)(\phi_n''-\alpha^2\phi_n)-(u_n''-\beta)\phi_n=\omega_n
\end{align*}
on $[a,b]$. Then $\phi_n\to\phi$ in $H^1(a,b)$.
\end{lemma}

\begin{lemma}\label{noncritical-point-formula}
Under the assumption of Lemma \ref{noncritical point}, we  have for any $\varphi\in H_0^1(a,b)$,
$$\int_{a}^{b}(\phi'\varphi'+\alpha^2\phi\varphi)dy+p.v.\int_{a}^{b}{((u''-\beta)\phi+\omega)
\varphi\over u-c}dy+i\pi{((u''-\beta)\phi+\omega)\varphi(y_0)\over |u'(y_0)|}=0.$$
\end{lemma}

The proof of Lemmas \ref{noncritical point} and \ref{noncritical-point-formula} is similar to Lemma 6.2 in \cite{WZZ2} with $g_n$ replaced by $(u_n''-\beta)\phi_n+\omega_n$. Here we omit the details. \smallskip

Next we study the compactness in the domain with critical points satisfying $u''-\beta\neq 0$. First of all, we study the behaviour of
the solution at critical points.

\begin{lemma}\label{critical point estimate}
Let $c\notin\mathbb{R}$. Assume that $y_0\in (u')^{-1}\{0\}\cap\omega^{-1}\{0\}$, let $[a,b]$ be an interval so that $\phi,\omega\in H^1(a,b)$, $y_0\in[a,b]\subset[y_1,y_2]$,
\beno
|u(y_0)-c|<\min\{1,\max\{|y_0-a|^2,|y_0-b|^2\}\},\quad
(\beta-u''(y))(\beta-u''(y_0))>0,
\eeno
and $(u-c)(\phi''-\alpha^2\phi)-(u''-\beta)\phi=\omega$ on $[a,b]$. Then we have
$$
|\phi(y_0)|\leq C |u(y_0)-c|^{1\over4}\big(\|\phi\|_{H^1(a,b)}+\|\omega\|_{H^1(a,b)}\big),
$$
where $C$ depends on $\max\{|y_0-a|,|y_0-b|\}$, $\alpha,$ $\beta$ and $u$.
\end{lemma}

\begin{proof}
Without loss of generality, we assume $|y_0-b|\geq|y_0-a|$. Let
$$u_1(y)=u(y)-{\beta(y-y_0)^2\over2}  \;\text{on}\;[a,b].$$
 Note that there exists $c_0>0$ such that $|\beta-u''|>c_0>0$ on $[a,b]$. {We normalize}  $\phi, \om$ so that $\|\phi\|_{H^1(a,b)}+\|\omega\|_{H^1(a,b)}=1. $ Direct computations show that for any $y,\tilde y\in[y_0,b]$,
\begin{align}
&|u_1'(y)|=|u'(y)-u'(y_0)-\beta(y-y_0)|\leq C|y-y_0|,\label{eq:u1-est1}\\
&|u_1'(y)-u_1'(\tilde y)|=|u'(y)-u'(\tilde y)-\beta(y-\tilde y)|\geq c_0|y-\tilde y|.\label{eq:u1-est2}
\end{align}

Let $\delta=|u(y_0)-c|^{1\over2}$. Then for any $0<y-y_0<\delta$, due to $H^1(a,b)\hookrightarrow C^{0,{1\over2}}(a,b)$, we have
\begin{align*}
|\omega(y)|\leq C\delta^{1\over2},\;|u(y)-c|\leq |u(y)-u(y_0)|+|u(y_0)-c|\leq C\delta^2.
\end{align*}
Let
\begin{align*}
g=((u-c)\phi'-u_1'\phi)'=\alpha^2(u-c)\phi+\beta(y-y_0)\phi'+\omega.
\end{align*}
Thus, for $0<y-y_0<\delta$,
$$\|g\|_{L^1(y_0,y_0+\delta)}\leq C\delta^3+C\left(\int_{y_0}^{y_0+\delta}(y-y_0)^2dy\int_{y_0}^{y_0+\delta}|\phi'|^2dy\right)^{1\over2}+C\delta^{3\over2}\leq C\delta^{3\over2}.
$$
 Choose $z_1\in(y_0,y_0+\delta/3)$ and $z_2\in(y_0+2\delta/3,y_0+\delta)$ so that
$|\phi'(z_1)|^2+|\phi'(z_2)|^2\leq 6\delta^{-1}\|\phi'\|^2_{L^2(y_0,y_0+\delta)}$. Otherwise, if ${\delta\over3}|\phi'(y)|^2>\|\phi'\|_{L^2(y_0,y_0+{\delta\over3})}^2$ for all $y\in(y_0,y_0+{\delta\over3})$, then ${\delta\over3}\|\phi'\|_{L^2(y_0,y_0+{\delta\over3})}^2>{\delta\over3}\|\phi'\|_{L^2(y_0,y_0+{\delta\over3})}^2$, which is a contradiction. Using the facts that
\begin{align*}
&\big|(u-c)\phi'|_{z_1}^{z_2}\big|\leq (|\phi'(z_1)|+|\phi'(z_2)|)\|u-c\|_{L^\infty(z_1,z_2)}\leq C\delta^{-{1\over2}}\delta^2=C\delta^{3\over2},\\
&\left|((u-c)\phi'-(u_1'\phi))\big|_{z_1}^{z_2}\right|=\left|\int_{z_1}^{z_2}g(y)dy\right|\leq C\delta^{3\over2},
\end{align*}
we infer that
$$|(u_1'\phi)|_{z_1}^{z_2}|\leq C\delta^{3\over2}.$$
Notice that $(u_1'\phi)|_{z_1}^{z_2}=\phi(y_0)u_1'|_{z_1}^{z_2}+u_1'(z_1)\phi|_{z_1}^{y_0}+u_1'(z_2)\phi|_{y_0}^{z_2}.$
We get by \eqref{eq:u1-est1} and \eqref{eq:u1-est2} that
\begin{align*}
{1\over3} \delta c_0|\phi(y_0)|\leq |\phi(y_0)u_1'|_{z_1}^{z_2}|&\leq |(u_1'\phi)|_{z_1}^{z_2}|+2\|u_1'\|_{L^\infty(y_0,z_2)}\int_{y_0}^{z_2}|\phi'(z)|dz\\
&\leq C\delta^{3\over2}+C\delta^{3\over2}\|\phi'\|_{L^2(a,b)}\leq C\delta^{3\over2}.
\end{align*}
This shows that
$|\phi(y_0)|\leq C |u(y_0)-c|^{1\over4}.$
\end{proof}

\begin{lemma}\label{critical-point-not-beta-u-sec-dao-equal0}
Let $c\in \text{Ran }(u)$, $y_0\in u^{-1}\{c\}\cap (y_1,y_2)$, $u'(y_0)=0$,  and $\delta>0$ so that $(u''(y_0)-\beta)(u''(y)-\beta)>0$ on $[y_0-\delta,y_0+\delta]\subset[y_1,y_2]$ and $[y_0-\delta,y_0+\delta]\cap u^{-1}\{c\}=\{y_0\}$. Assume that $\phi_n,\omega_n\in H^1(y_0-\delta,y_0+\delta)$ and $c_n\in\mathbb{C}$ so that $c_n^i>0$, $c_n\to c$, $\phi_n\rightharpoonup0,\omega_n\to0 $ in $H^1(y_0-\delta,y_0+\delta)$ and
\begin{align}\label{equation-for-noncritical point-for-u}
(u-c_n)(\phi_n''-\alpha^2\phi_n)-(u''-\beta)\phi_n=\omega_n
\end{align}
 holds on $[y_0-\delta,y_0+\delta]$.
Then $\phi_n\to0$ in $H^1(y_0-\delta,y_0+\delta)$.
\end{lemma}

\begin{proof}
Let $c=0$, $y_0=0$ and $u''(0)-\beta=2$ for convenience. Otherwise, we can consider $\hat u(y)= k(u(y+y_0)-u(y_0)), \hat \beta= k\beta, \hat \phi_n(y)={\phi}_n(y+y_0), \hat \omega_n(y)=k{\omega}_n(y+y_0)$ and $\hat c_n=k(c_n-u(y_0))$ with $k=2/(u''(y_0)-\beta)$, and the equation
\begin{align*}
(\hat u-\hat c_n)(\hat\phi_n''-\alpha^2\hat\phi_n)-({\hat u''-\hat\beta})\hat\phi_n=\hat\omega_n\;\;\text{on}\;[y_1-y_0,y_2-y_0].
\end{align*}

It suffices to show that $\phi_n\to0$ in $H^1(-\delta,\delta)$ in the case when  $\omega_n(0)=0$ for  $n\geq1$.
Indeed, we consider
\begin{align*}
\phi_{n*}(y)&=\phi_n(y)+{\omega_n(0)\over 2}\cosh (\alpha y),\\
\omega_{n*}(y)&=\omega_{n}(y)-(u''(y)-\beta){\omega_n(0)\over2}\cosh (\alpha y).
\end{align*}
Then it is easy to see that $\phi_{n*},\omega_{n*}\in H^1(-\delta,\delta)$, $\omega_{n*}(0)=0$ and
\begin{align*}(u-c_n)(\phi_{n*}''&-\alpha^2\phi_{n*})-(u''-\beta)\phi_{n*}=\omega_{n*},\\
\|\omega_{n*}\|_{H^1(-\delta,\delta)}&\leq \|\omega_n\|_{H^1(-\delta,\delta)}+C|\omega_n(0)|\leq C\|\omega_{n}\|_{H^1(-\delta,\delta)}\to0.
\end{align*}
Since $\|\phi_n-\phi_{n*}\|_{H^1(-\delta,\delta)}\leq C|\omega_n(0)|\leq C\|\omega_n\|_{H^1(-\delta,\delta)}\to0$ {and $\phi_{n}\rightharpoonup0$} in $H^1(-\delta,\delta)$, we get $\phi_{n*}\rightharpoonup0$ in $H^1(-\delta,\delta)$. Then we have $\|\phi_{n*}\|_{H^1(-\delta,\delta)}\to0$ and $\|\phi_{n}\|_{H^1(-\delta,\delta)}\leq\|\phi_{n}-\phi_{n*}\|_{H^1(-\delta,\delta)}+\|\phi_{n*}\|_{H^1(-\delta,\delta)}\to0$.\smallskip

So, we may assume that $\omega_n(0)=0$ for  $n\geq1$ in the sequel.
Let $c_n=r_n^2e^{2i\theta_n}$ with $\theta_n\in(0,{\pi\over2})$  for  $n\geq1$. Then $r_n\to0^+$.
By Lemma \ref{critical point estimate}, $|\phi_n(0)|\leq Cr_n^{1\over2}.$ We denote
\begin{align}\label{tilde-phi-n-tilde-omega-n-u-n}
\tilde \phi_n(y)=r_n^{-{1\over2}}\phi_n(r_ny),\;\tilde \omega_n(y)=r_n^{-{1\over2}}\omega_n(r_ny),\;u_n(y)=r_n^{-2}u(r_ny).
\end{align}
Then we find
\begin{align}\label{equation-after-scaling}
(u_n-e^{2i\theta_n})\tilde \phi_n''-(u_n''-\beta)\tilde \phi_n=\tilde \omega_n+(u_n-e^{2i\theta_n})(\alpha r_n)^2 \tilde \phi_n
\end{align}
on $[-{\delta/ r_n},{\delta/ r_n}]$ and
\begin{align}\label{first-estimate}
|\tilde \phi_n(0)|\leq C, \quad \|\tilde \phi'_n\|_{L^2(-\delta/r_n,\delta/r_n)}=\|\phi_n'\|_{L^2(-\delta,\delta)}\leq C,\\\nonumber
\tilde \omega_n(0)=0,\quad\|\tilde \omega_n'\|_{L^2(-\delta/r_n,\delta/r_n)}=\| \omega_n'\|_{L^2(-\delta,\delta)}\to0.
\end{align}
This implies that $\tilde \phi_n$ is uniformly bounded in $H^1_{loc}(\mathbb{R})$ and $\tilde \omega_n\to 0$ in $H^1_{loc}(\mathbb{R}).$
Up to a subsequence, we may assume that $\tilde \phi_n\rightharpoonup\tilde \phi_0$ in $H^1_{loc}(\mathbb{R})$, and $\theta_n\to\theta_0\in[0,{\pi\over2}]$. Then by (\ref{first-estimate}), we have $\tilde \phi'_0\in L^2(\mathbb{R}).$ Using the facts that
\begin{align*}
u_n(y)&=y^2\int_0^1\int_0^1tu''(r_nyts)dsdt\to{\beta+2\over2}y^2,\;&u_n'(y)&=y\int_0^1u''(r_nyt)dt\to(\beta+2)y,\\
u_n''(y)&=u''(r_ny)\to\beta+2,\;&u_n'''(y)&=r_nu'''(r_ny)\to0,
\end{align*}
in $L^2_{loc}(\mathbb{R})$, we infer that
\begin{align*}
u_n\to{\beta+2\over 2}y^2 \quad {\rm in} \,\, H^3_{loc}(\mathbb{R}).
\end{align*}

Next, we show that  $\tilde \phi_n\to\tilde \phi_0$ in $H^1_{loc}(\mathbb{R})$ and $\tilde \phi_0\equiv0$ on $\mathbb{R}$.
The proof is very complicated and is split into five cases in terms of $\theta_0$ and $\beta$. \smallskip

{\bf Case 1.} $\theta_0\in(0,{\pi\over2})$.

In this case, $\tilde \phi_n$  is uniformly bounded in $H^2_{loc}(\mathbb{R})$. So, $\tilde \phi_n\to \tilde {\phi}_0$ in $C^1_{loc}(\mathbb{R})$. Moreover,
$$\Big({\beta+2\over2} y^2-e^{2i\theta_0}\Big)\tilde \phi_0''=2\tilde \phi_0\quad \text{on}\,\,\mathbb{R}.$$

 For fixed  $0<\gamma<1$, let $\eta_R\in C_0^{\infty}(-R,R)$, $R>0$ be a cut-off function satisfying
\begin{itemize}\vspace{-0.2cm}
\item[{\rm (i)}] $0\leq \eta_R(y)\leq 1$ for $y\in[-R, R]$ and $\eta\equiv1$ on $[-\gamma R,\gamma R]$,
 \item[{\rm (ii)}]$|\eta_R'(y)|\leq {2\over (1-\gamma) R}$,  $y\in [-R,R]$.
\end{itemize}
We get by integration by parts that
\begin{align}\label{integration by parts after truncation}
\int_{-R}^R\tilde \phi_0''\eta_R\bar{\tilde \phi}_0dy=-\int_{-R}^R \tilde \phi_0'\eta_R'\bar{\tilde \phi}_0+|\tilde \phi'_0|^2\eta_Rdy.
\end{align}
By Hardy's inequality, we have $\|{\tilde \phi_0\over y}\|_{L^2(1,+\infty)}\leq C(\|\tilde \phi_0'\|_{L^2(\mathbb{R})}+|\tilde \phi_0(0)|)<+\infty $, which gives
\begin{align}\label{yuxiang}
\left|\int_{0}^R \tilde \phi_0'\eta_R'\bar{\tilde \phi}_0dy\right|&\leq
\int_{\gamma R}^R|\tilde \phi_0'|{2\over (1-\gamma) R}|{\tilde \phi}_0|dy\\\nonumber
&\leq \|\tilde \phi_0'\|_{L^2({\gamma R},R)}\left(\int_{\gamma R}^R|{2\over (1-\gamma) R}{\tilde \phi}_0|^2dy\right)^{1\over2}\\ \nonumber&\leq C\|\tilde \phi_0'\|_{L^2({\gamma R},R)} \|{\tilde \phi_0\over y}\|_{L^2(1,+\infty)} \to 0
\end{align}
as $R\to +\infty.$ Similarly, $\int_{-R}^{0} \tilde\phi_0'\eta_R'\bar{\tilde \phi}_0dy\to0$ as $R\to+\infty$.
Thus by (\ref{integration by parts after truncation}), we get
\begin{align*}
\int_{\mathbb{R}} |\tilde \phi'_0|^2dy=-\int_{\mathbb{R}}\tilde \phi_0''\bar{\tilde \phi}_0dy&
=-\int_{\mathbb{R}}{2|\tilde \phi_0|^2\over({\beta+2\over2})y^2-e^{2i\theta_0}}dy\\\nonumber
&=-\int_{\mathbb{R}}{2|\tilde\phi_0|^2\left({\beta+2\over2}y^2-\cos 2\theta_0+i\sin 2\theta_0\right)\over\left({\beta+2\over2}y^2-\cos2\theta_0\right)^2+\sin^22\theta_0}dy.\end{align*}
Taking the imaginary part of the equality, we deduce that
$\tilde \phi_0\equiv0$ on $\mathbb{R}$.\smallskip

{\bf Case 2.} $\theta_0=0$ and ${\beta+2\over2}>0$.

Let $a=\sqrt{2\over \beta+2}$. We first claim that for any $\varphi\in H^1(\mathbb{R})$ with compact support,
\begin{align}\label{u10}
\int_{\mathbb{R}}\tilde{\phi}_0'\varphi'dy+p.v.\int_{\mathbb{R}}\frac{2\tilde{\phi}_0\varphi}{{\beta+2\over2}y^2-1}dy
+i\pi\sum_{y=\pm a}\frac{{2}(\tilde{\phi}_0\varphi)(y)}{\sqrt{2(\beta+2)}}=0.
\end{align}
Indeed, since $\tilde \phi_n$  is uniformly bounded in $H^2_{loc}(\mathbb{R}\setminus \{\pm a\})$, thus  $\tilde \phi_n\to \tilde {\phi}_0$ in $C_{loc}^1(\mathbb{R}\setminus \{\pm a\})$, and
$$\Big({\beta+2\over2} y^2-1\Big)\tilde \phi_0''=2\tilde \phi_0\ \ \text{on}\ \ \mathbb{R}\setminus\{\pm a\},$$
which implies \eqref{u10} holds for any $\varphi\in H^1(\mathbb{R})$ with compact support and $\{\pm a\}\cap\text{supp}\ \varphi=\emptyset $. Lemma \ref{noncritical-point-formula} ensures that \eqref{u10} holds for any $\varphi\in H^1(\mathbb{R})$ with $\text{supp}\ \varphi\subset[\pm a-\varepsilon,\pm a+\varepsilon]$, where $\varepsilon\in(0,a)$. Therefore, \eqref{u10} holds for any $\varphi\in H^1(\mathbb{R})$ with compact support.

Now by \eqref{u10}, we have for $R>a$,
\begin{align}\label{phi1}
-\int_{-R}^R (|\tilde\phi_0'|^2\eta_R+\tilde \phi_0'\bar{\tilde \phi}_0\eta_R')dy&=
-\int_{-R}^R\tilde \phi_0'(\bar{\tilde \phi}_0\eta_R)'dy\\\nonumber
&=p.v.\int_{-R}^R{2|\tilde  \phi_0|^2\eta_R\over {\beta+2\over2}y^2-1}dy +i\pi\sum\limits_{y=\pm a}{{2}(|\tilde \phi_0|^2\eta_R)(y)\over\sqrt{2(\beta+2)}}.
\end{align}
Letting $R\to +\infty$ in (\ref{phi1}) and by (\ref{yuxiang}), we get
\begin{align*}
-\int_{\mathbb{R}}|\tilde \phi_0'|^2dy
=p.v.\int_{\mathbb{R}}{2|\tilde  \phi_0|^2\over {\beta+2\over2}y^2-1}dy +i\pi\sum\limits_{y=\pm a}{2|\tilde \phi_0|^2(y)\over\sqrt{2(\beta+2)}}.
\end{align*}
This shows that
\begin{align}\label{phi zero}
\tilde \phi_0(\pm a)=0,
\end{align}
which, together with the Sobolev embedding $H^1(J)\hookrightarrow C^{0,{1\over2}}(J)$ and \eqref{u10}, implies that
\begin{align*}
\left|\int_{\mathbb{R}}\tilde{\phi}_0'\varphi'dy\right|=\left|\int_{\mathbb{R}}\frac{2\tilde{\phi}_0\varphi}
{{\beta+2\over2}y^2-1}dy\right|\leq\left(\int_{\pm a-\varepsilon}^{\pm a+\varepsilon}\left|{2\tilde \phi_0\over {\beta+2\over2}y^2-1}\right|^pdy\right)^{\frac{1}{p}}\|\varphi\|_{L^{p'}}\leq C\|\varphi\|_{L^{p'}},
\end{align*}
for every $ \varphi\in H^1(\mathbb{R})$ with $\text{supp} \ \varphi\subset[\pm a-\varepsilon,\pm a+\varepsilon],$ where $1<p<2,\ 1/p+1/p'=1$,  $ \varepsilon\in(0,a)$, and $J$ is a compact interval. Thus, $\tilde \phi_0\in W^{2,p}(\pm a-\varepsilon,\pm a+\varepsilon),$ and by the Sobolev embedding $W^{2,p}(J)\hookrightarrow C^1(J)$, we have
\begin{align}\label{phi3}
\tilde \phi_0\in C^1(\mathbb{R}).
\end{align}
Noting that
\begin{align*}
-\int_{a}^R(\bar{\tilde\phi}_0'\tilde \phi_0'\eta_R+\bar{\tilde\phi}_0\tilde \phi_0'\eta_R')dy
=-\int_{a}^R(\bar{\tilde\phi}_0\eta_R)'\tilde \phi_0'dy=\int_{a}^R\tilde \phi_0''\bar{\tilde\phi}_0\eta_Rdy, \end{align*}
and similar to (\ref{yuxiang}), we have
$$\left|\int_{a}^R\bar{\tilde\phi}_0\tilde \phi_0'\eta_R'dy\right|\leq C\|\tilde \phi_0'\|_{L^2({\gamma R},R)}\|{\tilde \phi_0\over y}\|_{L^2(1,+\infty)} $$
for $R>0$ sufficiently large.
Letting $R\to+\infty$, we get
\begin{align*}
-\int_{a}^\infty|\tilde \phi_0'|^2dy=\int_{a}^\infty\tilde \phi_0''\bar{\tilde \phi}_0dy
=\int_{a}^\infty{2|\tilde\phi|^2\over{{\beta+2\over2}y^2-1}}dy.
\end{align*}
This shows that $\tilde \phi_0\equiv0$ on $[a,+\infty).$ By (\ref{phi zero}), (\ref{phi3}) and Lemma 2.2 in \cite{LYZ}, we have
 $\tilde\phi_0\equiv0$ on $\mathbb{R}$.

Furthermore, it follows from Lemma \ref{noncritical point} that $\tilde \phi_n\to \tilde \phi_0$ in $H^1\big(\pm a-\delta,\pm a+\delta\big)$ for $\delta>0$ sufficiently small, and thus $\tilde \phi_n\to \tilde {\phi}_0$ in $H^1_{loc}(\mathbb{R})\cap C_{loc}^1\big(\mathbb{R}\setminus \{\pm a\}\big)$.
 \smallskip

{\bf  Case 3.} $\theta_0=0$ and ${\beta+2\over2}<0$.

  Similar to Case 1, we have $\tilde \phi_n\to \tilde {\phi}_0$ in $C^1_{loc}(\mathbb{R})$ and
$$\Big({\beta+2\over2} y^2-1\Big)\tilde \phi_0''=2\tilde \phi_0 \quad\text{on}\; \mathbb{R}.$$

Let $V(y)={-2\over {\beta+2\over2}y^2-1}$. Then $V(y)>0$ for $y\in \mathbb{R}$ and
\begin{align}\label{equation-with-substitution}
(|\tilde \phi_0'|^2+V|\tilde \phi_0|^2)'=V'|\tilde\phi_0|^2.
\end{align}
Multiplying both sides of (\ref{equation-with-substitution}) by $y\eta_R$ and integrating it from $-R$ to $R$, we get
\begin{align}\label{integral-with-y-times-the-equation}
-\int_{-R}^RyV'|\tilde \phi_0|^2\eta_Rdy&=-\int_{-R}^Ry(|\tilde \phi_0'|^2+V|\tilde \phi_0|^2)'\eta_Rdy\\\nonumber
&=\int_{-R}^R(|\tilde \phi_0'|^2+V|\tilde \phi_0|^2)(\eta_R+y\eta_R')dy.\end{align}
Note that
\begin{align}\label{interationwithR}
\int_{-R}^RV|\tilde \phi_0|^2\eta_Rdy=-\int_{-R}^R\tilde \phi_0''\bar{\tilde \phi}_0\eta_Rdy=\int_{-R}^R\left(\tilde \phi_0'\eta_R'\bar{{\tilde\phi}}_0+|\tilde\phi_0'|^2\eta_R\right)dy.\end{align}
Thanks to $\int_{-R}^R \tilde \phi_0'\eta_R'\bar{\tilde \phi}_0dy\to 0 $
as $R\to +\infty$, we get by (\ref{interationwithR})  that
\begin{align}\label{limit-tunc-function}
\lim\limits_{R\to\infty} \int_{-R}^RV|\tilde\phi_0|^2dy=\int_{\mathbb{R}}V|\tilde\phi_0|^2dy=\|\tilde\phi_0'\|_{L^2(\mathbb{R})}^2<\infty.
\end{align}
This yields that
\begin{align}\label{integrationwithR2}
&\left(\int_{- R}^{-\gamma R}+\int_{\gamma R}^R\right)|(|\tilde \phi_0'|^2+V|\tilde \phi_0|^2)(\eta_R+y\eta_R')|dy\\\nonumber
&\leq C\|\tilde \phi_0'\|^2_{L^2((-R,-\gamma R)\cup(\gamma R, R))}+C\left(\int_{-R}^{-\gamma R}+\int_{\gamma R}^R\right)V|\tilde \phi_0|^2 dy\to0
\end{align}
as $R\to+\infty.$ Hence, (\ref{integral-with-y-times-the-equation}) and (\ref{limit-tunc-function})--(\ref{integrationwithR2}) imply
\begin{align*}
-\int_{\mathbb{R}}yV'|\tilde \phi_0|^2dy=\int_{\mathbb{R}}\left(|\tilde \phi_0'|^2+V|\tilde \phi_0|^2\right)dy=2\int_{\mathbb{R}}V|\tilde \phi_0|^2dy,
\end{align*}
that is,
\begin{align*}
\int_{\mathbb{R}}(2V+yV')|\tilde \phi_0|^2dy=0.
\end{align*}
A direct computation finds
\begin{align*}
2V(y)+yV'(y)={4\over\left({\beta+2\over2}y^2-1\right)^2}>0.
\end{align*}
This implies that $\tilde\phi_0\equiv0$ on $\mathbb{R}$.\smallskip

{\bf Case 4.} $\theta_0={\pi\over2}$ and ${\beta+2\over2}>0$.

 Similar to Case 1, we have $\tilde \phi_n\to \tilde {\phi}_0$ in $C^1_{loc}(\mathbb{R})$. Using the cut-off function $\eta_R$ and Hardy's inequality as above, we can also show that 
\begin{align*}
-\int_{\mathbb{R}}|\tilde\phi_0'|^2dy=-\lim\limits_{R\to\infty}\int_{-R}^R\bar{\tilde{\phi}}_0'(\tilde\phi_0'\eta_R+
\tilde\phi_0\eta_R')dy=\lim\limits_{R\to\infty}\int_{-R}^R{2|\tilde \phi_0|^2\eta_R\over {\beta+2\over 2} y^2+1}dy=\int_{\mathbb{R}}{2|\tilde \phi_0|^2\over {\beta+2\over 2} y^2+1}dy,
\end{align*}
which yields that $\tilde \phi_0\equiv0$ on $\mathbb{R}$.\smallskip

{\bf Case 5.} $\theta_0={\pi\over2}$ and ${\beta+2\over2}<0$.

 Similar to Case 2, we have  $\tilde \phi_n\to \tilde {\phi}_0$ in $H^1_{loc}(\mathbb{R})\cap C_{loc}^1(\mathbb{R}\setminus \{\pm\sqrt{{-2\over\beta+2}}\})$, and for any $\varphi\in H^1(\mathbb{R})$ with compact support,
\begin{align*}
\int_{\mathbb{R}}\tilde{\phi}_0'\varphi'dy+p.v.\int_{\mathbb{R}}\frac{2\tilde{\phi}_0\varphi}{{\beta+2\over2}y^2+1}dy
+i\pi\sum_{y=\pm\sqrt{-2\over \beta+2}}\frac{{2}(\tilde{\phi}_0\varphi)(y)}{\sqrt{-2(\beta+2)}}=0.
\end{align*}
Thus,
\begin{align}\label{phi1beta-piover2}
-\int_{-R}^R(|\tilde \phi_0'|^2\eta_R+\tilde \phi_0'\bar{\tilde \phi}_0\eta_R')dy&=
-\int_{-R}^R\tilde \phi_0'(\bar{\tilde \phi}_0\eta_R)'dy\\\nonumber
&=p.v.\int_{-R}^R{2|\tilde  \phi_0|^2\eta_R\over {\beta+2\over2}y^2+1}dy +i\pi\sum\limits_{y=\pm\sqrt{{-2\over \beta+2}}}{{2}(|\tilde \phi_0|^2\eta_R)(y)\over\sqrt{-2(\beta+2)}}.
\end{align}
Letting $R\to +\infty$ in (\ref{phi1beta-piover2}), we get
\begin{align*}
-\int_{\mathbb{R}}|\tilde \phi_0'|^2dy
=p.v.\int_{\mathbb{R}}{2|\tilde  \phi_0|^2\over {\beta+2\over2}y^2+1}dy +i\pi\sum\limits_{y=\pm\sqrt{{-2\over \beta+2}}}{2|\tilde \phi_0|^2(y)\over\sqrt{-2(\beta+2)}}.
\end{align*}
Then we have
\begin{align}\label{phi-zero-piover2}
\tilde \phi_0(\pm\sqrt{{-2\over\beta+2}})=0,
\end{align}
which, similar to Case 2, implies that
$\tilde \phi_0\in C^1(\mathbb{R}).$
By (\ref{phi-zero-piover2}), we have
\begin{align*}
-\int_{-\sqrt{{-2\over\beta+2}}}^{\sqrt{{-2\over\beta+2}}}|\tilde \phi_0'|^2dy
=\int_{-\sqrt{{-2\over\beta+2}}}^{\sqrt{{-2\over\beta+2}}}{2|\tilde \phi_0|^2\over \left({\beta+2\over2}\right)y^2+1}dy. \end{align*}
Since $\left({\beta+2\over2}\right)y^2+1>0$, we  have $\tilde \phi_0\equiv0$ on $\left[-\sqrt{{-2\over\beta+2}},\sqrt{{-2\over\beta+2}}\right]$.
 Then by   Lemma 2.2 in \cite{LYZ}, we have
 $\tilde\phi_0\equiv0$ on $\mathbb{R}$.\smallskip

In summary, we have shown that $\tilde \phi_n\to0$ in $H_{loc}^1(\mathbb{R})\cap C_{loc}^1(\mathbb{R}\setminus E)$, where
 \begin{align*}
		E=\left\{\begin{array}{ll}
		\left\{\pm\sqrt{2\over \beta+2}\right\},\;\;\textrm{if}\;\; \theta_0=0,\beta>-2,\\
		\left\{\pm\sqrt{-2\over \beta+2}\right\},\;\;\textrm{if} \;\;\theta_0={\pi\over2},\beta<-2,\\
\emptyset,\;\;\;\;\;\;\;\;\;\;\;\;\;\;\;\;\;\;\textrm{otherwise}.
		\end{array}\right.
	\end{align*}
Then we have
\begin{align}\label{convergence-near-0}
\|\phi_n'\|_{L^2(-{b\over \sqrt{|u''(0)|}}r_n,{b\over \sqrt{|u''(0)|}}r_n)}=\|\tilde \phi_n'\|_{L^2(-{b\over \sqrt{|u''(0)|}},{b\over \sqrt{|u''(0)|}})}\to0\end{align}
for any fixed $b>0$.

Thanks to $u''(0)\neq0$, choose  $\delta_1\in(0,\delta)$ such that $|u''(y)|> {|u''(0)|\over2}$ and $u''(y)-\beta>1$ on $[-\delta_1,\delta_1]$. Then there exists $C>1$ such that for any $y\in[-\delta_1,\delta_1]$,
$C^{-1}y^2<|u(y)|<Cy^2.$  Let $K>0$ such that ${b\over\sqrt {|u''(0)|}}r_n<\delta_1$ for any $n>K$. Then for any $y\in[{b\over\sqrt {|u''(0)|}}r_n,\delta_1]$,
\begin{align}\label{estimation-u(y)}
|u(y)|\geq \Big|u\Big({b\over\sqrt {|u''(0)|}}r_n\Big)\Big|={|u''(\xi_{r_n})|\over2}{b^2\over |u''(0)|}r_n^2> {|u''(0)|\over 4}{b^2\over |u''(0)|}r_n^2={b^2\over4}|c_n|,\end{align}
where $\xi_{r_n}\in(0,{b\over\sqrt {|u''(0)|}}r_n)$.
Using (\ref{equation-for-noncritical point-for-u}), we get by integration by parts that
\begin{align}\label{2oversqrtu0}
&\int_{{b\over\sqrt {|u''(0)|}}r_n}^{\delta_1}\left(|\phi_n'|^2+\alpha^2|\phi_n|^2+{(u''-\beta)|\phi_n|^2\over u-c_n}\right)dy\\ \nonumber
&=-\int_{{b\over\sqrt {|u''(0)|}}r_n}^{\delta_1}{{\om_n}\bar{\phi}_n\over u-c_n}dy +\phi_n'\bar{\phi}_n|_{{b\over\sqrt {|u''(0)|}}r_n}^{\delta_1}.
\end{align}
To proceed, we consider two cases.\smallskip

{\bf  Case I.} $u''(0)>0$ (resp. ${u''(0)-\beta\over u''(0)}>0$).

Note that $u\geq0$ on $[-\delta_1,\delta_1]$. Choose $b=3$. Then $u(y)-c_n^r>u(y)-|c_n|\geq|c_n^i|$ and $u(y)-c_n^r\leq u(y)+|c_n|\leq 2u(y)$ for any $n>K$ and any $y\in[{3\over\sqrt {u''(0)}}r_n,\delta_1]$, which gives
$${u(y)-c_n^r\over (u(y)-c_n^r)^2+(c_n^i)^2}\geq{1\over 2(u(y)-c_n^r)}\geq {1\over 4u(y)}\geq {1\over Cy^2}. $$
Thus, we have
\begin{align*}
  &\text{Re}{u''(y)-\beta\over u(y)-c_n}={(u''(y)-\beta)(u(y)-c_n^r)\over (u(y)-c_n^r)^2+(c_n^i)^2}\geq {1\over Cy^2}, \\
  &|u(y)-c_n|\geq u(y)-|c_n|\geq {5\over9}u(y)\geq{y^2\over C},
  \end{align*}
    for any $n>K$ and any $y\in[{3\over\sqrt {u''(0)}}r_n,\delta_1]$. Therefore, we get by (\ref{2oversqrtu0}) that
\begin{align}
&\int_{{3\over\sqrt {u''(0)}}r_n}^{\delta_1}\left(|\phi_n'|^2+\alpha^2|\phi_n|^2\right)dy +{1\over C}\|{\phi_n\over y}\|^2_{L^2({3\over\sqrt {u''(0)}}r_n,\delta_1)}\nonumber\\
&\leq C\|{\phi_n\over y}\|_{L^2({3\over\sqrt {u''(0)}}r_n,\delta_1)}\|{\omega_n\over y}\|_{L^2({3\over\sqrt {u''(0)}}r_n,\delta_1)}+\big|\phi_n'\bar{\phi}_n|_{{3\over\sqrt {u''(0)}}r_n}^{\delta_1}\big|\nonumber\\
&\leq \varepsilon\|{\phi_n\over y}\|^2_{L^2({3\over\sqrt {u''(0)}}r_n,\delta_1)}+C\|{\omega_n\over y}\|^2_{L^2({3\over\sqrt {u''(0)}}r_n,\delta_1)}+\big|\phi_n'\bar{\phi}_n|^{\delta_1}_{{3\over\sqrt {u''(0)}}r_n}\big|\nonumber
\end{align}
for some $0<\varepsilon<C^{-1}$. This gives
\begin{align}\label{desired-estimation-for-positive2}
\int_{{3\over\sqrt {u''(0)}}r_n}^{\delta_1}|\phi_n'|^2dy
\leq C\|{\omega_n\over y}\|^2_{L^2({3\over\sqrt {u''(0)}}r_n,\delta_1)}+\big|\phi_n'\bar{\phi}_n|^{\delta_1}_{{3\over\sqrt {u''(0)}}r_n}\big|.
\end{align}

{\bf Case II.} $u''(0)<-16$ (resp. $-{1\over8}<{u''(0)-\beta\over u''(0)}<0$).

Choose $\varepsilon_0>0$ sufficiently small so that $u''(0)<-16-2\varepsilon_0$.
Due to $u\in C^2([-\delta_1,\delta_1])$, there exists $\delta_2\in(0,\delta_1)$ so that $u''(y)<-16-\varepsilon_0$ and $u''(y)-\beta<2+2^{-4}\varepsilon_0$ for any $y\in[-\delta_2,\delta_2]$.
Then for any $y\in(0,\delta_2]$,
\begin{align*}
{u''(y)-\beta\over |u(y)|}={u''(y)-\beta\over{|u''(\xi_y)|\over2}y^2}\leq {2+2^{-4}\varepsilon_0\over(8+2^{-1}\varepsilon_0)y^2}={1-\varepsilon_1\over4y^2}
\end{align*}
where $\xi_y\in(0,y)$ and $\varepsilon_1={2^{-2}\varepsilon_0\over8+2^{-1}\varepsilon_0}>0$.
Let  $b={4\over \sqrt{\varepsilon_1}}+2\sqrt{2}$ in (\ref{convergence-near-0}).
By (\ref{estimation-u(y)}), we have $|u(y)-c_n^r|>|u(y)|-|c_n|\geq|u(y)|-{4\over b^2}|u(y)|={b^2-4\over b^2}|u(y)|$,
 and thus
 \begin{align*}
 \left|\text{Re}{u''(y)-\beta\over u(y)-c_n}\right|= {(u''(y)-\beta)|u(y)-c_n^r|\over |u(y)-c_n|^2}\leq{u''(y)-\beta\over |u(y)-c_n^r|}\leq
 {b^2\over b^2-4}{u''(y)-\beta\over|u(y)|}\leq{(1-\varepsilon_1)b^2\over 4(b^2-4)y^2}
 \end{align*}
 for any $y\in[{b\over\sqrt {u''(0)}}r_n,\delta_2]$ and $n$ sufficiently large. Then by (\ref{2oversqrtu0}) with $\delta_1$ replaced by $\delta_2$, we obtain
\begin{align}
&\int_{{b\over\sqrt {|u''(0)|}}r_n}^{\delta_2}\left(|\phi_n'|^2-{(1-\varepsilon_1)b^2|\phi_n|^2\over 4(b^2-4)y^2}\right)dy\nonumber\\
&\leq C\|{\phi_n\over y}\|_{L^2({b\over\sqrt {|u''(0)|}}r_n,\delta_2)}\|{\omega_n\over y}\|_{L^2({b\over\sqrt {|u''(0)|}}r_n,\delta_2)}+\big|\phi_n'\bar{\phi}_n|_{{b\over\sqrt {|u''(0)|}}r_n}^{\delta_2}\big|\nonumber\\
&\leq \varepsilon_2\|{\phi_n\over y}\|^2_{L^2({b\over\sqrt {|u''(0)|}}r_n,\delta_2)}+C\|{\omega_n\over y}\|^2_{L^2({b\over\sqrt {|u''(0)|}}r_n,\delta_2)}+\big|\phi_n'\bar{\phi}_n|^{\delta_2}_{{b\over\sqrt {|u''(0)|}}r_n}\big|\nonumber
\end{align}
for some $\varepsilon_2\in(0,{\varepsilon_1b^2\over 8(b^2-4)})$. Then we have
\begin{align}\label{desired-estimation-for-negative}
&\int_{{b\over\sqrt {|u''(0)|}}r_n}^{\delta_2}\left({\varepsilon_1b^2/2-8\over b^2-8}|\phi_n'|^2+{(1-\varepsilon_1/2)b^2\over4(b^2-8)}\left(4|\phi_n'|^2-{|\phi_n|^2\over y^2}\right)\right)dy\\
&=\int_{{b\over\sqrt {|u''(0)|}}r_n}^{\delta_2}\left(|\phi_n'|^2-{(1-\varepsilon_1/2)b^2|\phi_n|^2\over 4(b^2-8)y^2}\right)dy\nonumber\\
&\leq\int_{{b\over\sqrt {|u''(0)|}}r_n}^{\delta_2}\left(|\phi_n'|^2-\left({(1-\varepsilon_1)b^2\over 4(b^2-4)}+\varepsilon_2\right){|\phi_n|^2\over y^2}\right)dy\nonumber\\
&\leq C\|{\omega_n\over y}\|^2_{L^2({b\over\sqrt {|u''(0)|}}r_n,\delta_2)}+\big|\phi_n'\bar{\phi}_n|^{\delta_2}_{{b\over\sqrt {|u''(0)|}}r_n}\big|.\nonumber
\end{align}
Note that  ${\varepsilon_1b^2/2-8\over b^2-8}>0$ and ${(1-\varepsilon_1/2)b^2\over4(b^2-8)}>0$ by the choice of $b$.
Direct computation implies
\begin{align}\label{hardy-like-inequality}
&\int_{{b\over\sqrt {|u''(0)|}}r_n}^{\delta_2}\left|2\phi_n'-{\phi_n\over y}\right|^2dy\\
&=\int_{{b\over\sqrt {|u''(0)|}}r_n}^{\delta_2}\left(4|\phi_n'|^2+{|\phi_n|^2\over y^2}-{2\over y}(|\phi_n|^2)'\right)dy\nonumber\\
&=\int_{{b\over\sqrt {|u''(0)|}}r_n}^{\delta_2}\left(4|\phi_n'|^2-{|\phi_n|^2\over y^2}\right)dy-{2|\phi_n|^2\over y}\big|_{{b\over\sqrt {|u''(0)|}}r_n}^{\delta_2}.\nonumber
\end{align}
Plugging (\ref{hardy-like-inequality}) into (\ref{desired-estimation-for-negative}), we obtain
\begin{align*}
&\int_{{b\over\sqrt {|u''(0)|}}r_n}^{\delta_2}\left({\varepsilon_1b^2/2-8\over b^2-8}|\phi_n'|^2+{(1-\varepsilon_1/2)b^2\over4(b^2-8)}\left|2\phi_n'-{\phi_n\over y}\right|^2\right)dy\\
&\leq -{(1-\varepsilon_1/2)b^2\over4(b^2-8)}{2|\phi_n|^2\over y}\big|_{{b\over\sqrt {|u''(0)|}}r_n}^{\delta_2}+C\|{\omega_n\over y}\|^2_{L^2({b\over\sqrt {|u''(0)|}}r_n,\delta_2)}+\big|\phi_n'\bar{\phi}_n|^{\delta_2}_{{b\over\sqrt {|u''(0)|}}r_n}\big|,\nonumber
\end{align*}
which gives
\begin{align}\label{desired-estimation-for-negative22}
&\int_{{b\over\sqrt {|u''(0)|}}r_n}^{\delta_2}|\phi_n'|^2dy\leq { b^2-8\over\varepsilon_1b^2/2-8}\\
&\quad\times \left(-{(1-\varepsilon_1/2)b^2\over4(b^2-8)}{2|\phi_n|^2\over y}\big|_{{b\over\sqrt {|u''(0)|}}r_n}^{\delta_2}+C\|{\omega_n\over y}\|^2_{L^2({b\over\sqrt {|u''(0)|}}r_n,\delta_2)}+\big|\phi_n'\bar{\phi}_n|^{\delta_2}_{{b\over\sqrt {|u''(0)|}}r_n}\big|\right).\nonumber
\end{align}

Next we prove that each term in RHS of (\ref{desired-estimation-for-positive2}) and  (\ref{desired-estimation-for-negative22}) tends to $0$ as $n\to\infty$.
By Hardy's inequality, $\|{\omega_n\over y}\|_{L^2({b\over\sqrt {|u''(0)|}}r_n,\tilde\delta)}\leq C\|\omega_n\|_{H^1(-\delta,\delta)}\to0$, where $\tilde \delta=\delta_1$ in Case I and $\tilde \delta=\delta_2$ in Case II.
 Note that ${b\over\sqrt{|u''(0)|}}\notin E$, where $b=3$ in Case I and  $b={4\over \sqrt{\varepsilon_1}}+2\sqrt{2}$ in Case II.
Since $\phi_n$ is uniformly bounded in $H^2_{loc}((-\delta,\delta)\setminus\{0\})$ and  $\phi_n\rightharpoonup0$ in $H^1(-\delta,\delta)$, we get $\phi_n\to0$ in $C_{loc}^1([-\delta,\delta]\setminus\{0\})$. This, together with  $\tilde \phi_n\to0$ in $H_{loc}^1(\mathbb{R})\cap C_{loc}^1(\mathbb{R}\setminus E)$, implies that
\beno
&&\phi_n'\bar{\phi}_n|_{{b\over\sqrt {|u''(0)|}}r_n}^{\tilde\delta}=
\phi_n'\bar{\phi}_n(\tilde\delta)-\tilde\phi_n'\bar{\tilde\phi}_n\Big({b\over\sqrt {|u''(0)|}}\Big)\to0,\\
&&{|\phi_n|^2\over y}\big|_{{b\over\sqrt {|u''(0)|}}r_n}^{\delta_2}={|\phi_n(\delta_2)|^2\over \delta_2}-|\tilde\phi_n\Big({b\over{\sqrt{|u''(0)|}}}\Big)|^2/{b\over\sqrt{|u''(0)|}}\to0.
\eeno
Therefore, we have $\|\phi_n'\|_{L^2({b\over\sqrt {|u''(0)|}}r_n,\tilde\delta)}\to0$. A similar argument shows $\|\phi_n'\|_{L^2(-\tilde\delta,-{b\over\sqrt {|u''(0)|}}r_n)}\to0$. This, together with (\ref{convergence-near-0}), implies that
$\|\phi_n'\|_{L^2(-\tilde\delta,\tilde\delta)}\to 0$ and hence $\|\phi_n\|_{H^1(-\delta,\delta)}\to 0$.
\end{proof}

\begin{lemma}
\label{convergence-limiting-endpoint}
Let $c\in \text{Ran}\ (u)$, $y_1 \in u^{-1}\{c\}$, $u'(y_1)=0$,  and $\delta\in(0,y_2-y_1)$ so that $(u''(y)-\beta)(u''(y_1)-\beta)>0$ on $(y_1,y_1+\delta]$  and $[y_1,y_1+\delta]\cap u^{-1}\{c\}=\{y_1\}$.
 Assume that $\phi_n,{\omega_n}\in H^1(y_1,y_1+\delta)$,   $\phi_n(y_1)=\omega_n(y_1)=0$, $c_n^i>0$, $c_n\to c$, $\phi_n\rightharpoonup0,{\omega_n}\to0$ in $H^1(y_1,y_1+\delta)$ and
(\ref{equation-for-noncritical point-for-u}) holds on $[y_1,y_1+\delta]$.
Then $\phi_n\to0$ in $H^1(y_1,y_1+\delta)$.
\end{lemma}

\begin{remark}\label{remark-endpoint}
Similar result in Lemma \ref{convergence-limiting-endpoint} holds true with $y_1$ and $]y_1,y_1+\delta[$ replaced by
$y_2$ and $]y_2-\delta,y_2[$.
\end{remark}

\begin{proof}
Thanks to $\phi_n(y_1)=\omega_n(y_1)=0$ for each $n\geq 1$, the proof is similar to Lemma \ref{critical-point-not-beta-u-sec-dao-equal0}. So, we just sketch it here.

Without loss of generality, we may assume that $c=0$, $y_1=0$ and  $u''(0)-\beta=2$.
Let $c_n=r_n^2e^{2i\theta_n}$ and $\theta_n\in(0,{\pi\over2})$ for $n\geq1$.
We denote
$
\tilde \phi_n,\tilde \omega_n$ and $u_n$ as the same meanings in (\ref{tilde-phi-n-tilde-omega-n-u-n}).
Then (\ref{equation-after-scaling}) holds on $[0,{\delta\over r_n}]$.
Let $\tilde \phi_n\rightharpoonup \tilde \phi_0$ in $H^1_{loc}(\mathbb{R}^+)$ and $\theta_n\to\theta_0\in[0,{\pi\over2}]$. Then $\tilde \phi'_0\in L^2(\mathbb{R}^+)$ and $\tilde \phi_0(0)=0. $
To show that $\tilde \phi_n\to \tilde \phi_0$ in $H^1_{loc}(\mathbb{R}^+)$ and $\tilde \phi_0\equiv0$ on $\mathbb{R}^+$, we again consider five cases.\smallskip

{\bf Case 1.} $\theta_0\in(0,{\pi\over2})$.

In this case, we have $\tilde \phi_n\to \tilde {\phi}_0$ in $C_{loc}^1(\mathbb{R}^+)$, and
$({\beta+2\over2} y^2-e^{2i\theta_0})\tilde \phi_0''=2\tilde \phi_0.$
Using a cut-off argument  and noting that $\tilde \phi_0(0)=0$, we get
\begin{align*}
\int_{0}^\infty |\tilde \phi'_0|^2dy=-\int_{0}^\infty{2|\tilde\phi_0|^2\left({\beta+2\over2}y^2-\cos 2\theta_0+i\sin 2\theta_0\right)\over\left({\beta+2\over2}y^2-\cos2\theta_0\right)^2+\sin^22\theta_0}dy.\end{align*}
Taking the imaginary part of the equality, we get
$\tilde \phi_0\equiv0$ on $\mathbb{R}^+$.\smallskip

{\bf Case 2.} $\theta_0=0$ and ${\beta+2\over2}>0$.

In this case, we have $\tilde \phi_n\to \tilde {\phi}_0$ in $H^1_{loc}(\mathbb{R}^+)\cap C_{loc}^1(\mathbb{R}^+\setminus \{a\})$ and
$({\beta+2\over2} y^2-1)\tilde \phi_0''=2\tilde \phi_0,$ where $a=\sqrt{{2\over\beta+2}}$.
Using a cut-off argument  and $\tilde \phi_0(0)=0$, we have
\begin{align*}
-\int_{0}^\infty|\tilde \phi_0'|^2dy
=p.v.\int_{0}^\infty{2|\tilde  \phi_0|^2\over {\beta+2\over2}y^2-1}dy +i\pi{{2}\left|\tilde \phi_0\left(a\right)\right|^2\over \sqrt{2(\beta+2)}}.
\end{align*}
Then
$\tilde \phi_0\left(a\right)=0$
and thus $
\tilde \phi_0\in C^1(\mathbb{R}^+).
$
Note that
\begin{align*}
-\int_{a}^\infty|\tilde \phi_0'|^2dy=\int_{a}^\infty\tilde \phi_0''\bar{\tilde \phi}_0dy
=\int_{a}^\infty{2|\tilde\phi|^2\over{{\beta+2\over2}y^2-1}}dy.
\end{align*}
Then we get $\tilde \phi_0\equiv0$ on $[a,+\infty),$ and hence
 $\tilde\phi_0\equiv0$ on $\mathbb{R}^+$.\smallskip

{\bf Case 3.} $\theta_0=0$ and ${\beta+2\over2}<0$.

In this case, we have $\tilde \phi_n\to \tilde {\phi}_0$ in $C_{loc}^1(\mathbb{R}^+)$. Let $V(y)={-2\over {\beta+2\over2}y^2-1}$. Then we have $V>0$ on $\mathbb{R}^+$  and
$\int_{0}^\infty(2V+yV')\tilde |\phi_0|^2dy=0.$
Moreover,
$2V(y)+yV'(y)={4\over\left({\beta+2\over2}y^2-1\right)^2}>0$,
and thus $\tilde \phi_0\equiv0$ on $\mathbb{R}^+$.\smallskip

{\bf Case 4.} $\theta_0={\pi\over2}$ and ${\beta+2\over2}>0$.

In this case, we have $\tilde \phi_n\to \tilde {\phi}_0$ in $C_{loc}^1(\mathbb{R}^+)$. Since $-\int_{0}^\infty|\tilde\phi_0|'^2dy=\int_{0}^\infty{2|\tilde \phi_0|^2\over {\beta+2\over 2} y^2+1}dy,$
we have $\tilde \phi_0\equiv0$ on $\mathbb{R}^+$.\smallskip

{\bf Case 5.} $\theta_0={\pi\over2}$ and ${\beta+2\over2}<0$.

In this case, we have $\tilde \phi_n\to \tilde {\phi}_0$ in $H^1_{loc}(\mathbb{R}^+)\cap C_{loc}^1(\mathbb{R}^+\setminus \{\sqrt{{-2\over\beta+2}}\})$, and
$({\beta+2\over2} y^2+1)\tilde \phi_0''=2\tilde \phi_0.$
Then we get
\begin{align*}
-\int_{0}^\infty|\tilde \phi_0'|^2dy
=p.v.\int_{0}^\infty{2|\tilde  \phi_0|^2\over {\beta+2\over2}y^2+1}dy +i\pi{{2}\left|\tilde \phi_0\left(\sqrt{{-2\over \beta+2}}\right)\right|^2\over\sqrt{-2(\beta+2)}}.
\end{align*}
Then $\tilde \phi_0\left(\sqrt{{-2\over\beta+2}}\right)=0,$ and thus
$\tilde \phi_0\in C^1(\mathbb{R}^+).$
Since $
-\int_{0}^{\sqrt{{-2\over\beta+2}}}|\tilde \phi_0'|^2dy
=\int_{0}^{\sqrt{{-2\over\beta+2}}}{2|\tilde \phi_0|^2\over \left({\beta+2\over2}\right)y^2+1}dy,$
 we  have $\tilde \phi_0\equiv0$ on $[0,\sqrt{{-2\over\beta+2}}]$, and thus
 $\tilde\phi_0\equiv0$ on $\mathbb{R}^+$.

Thus, $\|\phi_n'\|_{L^2(0,{b\over \sqrt{|u''(0)|}}r_n)}\to0$, and
moreover $\|\phi_n'\|_{L^2({b\over\sqrt {|u''(0)|}}r_n,\tilde\delta)}\to0$ for some $b>0$ and some $\tilde \delta>0$. Then $\|\phi_n\|_{H^1(0,\delta)}\to 0$.
\end{proof}

Finally, we consider the compactness in the domain with critical points satisfying $u''-\beta=0$.

\begin{lemma}\label{critical-point-beta-u2dao=0}
 Let $c\in \text{Ran}\ (u)$, $y_0\in u^{-1}\{c\}$, $u'(y_0)=0$, $u''(y_0)-\beta=0$, and $\delta>0$ so that $u''(y)u''(y_0)>0$ on $[y_0,y_0+\delta]$, $[y_0,y_0+\delta]\cap u^{-1}\{c\}=\{y_0\}$ and  $[y_0,y_0+\delta]\subset [y_1,y_2]$.
 Assume that $\phi_n,{\omega_n\over u'}\in H^1(y_0,y_0+\delta)$, $c_n\in\mathbb{C}$ such that $c^i_n>0$, $c_n\to c$, $\phi_n\rightharpoonup\phi,{\omega_n\over u'}\to{\omega\over u'} $ in $H^1(y_0,y_0+\delta)$ and $(\ref{equation-for-noncritical point-for-u})$ holds on $[y_0,y_0+\delta]$.
Then $\phi_n\to\phi$ in $H^1(y_0,y_0+\delta)$.
\end{lemma}\begin{remark}\label{remark-on-critical-beta}
If $[y_0-\delta,y_0]\subset [y_1,y_2]$,
similar result in Lemma \ref{critical-point-beta-u2dao=0} holds with $]y_0,y_0+\delta[$ replaced by $]y_0-\delta,y_0[$, including the results of the uniform $H^1$ bound of $g_n$ and the uniform $L^p$ bound of $\ln(u-c_n)\ (1<p<+\infty)$ in the following proof.
\end{remark}

\begin{proof}
We denote
\begin{align}\label{def-gn-func}
g_n={(u''-\beta)\phi_n+\omega_n\over u'}.
\end{align}
Using the facts that
\begin{align*}
{u''-\beta\over u'}={\int_{y_0}^{y}u'''(z)dz\over\int_{y_0}^{y}u''(z)dz}={\int_0^1u'''(y_0+t(y-y_0))dt\over\int_0^1u''(y_0+t(y-y_0))dt},
\end{align*}
and
\begin{align*}
\left({u''-\beta\over u'}\right)'
=&{\int_0^1tu''''(y_0+t(y-y_0))dt\int_0^1u''(y_0+t(y-y_0))dt\over
\left(\int_0^1u''(y_0+t(y-y_0))dt\right)^2}\\
&-{\int_0^1u'''(y_0+t(y-y_0))dt\int_0^1tu'''(y_0+t(y-y_0))dt\over\left(\int_0^1u''(y_0+t(y-y_0))dt\right)^2},
\end{align*}
we have ${u''-\beta\over u'}\in H^1(y_0,y_0+\delta).$ Since $\phi_n$ and ${\omega_n\over u'}$ are uniformly bounded in $H^1(y_0,y_0+\delta)$, we infer that
$g_n$ is  uniformly bounded in $H^1(y_0,y_0+\delta)$.

Thanks to $(u-c_n)(\phi_n''-\alpha^2\phi_n)=u'g_n,$ we find
\begin{align*}
(\phi_n'-g_n\ln (u-c_n))'=\alpha^2\phi_n-g_n'\ln (u-c_n).
\end{align*}Here $\ln (u-c_n)=\ln |u-c_n|+i\left(\arctan({u-c_n^r\over c_n^i})-{\pi\over2}\right). $
Let us claim that $\ln(u-c_n)$ is uniformly bounded in $L^p(y_0,y_0+\delta)$ for every $1<p<+\infty.$ This, along with
\begin{align*}
\|g_n'\ln (u-c_n)\|_{L^1(y_0,y_0+\delta)}\leq\|g_n'\|_{L^2(y_0,y_0+\delta)}\|\ln (u-c_n)\|_{L^2(y_0,y_0+\delta)},
\end{align*}
and that $g_n$ is uniformly bounded in $L^\infty(y_0,y_0+\delta)$, yields that $\phi_n'-g_n\ln (u-c_n)$ is uniformly bounded in $\dot{W}^{1,1}(y_0,y_0+\delta)\cap L^2(y_0,y_0+\delta)$. Thus,
$\phi_n'-g_n\ln (u-c_n)$ is uniformly bounded in $L^\infty(y_0,y_0+\delta)$. This implies that
\begin{align*}
\lim\limits_{\varepsilon\to0^+}\sup\limits_{n}\|\phi_n\|_{H^1(y_0,y_0+\varepsilon)}=0,
\end{align*}
where $\varepsilon\in(0,\delta)$. This along with the fact $\phi_n\to\phi$ in $H^1_{loc}((y_0,y_0+\delta])$ implies that
$\phi_n\to\phi$ in $H^1(y_0,y_0+\delta)$.\smallskip

Finally, we show that $\ln(u-c_n)$ is uniformly bounded in $L^p(y_0,y_0+\delta)$ for $1<p<+\infty.$
Thanks to $u'(y_0)=0$ and $u''(y_0)\neq 0,$ there exist $\delta_1\in(0,\delta)$ and $c_0>0$ such that $ |u'(y)|\geq c_0 (y-y_0)$  for any $y\in[y_0,y_0+\delta_1]$. Note that
$(y-y_0)^2-(z-y_0)^2=(y-z)^2+2(y-z)(z-y_0)\geq (y-z)^2$ if $y_0\leq z\leq y\leq y_0+\delta_1$ and
$(z-y_0)^2-(y-y_0)^2=(z-y)^2+2(z-y)(y-y_0)\geq (z-y)^2$ if $y_0\leq y< z\leq y_0+\delta_1$.
Thus,
\begin{align*}
 |u(y)-u(z)|=\left|\int_y^z|u'(\xi)|d\xi\right|\geq c_0\left|\int_y^z(\xi-y_0)d\xi\right|\geq {c_0\over2} |(y-y_0)^2-(z-y_0)^2|\geq {c_0\over2} |y-z|^2,
 \end{align*}
where $y,z\in[y_0,y_0+\delta_1]$. Choose $y_n\in [y_0,y_0+\delta_1]$  such that $u(y_n)=c_n^r$ if $c_n^r\in u([y_0,y_0+\delta_1])$; $y_n=y_0$ if $c_n^r<u(y_0)<u(y_0+\delta_1)$ or $c_n^r>u(y_0)>u(y_0+\delta_1)$; $y_n=y_0+\delta_1$ if $c_n^r<u(y_0+\delta_1)<u(y_0)$ or $c_n^r>u(y_0+\delta_1)>u(y_0)$. Then $C\geq |u(y)-c_n|\geq|u(y)-c_n^r|\geq|u(y)-u(y_n)|\geq {c_0\over2}|y-y_n|^2$ and
$$|\ln(u-c_n)|\leq |\ln|u-c_n||+\pi\leq C-\ln|u-c_n|\leq C-\ln({c_0\over2}|y-y_n|^2) $$
for $y\in(y_0,y_0+\delta_1)$. Hence,
$$\int_{y_0}^{y_0+\delta_1}|\ln(u-c_n)|^pdy\leq C \int_{y_0}^{y_0+\delta_1}(|\ln|y-y_n|^2|^p+1) dy\leq C \int_{-\delta_1}^{\delta_1}(|\ln|z|^2|^p+1) dz\leq C.$$
This shows that $\ln(u-c_n)$ is uniformly bounded in $L^p(y_0,y_0+\delta)$ for $1<p<+\infty.$
\end{proof}

\begin{lemma}\label{critical-point-beta-u2dao=0-formula}
 Let $c\in \text{Ran}\ (u)$, $y_0\in u^{-1}\{c\}\cap (y_1,y_2)$, $u'(y_0)=0$, $u''(y_0)-\beta=0$, and $\delta>0$ so that  $u''(y)u''(y_0)>0$ on $[y_0,y_0+\delta]$, $[y_0-\delta,y_0+\delta]\cap u^{-1}\{c\}=\{y_0\}$ and $[y_0-\delta,y_0+\delta]\subset (y_1,y_2)$.
 Assume that $\phi_n,{\omega_n\over u'}\in H^1(y_0-\delta,y_0+\delta)$,  $c_n^i>0$, $c_n\to c$, $\phi_n\rightharpoonup\phi,{\omega_n\over u'}\to{\omega\over u'} $ in $H^1(y_0-\delta,y_0+\delta)$ and $(\ref{equation-for-noncritical point-for-u})$  holds on $[y_0-\delta,y_0+\delta]$.
Then  for all $\varphi\in H_0^1(y_0-\delta,y_0+\delta)$,
\begin{align*}
\int_{y_0-\delta}^{y_0+\delta}(\phi'\varphi'+\alpha^2\phi\varphi)dy+p.v.\int_{y_0-\delta}^{y_0+\delta}{((u''-\beta)\phi+\omega)\varphi
\over u-c}dy=0.\end{align*}
\end{lemma}
\begin{proof}
By (\ref{equation-for-noncritical point-for-u}), for any $\varphi\in H^1_0(y_0-\delta,y_0+\delta)$, we have
\begin{align*}
\int_{y_0-\delta}^{y_0+\delta}\left(\phi_n'\varphi'+\alpha^2\phi_n\varphi+{((u''-\beta)\phi_n+\omega_n)\varphi \over u-c_n}\right)dy=0.
\end{align*}
Thanks to $\phi_n\rightharpoonup\phi $ in $H^1(y_0-\delta,y_0+\delta)$, we get
\begin{align*}
\lim\limits_{n\to\infty}\int_{y_0-\delta}^{y_0+\delta}\left(\phi_n'\varphi'+\alpha^2\phi_n\varphi\right) dy=\int_{y_0-\delta}^{y_0+\delta}\left(\phi'\varphi'+\alpha^2\phi\varphi\right) dy.
\end{align*}
So, it suffices to show that
\begin{align}\label{real-part}
\lim\limits_{n\to\infty}\int_{y_0-\delta}^{y_0+\delta}{((u''-\beta)\phi_n+\omega_n)\varphi\over u-c_n} dy&=
p.v.\int_{y_0-\delta}^{y_0+\delta}{((u''-\beta)\phi+\omega)\varphi\over u-c} dy.
\end{align}

First of all, for any $\varepsilon>0$, there exists $\tau_1\in(0,\delta)$ so that if $0<\tau<\tau_1$, then
\begin{align*}
\left|\int_{E_{\tau}^c}{((u''-\beta)\phi+\omega)\varphi\over u-c}dy-p.v.\int_{y_0-\delta}^{y_0+\delta}{((u''-\beta)\phi+\omega)\varphi\over u-c}dy\right|<{\varepsilon\over3},
\end{align*}
where $E_\tau^c=[y_0-\delta,y_0+\delta]\setminus[y_0-\tau,y_0+\tau]$.
Note that
\begin{align*}
{((u''-\beta)\phi_n+\omega_n)\varphi\over u-c_n}\longrightarrow {((u''-\beta)\phi+\omega)\varphi\over u-c}
\end{align*}
uniformly  in $E_{\tau}^c$ as $n\to\infty$. Hence, if $n$ is sufficiently large, then
\begin{align*}
\left|\int_{E_\tau^c}{((u''-\beta)\phi_n+\omega_n)\varphi\over u-c_n}dy-
\int_{E_{\tau}^c}{((u''-\beta)\phi+\omega)\varphi\over u-c}dy\right|<{\varepsilon\over3}.
\end{align*}

Let $\tau<\tau_1$ and note that
\begin{align*}
&\int_{y_0-\tau}^{y_0+\tau}{((u''-\beta)\phi_n+\omega_n)\varphi\over u-c_n}dy
=\int_{y_0-\tau}^{y_0+\tau}{(g_n\varphi)u'\over u-c_n}dy\\
=&\left((g_n\varphi)\ln(u-c_n)\right)\big|_{y_0-\tau}^{y_0+\tau}-\int_{y_0-\tau}^{y_0+\tau}{(g_n\varphi)'\ln(u-c_n)}dy:=I_{n,\tau}+II_{n,\tau},
\end{align*}
where $g_n$ is given in (\ref{def-gn-func}). Direct computation gives
\begin{align*}
I_{n,\tau}=(g_n\varphi)\big|_{y_0-\tau}^{y_0+\tau}\ln(u(y_0+\tau)-c_n)+(g_n\varphi)(y_0-\tau)\ln(u-c_n)\big|_{y_0-\tau}^{y_0+\tau}.
\end{align*}
By the Sobolev embedding $H^1(y_0-\tau,y_0+\tau)\hookrightarrow C^{0,{1\over2}}(y_0-\tau,y_0+\tau)$, we have
\begin{align*}
&\left|(g_n\varphi)\big|_{y_0-\tau}^{y_0+\tau}\right|
\leq C\tau^{1\over2},\ |(g_n\varphi)(y_0-\tau)|\leq C.
\end{align*}
This, together with $c_n\to u(y_0)$, yields  \begin{align*}
|I_{n,\tau}|\leq& C\tau^{1\over2}|\ln(u(y_0+\tau)-c_n)|+C\left|\ln(u-c_n)\big|_{y_0-\tau}^{y_0+\tau}\right|
\\ \longrightarrow& C\tau^{1\over2}|\ln(u(y_0+\tau)-u(y_0))|+C\left|\ln\frac{u(y_0+\tau)-u(y_0)}{u(y_0-\tau)-u(y_0)}\right|
\end{align*}
as $n\to \infty$. Note that
$u(y)-u(y_0)=u''(\xi_y)|y-y_0|^2/2,$
where $\xi_y\in(y,y_0)$ or $\xi_y\in(y_0,y)$, and thus
$$C_1|y-y_0|^2\leq |u(y)-u(y_0)|\leq C_2|y-y_0|^2$$
for $y\in[y_0-\tau,y_0+\tau]$.
Then we have
\begin{align*}
\limsup_{n\to+\infty}|I_{n,\tau}|&\leq C\tau^{1\over2}|\ln(u(y_0+\tau)-u(y_0))|+C\left|\ln\frac{u(y_0+\tau)-u(y_0)}{u(y_0-\tau)-u(y_0)}\right|\\&\leq C\tau^{1\over2}(|\ln\tau^2|+C)+C\left|\ln\frac{u''(\xi_{y_0+\tau})\tau^2/2}{u''(\xi_{y_0-\tau})\tau^2/2}\right|,
\end{align*}
which gives
\begin{align*}
\limsup_{\tau\to0+}\limsup_{n\to+\infty}|I_{n,\tau}|&\leq C\lim_{\tau\to0+}\tau^{1\over2}(|\ln\tau^2|+C)+C\lim_{\tau\to0+}\left|\ln\frac{u''(\xi_{y_0+\tau})}{u''(\xi_{y_0-\tau})}\right|=0.
\end{align*}
Thus, if $\tau>0$ is sufficiently small and $n$ is sufficiently large, we have
\begin{align}\label{real-part3}
|I_{n,\tau}|&<\varepsilon/6.
\end{align}
Using the facts that the uniform $H^1$ bound of $g_n$ and the uniform $L^4$ bound of $\ln(u-c_n)$, we have
\begin{align}\label{real-part4}
|II_{n,\tau}|\leq& (2\tau)^{1/4}\|(g_n\varphi)'\|_{L^2(y_0-\tau,y_0+\tau)}\|\ln(u-c_n)\|_{L^4(y_0-\tau,y_0+\tau)}
\leq C\tau^{1/4}\leq \varepsilon/ 6,
\end{align}
when $\tau>0$ is sufficiently small. Now, it follows from (\ref{real-part3}) and (\ref{real-part4}) that
\begin{align*}
\left|\int_{y_0-\tau}^{y_0+\tau}{((u''-\beta)\phi_n+\omega_n)\varphi\over u-c_n}dy\right|\leq {\varepsilon\over3},
\end{align*}
when $\tau>0$ is sufficiently small and $n$ is sufficiently large.
Therefore, (\ref{real-part}) holds.
\end{proof}

\subsection{Limiting absorption principle for general shear flows}

In this subsection, we establish the limiting absorption principle for a class of shear flows satisfying \textbf{(H1)}.

The spectrum $\sigma(\mathcal{R}_{\alpha,\beta})$ is compact and $\sigma_{ess}(\mathcal{R}_{\alpha,\beta})=\text{Ran }(u)$ for any $\alpha>0$ and $\beta\in\mathbb{R}$. The  embedding eigenvalue of $\mathcal{R}_{\alpha,\beta}$ is defined  as follows.

\begin{definition}\label{embedding eigenvalue}
 Let $\alpha>0$ and $\beta\in\mathbb{R}$. $c\in \text{Ran }(u)$ is called an embedding eigenvalue of $\mathcal{R}_{\alpha,\beta}$ if there exists a nontrivial $\phi\in H_0^1(y_1,y_2)$ such that for any $\varphi\in H_0^1(y_1,y_2)$ and $supp \ \varphi\subset(y_1,y_2)\setminus\{y\in(y_1,y_2)|u(y)=c,u'(y)=0,u''(y)\neq\beta\}$,
$$\int_{y_1}^{y_2}(\phi'\varphi'+\alpha^2\phi\varphi)dy+p.v.\int_{y_1}^{y_2}{(u''-\beta)\phi\varphi\over u-c}dy+i\pi\sum\limits_{y\in u^{-1}\{c\},u'(y)\neq0}{(u''-\beta)\phi\varphi(y)\over |u'(y)|}=0.$$
\end{definition}

\begin{theorem}\label{uniform-H1-bound}
Let $\alpha>0$ and $\beta\in\mathbb{R}$.
Assume that $u$ satisfies \textbf{(H1)}, $\mathcal{R}_{\alpha,\beta}$ has no embedding eigenvalues, $\omega(y)=0$ for any $y\in\{y_1,y_2\}\cap(u')^{-1}\{0\}$, and ${\omega\over p}\in H^1(y_1,y_2)$,  where $p$ is given in (\ref{def-p}).
Then there exists $\varepsilon_0>0$ such that $\Omega_{\varepsilon_0}\cap\sigma_d(\mathcal{R}_{\alpha,\beta})=\emptyset$ and
for any $c\in \Omega_{\varepsilon_0}\setminus \text{Ran }(u)$, the unique solution $\Phi$ to the boundary value problem
\begin{align}\label{rayleigh-equation-Phi}
(u-c)(\Phi''-\alpha^2\Phi)-(u''-\beta)\Phi=\omega, \;\Phi(y_1)=\Phi(y_2)=0
\end{align}
has the uniform $H^1$ bound
\begin{align}\label{uniform bound-Phi}
\|\Phi\|_{H^1(y_1,y_2)}\leq C\|{\omega\over p}\|_{H^1(y_1,y_2)},
\end{align}
where $\Omega_{\varepsilon_0}=\big\{c\in\mathbb{C}|\exists\ c_0\in \text{Ran }(u)\ \text{such that}\ |c-c_0|<\varepsilon_0\big\}.$

Moreover, there exist $\Phi_{\pm}(\cdot,c)\in H_0^1(y_1,y_2)$ for each $c\in \text{Ran }(u)$ such that $\Phi(\cdot,c\pm i\varepsilon)\to \Phi_{\pm}(\cdot,c)$ in $C([y_1,y_2])$ as $\varepsilon\to0^+$ and
 \begin{align}\label{bound-Phi-pm}
 \|\Phi_{\pm}(\cdot,c)\|_{H^1(y_1,y_2)}\leq C \|{\omega\over p}\|_{H^1(y_1,y_2)},
 \end{align}
uniformly for $c\in \text{Ran }(u)$.

\end{theorem}
\begin{proof}
We first prove (\ref{uniform bound-Phi}). Assume that  $c^i>0$. The proof for the case  $c^i<0$ is similar.

Suppose that (\ref{uniform bound-Phi}) is not true. Then there exists $\Phi_n\in H_0^1(y_1,y_2)$, ${\omega_n\over p}\in H^1(y_1,y_2)$, $\omega_n(y)=0$ for any $y\in\{y_1,y_2\}\cap(u')^{-1}\{0\}$,  and  $c_n$ with $c_n^i>0$ such that $\|\Phi_n\|_{H^1(y_1,y_2)}=1$, $\|{\omega_n\over p}\|_{H^1(y_1,y_2)}\to0$, $c_n\to c_0\in \text{Ran }(u)$ and
\begin{align*}
(u-c_n)(\Phi_n''-\alpha^2\Phi_n)-(u''-\beta)\Phi_n=\omega_n.
\end{align*}
Up to a subsequence,  there exists $\Phi_0\in H_0^1(y_1,y_2)$ such that $\Phi_n\rightharpoonup\Phi_0$ in $H^1(y_1,y_2)$.

Next we show that for any $\varphi\in H_0^1(y_1,y_2)$ with $\text{supp} \ \varphi\in(y_1,y_2)\setminus\{y\in(y_1,y_2)|u(y)=c_0,u'(y)=0,u''(y)\neq\beta\}$,
\begin{align}\label{integral-formula-for-limit-function}
\int_{y_1}^{y_2}(\Phi_0'\varphi'+\alpha^2\Phi_0\varphi)dy&+p.v.\int_{y_1}^{y_2}{(u''-\beta)\Phi_0\varphi\over u-c_0}dy\\
&+i\pi\sum\limits_{y\in u^{-1}\{c_0\},u'(y)\neq0}{(u''-\beta)\Phi_0\varphi(y)\over |u'(y)|}=0.\nonumber
\end{align}
Let  $y_0\in \{y\in(y_1,y_2)|u(y)=c_0,u'(y)\neq0\}$. By Lemma \ref{noncritical-point-formula}, (\ref{integral-formula-for-limit-function}) holds
for any $\varphi\in H_0^1(y_1,y_2)$ with $\text{supp}\ \varphi\subset(y_0-\delta,y_0+\delta)$, where $\delta>0$ is sufficiently small.
Let $y_0\in\{y\in(y_1,y_2)|u(y)=c_0,u'(y)=0,u''(y)=\beta\}$. By Lemma \ref{critical-point-beta-u2dao=0-formula}, (\ref{integral-formula-for-limit-function}) holds
for any $\varphi\in H_0^1(y_1,y_2)$ with $\text{supp}\ \varphi\subset(y_0-\delta,y_0+\delta)$.
Moreover,
since $\Phi_n$ is uniformly bounded in $H_{loc}^3([y_1,y_2]\setminus u^{-1}\{c_0\}),$ we have
$\Phi_n\to\Phi_0$ in $C^2_{loc}([y_1,y_2]\setminus u^{-1}\{c_0\})$ and thus  for any $y\in [y_1,y_2]\setminus u^{-1}\{c_0\}$,
\begin{align}\label{not-zero-point}
(u-c_0)(\Phi_0''-\alpha^2\Phi_0)-(u''-\beta)\Phi_0=0, \;\Phi_0(y_1)=\Phi_0(y_2)=0.
\end{align}
Therefore, (\ref{integral-formula-for-limit-function}) holds for any $\varphi\in H_0^1(y_1,y_2)$ with $\text{supp} \ \varphi\in(y_1,y_2)\setminus\{y\in(y_1,y_2)|u(y)=c_0,u'(y)=0,u''(y)\neq\beta\}$.

If $\Phi_0\neq0$, from Definition \ref{embedding eigenvalue}, we know that $c_0$ is an embedding eigenvalue of $\mathcal{R}_{\alpha,\beta}$, which is  a contradiction. Thus, $\Phi_0\equiv0$ on $[y_1,y_2]$.

Now we show that $\Phi_n\to0$ in $H^1(y_1,y_2)$.  Let $y_0\in\{y\in[y_1,y_2]|u(y)=c_0,u'(y)\neq0\}$. Then by Lemma \ref{noncritical point}, $\Phi_n\to0$ in $H^1((y_0-\delta,y_0+\delta)\cap[y_1,y_2])$. Let $y_0\in\{y\in(y_1,y_2)|u(y)=c_0,u'(y)=0,u''(y)\neq\beta\}$. It follows from Lemma \ref{critical-point-not-beta-u-sec-dao-equal0} that
$\Phi_n\to0$ in $H^1(y_0-\delta,y_0+\delta)$.  Let $y_0\in\{y\in\{y_1,y_2\}|u(y)=c_0,u'(y)=0,u''(y)\neq\beta\}$. Then
$\Phi_n\to0$ in $H^1((y_0-\delta,y_0+\delta)\cap[y_1,y_2])$
due to Lemma \ref{convergence-limiting-endpoint} and Remark \ref{remark-endpoint}. Let $y_0\in\{y\in[y_1,y_2]|u(y)=c_0,u'(y)=0,u''(y)=\beta\}$. In view of Lemma \ref{critical-point-beta-u2dao=0} and Remark \ref{remark-on-critical-beta}, we have $\Phi_n\to0$ in $H^1((y_0-\delta,y_0+\delta)\cap[y_1,y_2])$.
Note that $\Phi_n\to0$ in $C^2_{loc}([y_1,y_2]\setminus u^{-1}\{c_0\})$. Therefore, $\Phi_n\to0$ in $H^1(y_1,y_2)$, which contradicts $\|\Phi_n\|_{H^1(y_1,y_2)}=1$, $n\geq1$.

We have shown (\ref{uniform bound-Phi}) when $c\in \Omega_{\varepsilon_0}\setminus\mathbb{R}$. Since $\Phi(\cdot,c)\to\Phi(\cdot,c_1)$ in $H^1(y_1,y_2)$ as $c\to c_1 \in (\Omega_{\varepsilon_0}\cap\mathbb{R})\setminus \text{Ran }(u)$, we have (\ref{uniform bound-Phi}) holds true for all $c\in \Omega_{\varepsilon_0}\setminus \text{Ran }(u)$. \smallskip

Next, we prove the second part of the theorem. We only show the conclusion for $\Phi_{+}$, and the proof for $\Phi_{-}$ is similar.

Consider $\Phi$ as a mapping $c\mapsto\Phi(\cdot,c)$ from $\Omega_{\varepsilon_0}\setminus \mathbb{R}$ to $C([y_1,y_2])$.
Then we show that $\Phi$ is uniformly continuous in $\Omega_{+}=\{c+i\varepsilon|c\in \text{Ran}(u), 0<\varepsilon\leq {\varepsilon_0\over2}\}.$
Suppose otherwise, there exist $c_{n,1},c_{n,2}\in\Omega_{+}$ and $\kappa>0$  such that  $|c_{n,1}-c_{n,2}|\to0$ and
$
\|\Phi(\cdot,c_{n,1})-\Phi(\cdot,c_{n,2})\|_{C([y_1,y_2])}>\kappa.$
By (\ref{uniform bound-Phi}), $\Phi(\cdot, c_{n,j})$, $n\geq1$, is uniformly bounded in $H^1(y_1,y_2)$, where $j=1,2$. Then up to a subsequence, $\Phi(\cdot, c_{n,j})\rightharpoonup \Phi_j$ in
$H^1(y_1,y_2)$ for some $\Phi_j\in H_0^1(y_1,y_2)$ and $c_{n,j}\to c_0$ for some $c_0\in\bar{\Omega}_{+}$, where $j=1,2$. So $\Phi(\cdot, c_{n,j})\to \Phi_j$ in $C([y_1,y_2])$ and thus $\|\Phi_1-\Phi_2\|_{C([y_1,y_2])}\geq \kappa$. We divide the following discussion into two cases.\smallskip

{\bf Case 1.} $c_0\in \Omega_{+}$.

In this case, $\Phi_j\in C^2([y_1,y_2])$ and satisfies (\ref{rayleigh-equation-Phi}) with $c=c_0$, where $j=1,2$. Then $\Phi_1-\Phi_2$ is a solution of  (\ref{not-zero-point}) with $\om=0$. So, $\Phi_1-\Phi_2\equiv0$ on $[y_1,y_2]$, which is a contradiction.\smallskip

{\bf Case 2.} $c_0\in  \text{Ran }(u)$.

First of all, we show that for any $\varphi\in H_0^1(y_1,y_2)$ with $supp \ \varphi\subset(y_1,y_2)\setminus\{y\in(y_1,y_2)|u(y)=c_0,u'(y)=0,u''(y)\neq\beta\}$,
\begin{align}\label{integral-formula-for-limit-function-inhomogeneuous}
\int_{y_1}^{y_2}(\Phi_j'\varphi'+\alpha^2\Phi_j\varphi)dy&+p.v.\int_{y_1}^{y_2}{((u''-\beta)\Phi_j+\omega)\varphi\over u-c_0}dy
\\\nonumber &+i\pi\sum\limits_{y\in u^{-1}\{c_0\}, u'(y)\neq0}{((u''-\beta)\Phi_j+\omega)\varphi(y)\over |u'(y)|}=0,
\end{align}
where $j=1,2$. Choose $\delta>0$  sufficiently small. Let  $y_0\in \{y\in(y_1,y_2)|u(y)=c_0,u'(y)\neq0\}$. By Lemma \ref{noncritical-point-formula}, (\ref{integral-formula-for-limit-function-inhomogeneuous}) holds
for any $\varphi\in H_0^1(y_1,y_2)$ with $supp\ \varphi\subset(y_0-\delta,y_0+\delta)$.
Let $y_0\in\{y\in(y_1,y_2):u(y)=c_0,u'(y)=0,u''(y)=\beta\}$. By Lemma \ref{critical-point-beta-u2dao=0-formula}, (\ref{integral-formula-for-limit-function-inhomogeneuous}) holds
for any $\varphi\in H_0^1(y_1,y_2)$ with $supp\ \varphi\subset(y_0-\delta,y_0+\delta)$.
This, together with the fact that
$\Phi(\cdot,c_{n,j})\to\Phi_j$ in $C^2_{loc}([y_1,y_2]\setminus u^{-1}\{c_0\})$, implies   (\ref{integral-formula-for-limit-function-inhomogeneuous}) holds for any $\varphi\in H_0^1(y_1,y_2)$ with $supp \ \varphi\subset(y_1,y_2)\setminus\{y\in(y_1,y_2)|u(y)=c_0,u'(y)=0,u''(y)\neq\beta\}$.

Set $\Phi_0=\Phi_1-\Phi_2$. Then we have
\begin{align*}
\int_{y_1}^{y_2}(\Phi_0'\varphi'+\alpha^2\Phi_0\varphi)dy+p.v.\int_{y_1}^{y_2}{(u''-\beta)\Phi_0\varphi\over u-c_0}dy
 +i\pi\sum\limits_{y\in u^{-1}\{c_0\}, u'(y)\neq0}{(u''-\beta)\Phi_0\varphi(y)\over |u'(y)|}=0
\end{align*}
for any $\varphi\in H_0^1(y_1,y_2)$ with
$supp \ \varphi\subset(y_1,y_2)\setminus\{y\in(y_1,y_2)|u(y)=c_0,u'(y)=0,u''(y)\neq\beta\}$.
Then $c_0$ is an embedding eigenvalue of $\mathcal{R}_{\alpha,\beta}$, which is a contradiction.

Define $\Phi_+(\cdot,c):=\lim\limits_{\varepsilon\to0^+}\Phi(\cdot,c+i\varepsilon)$ in $C([y_1,y_2])$.
By (\ref{uniform bound-Phi}), $\|\Phi(\cdot,c+i\varepsilon)\|_{H^1(y_1,y_2)}\leq C\|{\omega\over p}\|_{H^1(y_1,y_2)}$
for all $0<\varepsilon<\varepsilon_0$. Then up to  a subsequence, $\Phi(\cdot,c+i\varepsilon)\rightharpoonup\Phi_+(\cdot,c)$ in $H^1(y_1,y_2)$ and (\ref{bound-Phi-pm}) holds.
\end{proof}

\subsection{Limiting absorption principle for monotone shear flows}

In this subsection, we establish the limiting absorption principle for monotone flows considered in Section 2, i.e. Lemma \ref{lem5.6}.
The main difference is that we present a uniform $H^1$ bound of $\Phi$ in the wave number $\al$.

\begin{lemma}\label{lem5.4}
If $f\in H_0^1(0,1)$, then for $c^i>0$,
\beno
\left|\int_0^1\frac{f(y)}{u(y)-c}dy\right|\leq C\alpha^{-\frac{1}{2}}\big(\|\partial_yf\|_{L^2}+\alpha\|f\|_{L^2}\big),
\eeno
where the constant $C$ only depends on $c_0$.
\end{lemma}

\begin{proof}
Due to $c^i>0$, we have
\begin{align*}
&\int_0^1\frac{f(y)}{u(y)-c}dy=i\int_0^1{f(y)}\int_0^{+\infty}e^{-it(u(y)-c)}dtdy
=i\int_0^{+\infty}e^{itc}\int_0^1{f(y)}e^{-itu(y)}dydt.
\end{align*}
Let $g(t)=\int_0^1{f(y)}e^{-itu(y)}dy$. Then we get
\begin{align*}
&\left|\int_0^1\frac{f(y)}{u(y)-c}dy\right|\leq\int_0^{+\infty}|e^{itc}|\left|\int_0^1{f(y)}e^{-itu(y)}dy\right|dt
=\int_0^{+\infty}e^{-tc^i}\left|g(t)\right|dt\leq\|g\|_{L^1(\mathbb{R})}.
\end{align*}
Due to $f\in H_0^1(0,1),$ we have
\begin{align*}
&g(t)=\int_0^1{f(y)}e^{-itu(y)}dy=\int_{u(0)}^{u(1)}e^{-i tz}(f/u')\circ u^{-1}(z)dz,\\&itg(t)=\int_{u(0)}^{u(1)}e^{-i tz}((f/u')'/u')\circ u^{-1}(z)dz,
\end{align*}
from which and Plancherel's formula, we infer that
\begin{align*}
\|g\|_{L^2(\mathbb{R})}^2&={2\pi}\|(f/u')\circ u^{-1}\|_{L^2(u(0),u(1))}^2={2\pi}\||f|^2/u'\|_{L^1(0,1)}\leq (2\pi/c_0)\|f\|_{L^2}^2,\\ \|tg(t)\|_{L^2(\mathbb{R})}^2&={2\pi}\||(f/u')'|^2/u'\|_{L^1(0,1)}\leq C\big(\|\partial_yf\|_{L^2}+\|f\|_{L^2}\big)^2.
\end{align*}
Thus, we obtain
\begin{align*}
\|(\alpha^2+t^2)^{\frac{1}{2}}g(t)\|_{L^2(\mathbb{R})}^2=&\alpha^2\|g\|_{L^2(\mathbb{R})}^2+\|tg(t)\|_{L^2(\mathbb{R})}^2
\\ \leq& C\alpha^2\|f\|_{L^2}^2+ C\big(\|\partial_yf\|_{L^2}+\|f\|_{L^2}\big)^2\leq C\big(\|\partial_yf\|_{L^2}+\alpha\|f\|_{L^2}\big)^2,
\end{align*}
and
\begin{align*}
\left|\int_0^1\frac{f(y)}{u(y)-c}dy\right|&\leq\|g\|_{L^1(\mathbb{R})}\leq\|(\alpha^2+t^2)^{\frac{1}{2}}g(t)\|_{L^2(\mathbb{R})}
\|(\alpha^2+t^2)^{-\frac{1}{2}}\|_{L^2(\mathbb{R})}\\ &\leq C\big(\|\partial_yf\|_{L^2}+\alpha\|f\|_{L^2}\big)\alpha^{-\frac{1}{2}}.
\end{align*}
This completes the proof.
\end{proof}

\begin{lemma}\label{lem5.5}
Let $\alpha\ge c_0>0$, $c^i>0$. Then the unique solution $\Phi$ to the boundary value problem
\begin{align*}
(u-c)(\Phi''-\alpha^2\Phi)=\omega, \;\Phi(0)=\Phi(1)=0
\end{align*}
has the uniform bound
\begin{align*}
\|\partial_y\Phi\|_{L^2}+\alpha\|\Phi\|_{L^2}\leq C\alpha^{-1}\big(\|\partial_y\om\|_{L^2}+\alpha\|\om\|_{L^2}\big).
\end{align*}
Moreover, if $\om(0)=\om(1)=0 $, then we have
\begin{align*}
|\partial_y\Phi(0)|+|\partial_y\Phi(1)|\leq C\alpha^{-\frac{1}{2}}\big(\|\partial_y\om\|_{L^2}+\alpha\|\om\|_{L^2}\big).
\end{align*}
\end{lemma}

\begin{proof}By Gagliardo-Nirenberg inequality, we get
\begin{align*}
\|\om\|_{L^{\infty}}\leq C\|\om\|_{L^{2}}^{\frac{1}{2}}\|\om\|_{H^{1}}^{\frac{1}{2}}\leq C\alpha^{-\frac{1}{2}}\big(\|\om\|_{H^1}+\alpha\|\om\|_{L^2}\big)\leq C\alpha^{-\frac{1}{2}}\big(\|\partial_y\om\|_{L^2}+\alpha\|\om\|_{L^2}\big),
\end{align*}
and similarly $\|\Phi\|_{L^{\infty}}\leq  C\alpha^{-\frac{1}{2}}\big(\|\partial_y\Phi\|_{L^2}+\alpha\|\Phi\|_{L^2}\big).$
Since
\begin{align*}
\|\partial_y\Phi\|_{L^2}^2+\alpha^2\|\Phi\|_{L^2}^2=-\langle\Phi''-\alpha^2\Phi,\Phi\rangle= -\langle\om/(u-c),\Phi\rangle =-\int_0^1\frac{\om(y)\overline{\Phi(y)}}{u(y)-c}dy,
\end{align*}
and $\om\overline{\Phi}(0)=\om\overline{\Phi}(1)=0,$ we get by Lemma \ref{lem5.4} that
\begin{align*}
\|\partial_y\Phi\|_{L^2}^2+\alpha^2\|\Phi\|_{L^2}^2\leq& C\alpha^{-\frac{1}{2}}\big(\|\partial_y(\om\overline{\Phi})\|_{L^2}+\alpha\|\om\overline{\Phi}\|_{L^2}\big)\\ \leq& C\alpha^{-\frac{1}{2}}\big(\|\partial_y\om\|_{L^2}\|{\Phi}\|_{L^{\infty}}+\|\om\|_{L^{\infty}}\|\partial_y{\Phi}\|_{L^2}
+\alpha\|\om\|_{L^{\infty}}\|{\Phi}\|_{L^2}\big)\\
\leq& C\alpha^{-\frac{1}{2}}\big(\alpha^{-\frac{1}{2}}\|\partial_y\om\|_{L^2}+\|\om\|_{L^{\infty}})
(\alpha^{\frac{1}{2}}\|{\Phi}\|_{L^{\infty}}+\|\partial_y{\Phi}\|_{L^2}+{\alpha}\|{\Phi}\|_{L^2}\big)\\ \leq& C\alpha^{-\frac{1}{2}}\big(\alpha^{-\frac{1}{2}}\|\partial_y\om\|_{L^2}+\alpha^{\frac{1}{2}}\|\om\|_{L^{2}})
(\|\partial_y{\Phi}\|_{L^2}+{\alpha}\|{\Phi}\|_{L^2}\big),
\end{align*}
which implies the first inequality.

Recall that $\gamma_1$ and $\gamma_2$ are defined in (\ref{gamma12-def}). Then we have $|\gamma_j|\leq 1,\ |\gamma_j'|\leq C\alpha $ and
\begin{align*}
&|\partial_y\Phi(j)|=|\langle \Phi''-\alpha^2\Phi,\gamma_j\rangle|=|\langle \om/(u-c),\gamma_j\rangle|=\left|\int_0^1\frac{\om(y){\gamma_j(y)}}{u(y)-c}dy\right|,\ j=0,1.
\end{align*}
If $\om(0)=\om(1)=0$, then $\om\gamma_j\in H_0^1(0,1) ,$ and by Lemma \ref{lem5.4}, we have
\begin{align*}
&|\partial_y\Phi(j)|\leq C\alpha^{-\frac{1}{2}}\big(\|\partial_y(\om\gamma_j)\|_{L^2}+\alpha\|\om\gamma_j\|_{L^2}\big)\leq C\alpha^{-\frac{1}{2}}\big(\|\partial_y\om\|_{L^2}+\alpha\|\om\|_{L^2}\big),\ j=0,1,
\end{align*}
which gives the second inequality.\end{proof}

Now we are in a position to prove Lemma \ref{lem5.6}.

\begin{proof}
Suppose that the first inequality is not true. Then there exist $\Phi_n\in H_0^1(0,1)$, ${\omega_n}\in H^1(0,1)$ and  $c_n\in\mathbb{C},\ \alpha_n\in \Lambda$ with $c_n^i>0$ such that $\|\partial_y\Phi_n\|_{L^2}+\alpha_n\|\Phi_n\|_{L^2}=\alpha_n^{-1}$, $\|\partial_y\om_n\|_{L^2}+\alpha_n\|\om_n\|_{L^2}=\delta_n\to0$, $c_n^i\to 0$, $c_n\to c_0\in\mathbb{R}\cup\{\pm\infty\}$ and
\begin{align*}
(u-c_n)(\Phi_n''-\alpha_n^2\Phi_n)-(u''-\beta)\Phi_n=\omega_n.
\end{align*}

By Lemma \ref{lem5.5}, we have
\begin{align*}
1=&\alpha_n(\|\partial_y\Phi_n\|_{L^2}+\alpha_n\|\Phi_n\|_{L^2})\leq C\big(\|\partial_y(\om_n+(u''-\beta)\Phi_n)\|_{L^2}+\alpha_n\|\om_n+(u''-\beta)\Phi_n\|_{L^{2}}\big)\\
\leq& C\big(\|\partial_y\om_n\|_{L^2}+\|\partial_y\Phi_n\|_{L^2}+\alpha_n\|\om_n\|_{L^{2}}+\alpha_n\|\Phi_n\|_{L^{2}}\big)\leq C\big(\alpha_n^{-1}+\delta_n\big).
\end{align*}
Since $\delta_n\to0$, this implies that $\alpha_n$ is uniformly bounded.
Up to a subsequence, we may assume that $\alpha_n$ is constant($\alpha_n=\alpha>0$) and that there exists $\Phi_0\in H_0^1(0,1)$ so that $\Phi_n\rightharpoonup\Phi_0$ in $H^1(0,1)$.

If $c_n\to \pm\infty$, then $\|(u-c_n)^{-1}\|_{L^{\infty}}\to 0$ and
\begin{align*}
\|\partial_y\Phi_n\|_{L^2}^2+\alpha_n^2\|\Phi_n\|_{L^2}^2=&-\langle\Phi_n''-\alpha_n^2\Phi_n,\Phi_n\rangle
=-\langle(u-c_n)^{-1}(\om_n+(u''-\beta)\Phi_n),\Phi_n\rangle\\ \leq &\|(u-c_n)^{-1}\|_{L^{\infty}}(\|\om_n\|_{L^2}\|\Phi_n\|_{L^2}+C\|\Phi_n\|_{L^2}^2)\to0,
\end{align*}
which contradicts with $\|\partial_y\Phi_n\|_{L^2}+\alpha_n\|\Phi_n\|_{L^2}=\alpha_n^{-1},\ \alpha_n=\alpha$, $n\geq1$.

If $c_n\to c_0\in\mathbb{R}\setminus[u(0),u(1)]$, then $\Phi_n\to\Phi_0$ in $H^1(0,1)$ and $\Phi_0$ satisfies \eqref{not-zero-point} for any $y\in [0,1]$. Thus, $\|\partial_y\Phi_0\|_{L^2}+\alpha\|\Phi_0\|_{L^2}={\alpha^{-1}}$ and $c_0$ is an  eigenvalue of $\mathcal{R}_{\alpha,\beta}$, which is a contradiction.

If $c_n\to c_0\in[u(0),u(1)],$ as in the proof of Theorem \ref{uniform-H1-bound}, we know that $\Phi_0$ satisfies \eqref{integral-formula-for-limit-function} for any $\varphi\in H_0^1(0,1)$ with $\text{supp } \varphi\in(0,1)$, that $\Phi_0\equiv0$ on $[0,1]$ (since $\mathcal{R}_{\alpha,\beta}$ has no embedding eigenvalues), and that $\Phi_n\to0$ in $H^1(0,1)$, which contradicts with $\|\partial_y\Phi_n\|_{L^2}+\alpha\|\Phi_n\|_{L^2}=\alpha^{-1}$.

In summary, this shows  the first inequality.\smallskip

If $\om(0)=\om(1)=0 $, then $\om+(u''-\beta)\Phi=0$ at $y=0,1$. Then from Lemma \ref{lem5.5} and the first inequality, we deduce that 
\begin{align*}
|\partial_y\Phi(j)|&\leq C\alpha^{-\frac{1}{2}}\big(\|\partial_y(\om+(u''-\beta)\Phi)\|_{L^2}+\alpha\|\om+(u''-\beta)\Phi\|_{L^2}\big)\\
&\leq C\alpha^{-\frac{1}{2}}\big(\|\partial_y\om\|_{L^2}+\alpha\|\om\|_{L^2}+\|\partial_y\Phi\|_{L^2}+\alpha\|\Phi\|_{L^2}\big)\\
&\leq C\alpha^{-\frac{1}{2}}\big(\|\partial_y\om\|_{L^2}+\alpha\|\om\|_{L^2}\big),\ j=0,1,
\end{align*}
which gives the second inequality.
\end{proof}

\section{Linear inviscid damping and vorticity depletion}

In this section, we prove the linear inviscid damping and vorticity depletion for a class of  shear flows  satisfying \textbf{(H1)}.
Let us first prove Theorem \ref{thm:non-monotone}.
\begin{proof}
Since $P_{\sigma_d(\mathcal{R}_{\alpha,\beta})}\widehat{\psi}(0,\alpha,\cdot)=0$, we have
\begin{align*}
\widehat{\psi}(t,\alpha,y)={1\over 2\pi i}\int_{\partial\Omega_{\varepsilon}}e^{-i\alpha tc}(c-\mathcal{R}_{\alpha,\beta})^{-1}\widehat{\psi}(0,\alpha,y)dc,
\end{align*}
where $\Omega_{\varepsilon}, 0<\varepsilon\leq \varepsilon_0,$ is defined in Theorem \ref{uniform-H1-bound}. Let $c\in\partial\Omega_{\varepsilon}$ and $\Phi(\alpha,y,c)={1\over i\alpha}(c-\mathcal{R}_{\alpha,\beta})^{-1}\widehat{\psi}(0,\alpha,y)$.
Then $\Phi$ satisfies (\ref{rayleigh-equation-Phi}) with $\omega={{\widehat{\omega}}_0(\alpha,\cdot)\over i\alpha}$ and
\begin{align*}
\widehat{\psi}(t,\alpha,y)={\alpha\over 2\pi }\int_{\partial\Omega_{\varepsilon}}e^{-i\alpha tc}\Phi(\alpha,y,c)dc.
\end{align*}
Let $\tilde \Phi=\Phi_--\Phi_+.$ It follows from Theorem \ref{uniform-H1-bound} that
\begin{align*}
\widehat{\psi}(t,\alpha,y)&=\lim\limits_{\varepsilon\to0^+}{\alpha\over 2\pi }\int_{\partial\Omega_{\varepsilon}}e^{-i\alpha tc}\Phi(\alpha,y,c)dc\\
&={\alpha\over2\pi }\int_{\text{Ran }(u)}e^{-i\alpha t c}(\Phi_-(\alpha,y,c)-\Phi_+(\alpha,y,c))dc\\
&={\alpha\over2\pi }\int_{\text{Ran }(u)}e^{-i\alpha t c}\tilde \Phi(\alpha,y,c) dc.
\end{align*}
Then by Plancherel's formula, $2\pi\|\widehat{\psi}(\cdot,\alpha,y)\|^2_{L^2(\mathbb{R})}= |\alpha|\|\tilde\Phi(\alpha,y,\cdot)\|^2_{L^2(\text{Ran } (u))}$, and thus
\begin{align*}
&\|\widehat{v}(\cdot,\alpha,\cdot)\|_{H_t^1(\mathbb{R};L^2_y)}^2
=\int_{\mathbb{R}}\big(\|\widehat{v}(t,\alpha,\cdot)\|_{L^2_y}^2+\|\partial_t\widehat{v}(t,\alpha,\cdot)\|_{L^2_y}^2\big)dt\\
&=\int_{y_1}^{y_2}\int_{\mathbb{R}}\big(\alpha^2|\widehat{\psi}(t,\alpha,y)|^2+|\partial_y\widehat{\psi}(t,\alpha,y)|^2+\alpha^2
|\partial_t\widehat{\psi}(t,\alpha,y)|^2+|\partial_t\partial_y\widehat{\psi}(t,\alpha,y)|^2\big)dtdy\\
&={|\alpha|\over2\pi}\int_{y_1}^{y_2}\int_{\text{Ran }(u)}(1+(\alpha c)^2)\big(\alpha^2|\tilde{\Phi}(\alpha,y,c)|^2+|\partial_y\tilde{\Phi}(\alpha,y,c)|^2\big)dcdy\\
&\leq C\int_{\text{Ran }(u)}\|\tilde \Phi(\alpha,\cdot,c)\|_{H_y^1}^2dc\leq C\|{\widehat{\omega}_0(\alpha,\cdot)\over p}\|_{H_y^1}^2,
\end{align*}
where we used (\ref{bound-Phi-pm}) in the last inequality. In particular, we have
\begin{align*}
\|\widehat{v}(t,\alpha,\cdot)\|_{L^2_y}\leq C\|\widehat{v}(\cdot,\alpha,\cdot)\|_{C_t([t-1,t+1];L^2_y)}\leq C\|\widehat{v}(\cdot,\alpha,\cdot)\|_{H^1_t([t-1,t+1];L^2_y)}\to0
\end{align*}
as $t\to\infty$.
\end{proof}

Next we show the vorticity depletion phenomena of shear flow under Coriolis effects. For this, we need the following lemma.

\begin{lemma}\label{estimation on phi''-alpha2phi}
Let $\alpha>0$, $\beta\in\mathbb{R}$ and $c\notin\mathbb{R}$.
Assume that $y_0\in (u')^{-1}\{0\}$, $(a,b)$ be an interval such that $y_0\in (a,b)\subset[y_1,y_2]$, $|u(y_0)-c|<\min\{1,\max\{|y_0-a|^2,|y_0-b|^2\}\}$,  $(\beta-u''(y))(\beta-u''(y_0))>0$, $\phi, \omega\in H^1(a,b)$, and $(u-c)(\phi''-\alpha^2\phi)-(u''-\beta)\phi=\omega$ on $[a,b]$. Then
$$
|(\phi''-\alpha^2\phi)(y_0)|\leq C |u(y_0)-c|^{-{3\over4}}\big(\|\phi\|_{H^1(a,b)}+\|\omega\|_{H^1(a,b)}\big),
$$
where $C$ depends on $\max\{|y_0-a|,|y_0-b|\}$, $\alpha$, $\beta$ and $u$.
\end{lemma}

\begin{proof}
If $\omega(y_0)=0$, then it follows from Lemma \ref{critical point estimate} that
\begin{align*}
|(\phi''-\alpha^2\phi)(y_0)|=\left|{(u''-\beta)\phi\over u-c}\right|(y_0)\leq C|u(y_0)-c|^{-{3\over4}}\big(\|\phi\|_{H^1(a,b)}+\|\omega\|_{H^1(a,b)}\big).
\end{align*}
If $\omega(y_0)\neq0$, let
\begin{align*}
\phi_*(y)&=\phi(y)+{\omega(y_0)\over u''(y_0)-\beta}\cosh \alpha(y-y_0),\;\\
\omega_*(y)&=\omega(y)-(u''(y)-\beta){\omega(y_0)\over u''(y_0)-\beta}\cosh \alpha(y-y_0).
\end{align*}
Then $\phi_*,\omega_*\in H^1(a,b)$, $\omega_*(y_0)=0$ and
\begin{align*}
(u-c)(\phi''_*-\alpha^2\phi_*)-(u''-\beta)\phi_*=\omega_*.
\end{align*}
This reduces $\phi_*$ and $\omega_*$ to the case of $\om(y_0)=0$, and hence we obtain
\begin{align*}
|(\phi''-\alpha^2\phi)(y_0)|=&|(\phi_*''-\alpha^2\phi_*)(y_0)|\\
\leq& C|u(y_0)-c|^{-{3\over4}}\big(\|\phi_*\|_{H^1(a,b)}+\|\omega_*\|_{H^1(a,b)}\big)\\
\leq &C|u(y_0)-c|^{-{3\over4}}\big(\|\phi\|_{H^1(a,b)}+\|\omega\|_{H^1(a,b)}+|\omega(y_0)|\big)\\
\leq &C|u(y_0)-c|^{-{3\over4}}\big(\|\phi\|_{H^1(a,b)}+\|\omega\|_{H^1(a,b)}\big).
\end{align*}
This completes the proof.\end{proof}

\begin{theorem}
Under the assumption of Theorem \ref{thm:non-monotone}, if $u'(y_0)=0$, then
\begin{align*}
\lim_{t\to\infty}\widehat{\omega}(t,\alpha,y_0)=0.
\end{align*}
\end{theorem}

\begin{proof}
If $\beta-u''(y_0)\neq0$ and $y_0\in (y_1,y_2)$,  with the help of Theorem \ref{uniform-H1-bound} and Lemma \ref{estimation on phi''-alpha2phi}, the proof is similar to that of Theorem 2.1 in \cite{WZZ2}. If $\beta-u''(y_0)=0$ or $y_0\in\{y_1,y_2\}$, then $(\beta-u''(y_0))\widehat{\psi}(y_0)=0$. Using (\ref{linearized Euler equation}) at $y=y_0$ and taking Fourier transform with respect to $x$, we get
\begin{align*}
\partial_{t}\widehat{\omega}+i \alpha u(y_0)\widehat{\omega}=0.
\end{align*}
Then $\widehat{\omega}(t,\alpha,y_0)={\widehat{\omega}(0,\alpha,y_0)} e^{-i\alpha u(y_0)t}$. As $\widehat{\omega}(0,\alpha,y_0)=\widehat{\omega}_0(\alpha,y_0)=0$, we have $\widehat{\omega}(t,\alpha,y_0)\equiv0$ for any $t\in\mathbb{R}$. 
\end{proof}

\section{Application to the Sinus flow}

In this section, we consider the linear inviscid damping of the flow with Sinus profile:
\begin{equation*}
u(y)={\frac{1+\cos(\pi y)}{2}},\ \ \ \ y\in\lbrack-1,1].
\end{equation*}
Clearly, $(u')^{-1}\{0\}=\{0,\pm1\}$ and thus $\{\beta\in\mathbb{R}|u'(y)=0,{\beta\over u''(y)}<{9\over8}\}=(-{9\over16}\pi^2,{9\over16}\pi^2)$.
Now we want to examine the region of $(\alpha,\beta)$ in $(0,\infty)\times (-{9\over16}\pi^2,{9\over16}\pi^2)$ such that $\mathcal{R}_{\alpha,\beta}$ has no embedding eigenvalues.

\begin{center}
 \begin{tikzpicture}[scale=0.8]
 \draw [->](-10, 0)--(10, 0)node[right]{$\beta$};
 \draw [->](0,0)--(0,7) node[above]{$\alpha$};
 \draw (-7.8, 5).. controls (-7,5) and (2,5)..(7.8,5);
 \draw (0, 0).. controls (3, 4) and (8, 4)..(9,6.2);
 \draw (7.8, 0).. controls (8.3,1) and (9.05, 2.5)..(9,3);
  \draw (-7.8, 0).. controls (-8.3,1) and (-9.05, 2.5)..(-9,3);
   \path (-9, 0)  edge [-,dotted](-9, 7) [line width=0.7pt];
   \path (9, 0)  edge [-,dotted](9, 7) [line width=0.7pt];
   \path (0, 3)  edge [-,dotted](9, 3) [line width=0.7pt];
   \path (0, 6.2)  edge [-,dotted](9, 6.2) [line width=0.7pt];
   \path (7.8, 0)  edge [-,dotted](7.8, 5) [line width=0.7pt];
   \path (-7.8, 0)  edge [-,dotted](-7.8, 5) [line width=0.7pt];
    \path (-9, 3)  edge [-,dotted](0, 3) [line width=0.7pt];
    \node (a) at (-9.2,-0.5) {\small$-{9\over 16}\pi^2$};
    \node (a) at (-7.7,-0.5) {\small$-{1\over 2}\pi^2$};
    \node (a) at (0,-0.5) {\small$0$};
    \node (a) at (7.7,-0.5) {\small${1\over 2}\pi^2$};
    \node (a) at (9.2,-0.5) {\small${9\over 16}\pi^2$};
    \node (a) at (-0.5,3) {\small${\sqrt{7}\over4}\pi$};
    \node (a) at (-0.5,4.59) {\small${\sqrt{3}\over2}\pi$};
    \node (a) at (-0.59,6.2) {\small${\sqrt{15}\over4}\pi$};
     \node (a) at (-8.17,1.6) {$\gamma_4$};
     \node (a) at (-2,5.2) {$\gamma_1$};
     \node (a) at (5,3.9) {$\gamma_2$};
     \node (a) at (8.13,1.6) {$\gamma_3$};
 \end{tikzpicture}
\end{center}\vspace{-0.2cm}
 \begin{center}\vspace{-0.2cm}
   {\small {\bf Figure 1.} }
  \end{center}\vspace{-0.2cm}
Let
\begin{align*}
\gamma_1&=\{(\alpha,\beta)|\alpha={\sqrt{3}\over2}\pi,\beta\in(-{1\over2}\pi^2,{1\over2}\pi^2)\},\\
\gamma_2&=\{(\alpha,\beta)|\alpha=\pi\sqrt{1-r^2},\beta=\pi^2(-r^2+{1\over2}r+{1\over2}),r\in({1\over4},1)\},\\
\gamma_3&=\{(\alpha,\beta)|\alpha=\pi\sqrt{-r^2-r+{3\over4}},\beta=\pi^2(-r^2+{1\over2}r+{1\over2}),r\in({1\over4},{1\over2})\},\\
\gamma_4&=\{(\alpha,\beta)|\alpha=\pi\sqrt{-r^2-r+{3\over4}},\beta=\pi^2(r^2-{1\over2}r-{1\over2}),r\in({1\over4},{1\over2})\}.
\end{align*}

By detailed analysis on the spectrum of $\mathcal{R}_{\alpha,\beta}$ and applying Theorem \ref{thm:non-monotone}, our main result for the Sinus profile is stated as follows.

\begin{theorem}Consider the Rayleigh-Kuo operator $\mathcal{R}_{\alpha,\beta}$ with $(\alpha,\beta)\in(0,\infty)\times (-{9\over16}\pi^2,{9\over16}\pi^2).$ Assume that $\widehat{\omega}_0(\alpha,\pm1)=0$ and $P_{\sigma_d(\mathcal{R}_{\alpha,\beta})}\widehat{\psi}(0,\alpha,\cdot)=0$. Then we have

\begin{itemize}

\item[(1)] $\mathcal{R}_{\alpha,\beta}$ has exactly an embedding eigenvalue $c={1\over2}-{\beta\over \pi^2}$  if and only if $(\alpha,\beta)\in\gamma_1$;
    $\mathcal{R}_{\alpha,\beta}$ has exactly  an embedding eigenvalue $c=0$  if and only if $(\alpha,\beta)\in\gamma_2\cup\gamma_3$;
    $\mathcal{R}_{\alpha,\beta}$ has exactly  an embedding eigenvalue $c=1$  if and only if $(\alpha,\beta)\in\gamma_4\cup\{({\sqrt{3}\over2}\pi,-{1\over2}\pi^2)\}$; and
 $\mathcal{R}_{\alpha,\beta}$  has no embedding eigenvalues if and only if
$$(\alpha,\beta)\in \Gamma=\Big((0,\infty)\times (-{9\over16}\pi^2,{9\over16}\pi^2)\Big)\setminus\left(\gamma_1\cup\gamma_2\cup\gamma_3\cup\gamma_4\right).$$

\item[(2)] If $(\alpha,\beta)\in \Gamma$, $\beta\neq\pm{\pi^2\over2}$ and   $\widehat{\omega}_0(\alpha,\cdot)\in H_y^1(-1,1)$, trhen
\begin{align*}
\|\widehat{v}(\cdot, \alpha,\cdot)\|_{H^1_tL^2_y}\leq C \|{\widehat{\omega}_0(\alpha,\cdot)}\|_{H_y^1}\;\;
\text{and}\;\;\lim\limits_{t\to\infty}\|\widehat{v}(t,\alpha,\cdot)\|_{L^2_y}=0.
\end{align*}

\item[(3)]  If $(\alpha,\beta)\in \Gamma$, $\beta={\pi^2\over2}$ and $h_0(\alpha,y)={\widehat{\omega}_0(\alpha,y)\over y^2-1}\in H_y^1(-1,1)$, then
\begin{align*}
\|\widehat{v}(\cdot, \alpha,\cdot)\|_{H^1_tL^2_y}\leq C \|h_0(\alpha,\cdot)\|_{H_y^1}\;\;
\text{and}\;\;\lim\limits_{t\to\infty}\|\widehat{v}(t,\alpha,\cdot)\|_{L^2_y}=0.
\end{align*}

\item[(4)]  If  $(\alpha,\beta)\in \Gamma$, $\beta=-{\pi^2\over2}$ and $f_0(\alpha,y)={\widehat{\omega}_0(\alpha,y)\over y}\in H_y^1(-1,1)$, then
\begin{align*}
\|\widehat{v}(\cdot, \alpha,\cdot)\|_{H^1_tL^2_y}\leq C \|f_0(\alpha,\cdot)\|_{H_y^1}\;\;
\text{and}\;\;\lim\limits_{t\to\infty}\|\widehat{v}(t,\alpha,\cdot)\|_{L^2_y}=0.
\end{align*}

\end{itemize}

\end{theorem}

\begin{proof}
The proof of (1) is finished by Propositions \ref{1over2-betaoverpi2}--\ref{cin(0,1)}, while (2)--(4) are direct consequences of Theorem \ref{thm:non-monotone}.
\end{proof}

\begin{proposition}\label{1over2-betaoverpi2}
$\mathcal{R}_{\alpha,\beta}$ has  an embedding eigenvalue $c={1\over2}-{\beta\over \pi^2}$  if and only if $(\alpha,\beta)\in\gamma_1$.
\end{proposition}
\begin{proof}
It follows from (7.4) in \cite{Kuo1974} or (4.5) in \cite{LYZ} that when $(\alpha,\beta)\in\gamma_1$,  $\mathcal{R}_{\alpha,\beta}$ has an embedding eigenvalue $c={1\over2}-{\beta\over \pi^2}$ with the eigenfunction $\phi(y)=\cos({\pi y\over 2})$.
Conversely, we rewrite the homogeneous Rayleigh-Kuo equation with $c={1\over2}-{\beta\over \pi^2}$ to be a Sturm-Liouville problem
\begin{align}\label{Rayleigh-Kuo-to-SL-problem}
-\phi''+{u''-\beta\over u-c}\phi=-\phi''-\pi^2\phi=\lambda\phi,\;\; \phi(\pm1)=0,
	\end{align}
where $\lambda=-\alpha^2$. Then it is easy to see that the second eigenvalue of (\ref{Rayleigh-Kuo-to-SL-problem}) is $0$. Therefore, ${1\over2}-{\beta\over \pi^2}$ is not an embedding eigenvalue of $\mathcal{R}_{\alpha,\beta}$ when $(\alpha,\beta)\notin\gamma_1$.
\end{proof}

\begin{proposition}\label{embedding eigenvaluec=0}
$\mathcal{R}_{\alpha,\beta}$ has   an embedding eigenvalue $c=0$  if and only if  $(\alpha,\beta)\in\gamma_2\cup\gamma_3$.
\end{proposition}

\begin{proof}
 We get by (7.5) in \cite{Kuo1974} that $\mathcal{R}_{\alpha,\beta}$ has an embedding eigenvalue $c=0$ with the eigenfunction $\phi(y)=\cos^{2r}({\pi y\over 2})$ for $\alpha=\pi\sqrt{1-r^2}$ and $\beta=\pi^2(-r^2+{1\over2}r+{1\over2})$ with $r\in[{1\over2},1)$. By Definition \ref{embedding eigenvalue}, we know that this also holds true when $r\in({1\over4},{1\over2})$. Therefore, $\mathcal{R}_{\alpha,\beta}$ has an embedding eigenvalue $c=0$ when $(\alpha,\beta)\in\gamma_2$. By Lemma 4.3 in \cite{LYZ},  the second eigenvalue of
\begin{align}\label{SL-c=0}
-\phi''+{u''-\pi^2(-r^2+{1\over2}r+{1\over2})\over u}\phi=\lambda\phi,\;\; \phi(\pm1)=0
	\end{align}
is $\pi^2(r^2+r-{3\over4})$ with the eigenfunction $\phi_2=\cos^{2r}({\pi y\over2})\sin({\pi y\over2})$.
Since $\pi^2(r^2+r-{3\over4})<0$ when $r\in({1\over4},{1\over2})$ and $\pi^2(r^2+r-{3\over4})\geq0$ when $r\in[{1\over2},1)$, 
$\mathcal{R}_{\alpha,\beta}$ has an embedding eigenvalue $c=0$ when $(\alpha,\beta)\in\gamma_3$.

Conversely, we compute by induction that $\pi^2(r^2+2r)$ is  an eigenvalue of  (\ref{SL-c=0}) with the eigenfunction
 $$\phi_3(y)=\cos^{2r}({\pi y\over2})((2r+1)\sin^2({\pi y\over2})-{1\over2}).$$
  Since $\phi_3$ has two zeros in $(-1,1)$, we have by Theorem 10.12.1 in \cite{Zettl2005} that $\pi^2(r^2+2r)$ is the third eigenvalue of  (\ref{SL-c=0}). Noting that $\pi^2(r^2+2r)>0$ when $r\in({1\over4},1)$, we have that $0$ is not an embedding eigenvalue of $\mathcal{R}_{\alpha,\beta}$ when $(\alpha,\beta)\in (0,+\infty)\times(0,{9\over16}\pi^2)\setminus (\gamma_2\cup \gamma_3)$.
It follows from (4.8) in \cite{Tung1981} that $c=0$ is not an embedding eigenvalue of $\mathcal{R}_{\alpha,\beta}$ when $(\alpha,\beta)\in(0,+\infty)\times(-{9\over16}\pi^2,0]$.
\end{proof}

\begin{proposition}\label{embedding eigenvalue c=1}
$\mathcal{R}_{\alpha,\beta}$ has  an embedding eigenvalue $c=1$  if and only if $(\alpha,\beta)\in\gamma_4\cup\{({\sqrt{3}\over2}\pi,-{1\over2}\pi^2)\}$.
\end{proposition}

\begin{proof}
  Similar to the proof of Proposition \ref{1over2-betaoverpi2},  $c=1$ is an embedding eigenvalue of $\mathcal{R}_{{\sqrt{3}\over2}\pi,-{1\over2}\pi^2}$. Then we show that $c=1$ is  an embedding eigenvalue of $\mathcal{R}_{\alpha,\beta}$  when $(\alpha,\beta)\in\gamma_4$. Consider the boundary value problem
 \begin{align}\label{boundary value problem-for-examining}
 -\phi''+{u''+\pi^2(-r^2+{1\over2}r+{1\over2})\over u-1}\phi=\lambda\phi, \;\text{on}\; (0,1), \;\phi(0)=\phi(1)=0,
 \end{align}
 where $r\in({1\over4},{1\over2})$.   $\tilde\phi_1(y)=\sin^{2r}({\pi y\over2})\cos({\pi y\over2})\in H^1(0,1)$ with $\lambda=\pi^2(r^2+r-{3\over4})$ is a solution of (\ref{boundary value problem-for-examining}).
Since $\tilde\phi_1$ has no zeros in $(0,1)$ for any $r\in({1\over4},{1\over2})$, $\pi^2(r^2+r-{3\over4})$ is the first eigenvalue of (\ref{boundary value problem-for-examining}).
 By noting that $\pi^2(r^2+r-{3\over4})<0$ and $\tilde\phi_1\in H^1(0,1)$ for any $r\in({1\over4},{1\over2})$, we know that
 $c=1$ is an embedding eigenvalue of $\mathcal{R}_{\alpha,\beta}$ when $(\alpha,\beta)\in\gamma_4$.

 Conversely,
 direct computation implies that $(r^2+3r+{5\over4})\pi^2$ is also an eigenvalue of (\ref{boundary value problem-for-examining}) with the eigenfunction
 $$\tilde \phi_2=\sin^{2r}({\pi y\over2})(-{4\over3}(1+r)\cos^{3}({\pi y\over2})+\cos({\pi y\over2})).$$
 Since $\tilde \phi_2$ has exactly one zero in $(0,1)$ for any $r\in({1\over4},{1\over2})$, $(r^2+3r+{5\over4})\pi^2$ is the second eigenvalue of (\ref{boundary value problem-for-examining}).
Noting that $(r^2+3r+{5\over4})\pi^2>0$ when $r\in({1\over4},{1\over2})$,  $1$ is not an embedding eigenvalue of $\mathcal{R}_{\alpha,\beta}$ when $(\alpha,\beta)\in (0,+\infty)\times(-{9\over16}\pi^2,-{\pi^2\over2})\setminus \gamma_4$.
 Similar to the proof of Proposition \ref{1over2-betaoverpi2}, $c=1$ is  not an embedding eigenvalue of $\mathcal{R}_{\alpha,-{1\over2}\pi^2}$ with $\alpha\in(0,{\sqrt{3}\over2}\pi)\cup({\sqrt{3}\over2}\pi,+\infty)$.
By Lemma 4.2 in \cite{LYZ}, $c=1$ is not an embedding eigenvalue of $\mathcal{R}_{\alpha,\beta}$ when $(\alpha,\beta)\in(0,+\infty)\times(-{1\over2}\pi^2,0)$.
 It follows from (4.8) in \cite{Tung1981} that $c=1$ is not an embedding eigenvalue of $\mathcal{R}_{\alpha,\beta}$ when $(\alpha,\beta)\in(0,+\infty)\times[0,{9\over16}\pi^2)$. 
\end{proof}

Next, we exclude other embedding eigenvalues of $ \mathcal{R}_{\alpha,\beta}$.
 \begin{proposition}\label{cin(0,1)}
  For any $c\in(0,1)$ and $c\neq{1\over2}-{\beta\over\pi^2}$, it is not an embedding eigenvalue of $\mathcal{R}_{\alpha,\beta}$ when $(\alpha,\beta)\in (0,\infty)\times (-{9\over16}\pi^2,{9\over16}\pi^2)$.
  \end{proposition}
\begin{proof}
  Denote the two zeros of $u-c$ to be $z_1$ and $z_2$.
Suppose that $c$ is an embedding eigenvalue. Thanks to Definition \ref{embedding eigenvalue}, there exists $\phi\in H_0^1(-1,1)$ so that
$$\int_{-1}^{1}(|\phi'|^2+\alpha^2|\phi|^2)dy+p.v.\int_{-1}^{1}{(u''-\beta)|\phi|^2\over u-c}dy+i\pi\sum\limits_{y\in\{z_1,z_2\}}{(u''-\beta)|\phi|^2(y)\over |u'(y)|}=0.$$
Noting that $\beta-u''=\pi^2(u-{1\over2}+{\beta\over\pi^2})$ and $c\neq{1\over2}-{\beta\over\pi^2}$, we have $u''(z_j)-\beta\neq0$ and thus $\phi(z_j)=0$ for $j=1,2$. Moreover,
\begin{align*}
-\phi''+\alpha^2\phi+{u''-\beta\over u-c}\phi=0\;\;\text{on}\;\;(-1,1)\setminus\{z_1,z_2\}.
	\end{align*}

Let $\beta\in(-{9\over16}\pi^2,0]$. Then
\begin{align*}
\int_{z_1}^{ z_2}\left(|\phi'|^2+\alpha^2|\phi|^2+{u''-\beta\over u-c}|\phi|^2\right) dy=0.
\end{align*}
Thus, we get by integration by parts that 
\beno
\int_{z_1}^{ z_2}\Big|\phi'-u'\frac{\phi}{u-c}\Big|^2dy+\int_{z_1}^{ z_2}\left(\alpha^2-{\beta\over u-c}\right)|\phi|^2dy=0,
\eeno
which
implies $\phi\equiv0$ on $[z_1,z_2]$.
By Sobolev embedding $H^1(J)\hookrightarrow C^{0,{1\over2}}(J)$ and the fact that $c$ is an embedding eigenvalue of $\mathcal{R}_{\alpha,\beta}$, we have
\begin{align*}
\left|\int_{-1}^{1}{\phi}'\varphi'dy\right|=&\left|\int_{-1}^{1}\alpha^2\phi\varphi +\frac{(u''-\beta)\phi\varphi}
{u-c}dy\right|\\
\leq &C\left(\|\phi\|_{L^p}+\left(\int_{z_j-\varepsilon}^{z_j+\varepsilon}\left|{\phi\over u-c}\right|^pdy\right)^{1\over p}\right)\|\varphi\|_{L^{p'}}\leq C\|\varphi\|_{L^{p'}},
\end{align*}
for every $ \varphi\in H^1(-1,1)$ with $\text{supp}\ \varphi\subset[z_j-\varepsilon,z_j+\varepsilon],$ where $1<p<2,\ 1/p+1/p'=1$,  $j=1,2$, $\varepsilon>0$ is sufficiently small and $J$ is a compact interval. Thus, $ \phi\in W^{2,p}(z_j-\varepsilon,z_j+\varepsilon),$ and by Sobolev embedding $W^{2,p}(J)\hookrightarrow C^1(J)$, we have
$\phi\in C^1([-1,1])$.
Then by Lemma 2.2 in \cite{LYZ}, we have
 $\phi\equiv0$ on $[-1,1]$.

Let $\beta\in(0, {9\over16}\pi^2)$.  With a similar argument to $\beta\in(-{9\over16}\pi^2,0]$, we can first show that $\phi\equiv0$ on $[{-1},z_1]$ and $[z_2,1]$, then show that $\phi\equiv0$ on $[-1,1]$.

Therefore, $c$ is not an embedding eigenvalue of $\mathcal{R}_{\alpha,\beta}$ when $(\alpha,\beta)\in (0,\infty)\times (-{9\over16}\pi^2,{9\over16}\pi^2)$, and this completes the proof.
\end{proof}

\section*{Acknowledgement}
H. Zhu  would like to thank School of
Mathematical Science at Peking University, where part of this work was done when he was a visitor. Z. Zhang is partially supported by NSF of China under Grant 11425103.

\end{CJK*}

\end{document}